\documentclass[11pt,reqno]{amsart}
\usepackage{amsfonts}
\usepackage{amssymb,amsmath}

\numberwithin{equation}{section}

\newtheorem{theorem}{Theorem}[section]
\newtheorem{Lemma}{Lemma}[section]
\newtheorem{Prop}{Proposition}[section]

\newcommand{\beq}{\begin{equation}}
\newcommand{\eeq}{\end{equation}}
\newcommand{\bal}{\begin{aligned}}
\newcommand{\eal}{\end{aligned}}
\newcommand{\bna}{\begin{eqnarray}}
\newcommand{\ena}{\end{eqnarray}}
\newcommand{\bea}{\begin{eqnarray*}}
\newcommand{\eea}{\end{eqnarray*}}

\newcommand{\bb}{{\bf b}}

\newcommand{\D}{{\mathcal{D}}}
\newcommand{\A}{{\mathcal{A}}}

\newcommand{\p}{{\partial}}
\newcommand{\pO}{{\partial\Omega}}

\newcommand{\dd}{\mathrm{d}}

\def \D{{\mathrm{D}}}
\def \IBbb{{\mathbb I}}
\def \PBbb{{\mathbb P}}

\newcommand{\RD}{\mathbb{R}^{\D}}

\newcommand{\SD}{\mathbb{S}^{\D-1}}

\newcommand{\BP}{\iint\limits_{\Sigma_{+}}}

\newcommand{\sps}{\sqrt{\epsilon}}
\newcommand{\gps}{\tilde{g}_\epsilon}

\newcommand{\eps}{\varepsilon}

\newcommand{\ale}{\alpha_\epsilon}
\newcommand{\gpg}{\gamma_+g_\epsilon}
\newcommand{\gpgb}{\langle\gamma_+g_\epsilon\rangle_{\partial\Omega}}
\newcommand{\gamn}{\gamma_{\!-}}
\newcommand{\gamp}{\gamma_{\!+}}
\newcommand{\Sigpo}{\iint\limits_{\Sigma_{+}}}
\newcommand{\Signe}{\iint\limits_{\Sigma_{-}}}

\newcommand{\dnu}{\mathrm{d}\tilde{\nu}_\eps}
\newcommand{\Fe}{{F_\epsilon}}
\newcommand{\Fein}{{F_\epsilon^{\mathrm{in}}}}
\newcommand{\Ge}{{G_\epsilon}}
\newcommand{\gep}{g_\epsilon}
\newcommand{\Gein}{{G_\epsilon^{\mathrm{in}}}}
\newcommand{\gammag}{1+\eps^2\gamp g^2_\eps}
\newcommand{\gammagh}{1+\eps^2\gamp \hat{g}^2_\eps}

\newcommand{\Divv}{v\!\cdot\!\nabla_{\!x}}
\newcommand{\di}{\nabla_{\!x}\!\cdot\!}
\newcommand{\grad}{\nabla_{\!x}}

\def \del{{\partial}}

\def \DIV{\nabla_{\!\!x}\!\cdot\! }
\def \LAP{\Delta_x}

\def \VGRAD{v\!\cdot\!\nabla_{\!\!x}}

\def \LL{{\mathcal L}}

\def \AHat{{\widehat{\mathrm{A}}}}
\def \BHat{{\widehat{\mathrm{B}}}}

\def \Int{\mathrm{int}}
\def \b{\mathrm{b}}
\def \bb{\mathrm{bb}}
\def \u{\mathrm{u}}

\def \cvd{\tfrac{|v|^2}{2}-\tfrac{\D}{2}}
\def \<{\langle}
\def \>{\rangle}

\def \PP{{\mathcal P}}
\def \AA{{\mathcal A}}

\begin{document}

\title[Navier-Stokes-Fourier Limit]{Boundary layers and incompressible Navier-Stokes-Fourier limit of the Boltzmann Equation in Bounded Domain
(I) }
\author[N. Jiang]{Ning Jiang}
\address{School of Mathematics and Statistics, Wuhan university 430072, Wuhan, P.R. China}
\email{cicm10.math@whu.edu.cn}
\author[N. Masmoudi]{Nader Masmoudi}
\address{Courant Institute of Mathematical Sciences, 251 Mercer Street, New York, NY 10012}
\email{masmoudi@cims.nyu.edu}

\begin{abstract}

We establish the  incompressible Navier-Stokes-Fourier limit for solutions to the Boltzmann equation with a general cut-off collision
kernel in  a bounded domain. Appropriately scaled families of DiPerna-Lions-(Mischler) renormalized solutions with Maxwell
reflection boundary conditions are shown to have fluctuations that
converge as the Knudsen number goes to zero. Every limit point is a weak solution to the
Navier-Stokes-Fourier system with different types of  boundary
conditions depending on the ratio between  the accommodation
coefficient and the Knudsen number. The main new result of the paper
is that this convergence is strong in the case of Dirichlet boundary
condition.  Indeed, we prove that the acoustic waves are damped
immediately, namely they are damped in a boundary layer in time.
This damping is due to the presence of viscous and kinetic boundary
layers in space. As a consequence, we also justify the first correction to the infinitesimal Maxwellian that one obtains
from the Chapman-Enskog expansion with Navier-Stokes scaling.

This extends the work of Golse and Saint-Raymond \cite{Go-Sai04,
Go-Sai05} and Levermore and Masmoudi \cite{LM} to the case of a
bounded domain. The case of a bounded domain was considered
by Masmoudi and Saint-Raymond \cite{M-S}  for linear
Stokes-Fourier limit and Saint-Raymond \cite{SRM} for Navier-Stokes limit for hard potential kernels. Both \cite{M-S} and \cite{SRM} didn't  study the damping  of the acoustic waves. This paper extends the result of \cite{M-S} and \cite{SRM} to the nonlinear case and includes soft potential kernels. More importantly, for the Dirichlet boundary condition, this work strengthens the convergence so as to make the boundary layer visible. This answers an open problem proposed by Ukai \cite{Ukai}.

\end{abstract}
\maketitle

\section{Introduction}

The hydrodynamic limits from the Boltzmann equation got a lot of
interest in the previous two decades. Hydrodynamic regimes are those where the Knudsen number
$\eps$ is small.  The Knudsen number is the ratio of the mean free path and
the macroscopic length scales. The incompressible Navier-Stokes-Fourier (NSF) system can
be formally derived from the Boltzmann equation through a scaling in
which the fluctuations of the number density $F$ about an absolute
Maxwellian $M$ are scaled to be on the order $\eps$, see \cite{BGL1}.

The program that justifies the hydrodynamic limits from the
Boltzmann equation in the framework of DiPerna-Lions \cite{D-P} was
initiated by Bardos-Golse-Levermore \cite{BGL1, BGL2} in late 80's. Since then,
there has been lots of contributions to this program \cite{BGL3, GL,
Go-Sai04, Go-Sai05, JLM, LM, LM3, LM4, M-S, S-Ray}. In particular the
work of Golse and Saint-Raymond \cite{Go-Sai04}  is the first complete rigorous justification of
NSF limit from the Boltzmann equation in a class of
bounded collision kernels, without making any nonlinear weak
compactness hypothesis. They have recently extended their result to
the case of hard potentials \cite{Go-Sai05}. With some new nonlinear
estimates, Levermore and Masmoudi \cite{LM}  treated a broader class
of collision kernels which includes all hard potential cases and,
for the first time in this program, soft potential cases.

All of the above mentioned works were carried out in either the
periodic spatial domain or the  whole space, except for  \cite{M-S} and \cite{SRM}.
In \cite{M-S},  the linear Stokes-Fourier system was recovered  with the same
collision kernels assumption as in  \cite{GL}, while in \cite{SRM}, the Navier-Stokes limit was derived with the same kernels assumption as in \cite{Go-Sai05}, i.e. hard potential kernels. In \cite{M-S} and \cite{SRM}, the fluctuations of renormalized solutions to the Boltzmann equation in a bounded domain
(see \cite{Misch3}) was proved to pass to the limit and recovered
fluid boundary conditions, either Dirichlet, or Navier slip boundary condition, depending on the relative sizes of the
accommodation coefficient and the Knudsen number.

The dependance of the boundary conditions of the limiting fluid equations on the relative importance of the accommodation coefficient and the Knudsen number was observed by Sone and his collaborators. Their results, mostly formal, are presented in Chapter 3 and 4 in \cite{Sone2} for several types of kinetic boundary conditions. The work \cite{M-S} and \cite{SRM} rigorously justified the incompressible Stokes and Navier-Stokes equations  from Boltzmann equation imposed with Maxwell reflection boundary condition.

In his survey paper \cite{Ukai}, Ukai proposed the following question: {\em ``As far as the Boltzmann equation in a bounded domain is concerned, some progress has been made recently. In [37] , the convergence of the Boltzmann equation to the (linear) Stokes-Fourier equation was proved together with the convergence of the boundary conditions. It is a big challenging problem to extend the result to the nonlinear case and to strength the convergence so as to make visible the boundary layer."}  (In the above citation of Ukai's survey, the reference {\em [37]}  is the Saint-Raymond and Masmoudi's paper \cite{M-S}.)

In this paper and a forthcoming one, we study the incompressible NSF limit in a bounded domain from the Boltzmann equation with the Maxwell reflection boundary condition in which the accommodation might depend on the Knudsen number. We consider a  bounded domain $\Omega\subset \RD$, $\D\geq 2$, with
boundary $\pO\in C^2$. The NSF system governs  the
fluctuations of mass density, bulk velocity, and temperature $(\rho,
\mathrm{u}, \theta)$ about their spatially homogeneous equilibrium
values in a Boussinesq regime. Specifically, after a suitable choice
of units, these dimensionless fluctuations satisfy the incompressibility and
Boussinesq relations
\begin{equation}\label{incom-boussinesq}
\DIV \mathrm{u}=0\,, \quad \rho+ \theta=0\,,
\end{equation}
while their evolution is determined by the Navier-Stokes and heat equations
\begin{equation}\label{incom-NSF}
\begin{aligned}
 \p_t \mathrm{u}+ \mathrm{u}\!\cdot\!\grad \mathrm{u} +\grad p =\nu  \Delta_{\!x} \mathrm{u}\,,\quad  &\mathrm{u}|_{t=0}=\mathrm{u}_{0}\,,\\
 \p_t \theta+ \mathrm{u}\!\cdot\!\grad \theta =\tfrac{2}{\D+2}\kappa\Delta_x \theta\,,\quad  &\theta|_{t=0}=\theta_{0}\,,
\end{aligned}
\end{equation}
where $\nu >0$ is the kinematic viscosity and $\kappa> 0$ is the heat thermal conductivity.

Traditionally, two types of natural physical boundary conditions could be imposed for the incompressible NSF system
\eqref{incom-NSF}. The first is the homogeneous Dirichlet boundary condition, namely,
\begin{equation}\label{dirichlet-bc}
\mathrm{u}=0\,,\quad \theta=0\!\quad\mbox{on}\!\!\quad
\mathbb{R}^+\times \pO\,.
\end{equation}
The other is the so-called Navier slip boundary condition, which was proposed by Navier \cite{Navier}:
\begin{equation}\label{navier-slip-bc}
\begin{aligned}
&[2\nu  d(\mathrm{u})\!\cdot\!\mathrm{n}+\chi
\mathrm{u}]^{\mathrm{tan}}=0\,,\quad \mathrm{u}\!\cdot\!\mathrm{n}=0
\quad\mbox{on}\!\!\quad \mathbb{R}^+\times \partial\Omega\,,\\
&\kappa \partial_\mathrm{n} \theta +\chi \tfrac{\D+1}{\D+2}\theta=0
\quad\mbox{on}\!\!\quad\mathbb{R}^+\times \partial\Omega\,,
\end{aligned}
\end{equation}
where $d(\u)=\frac{1}{2}(\grad \mathrm{u} + \grad \mathrm{u}^\top)$ denotes the symmetric part of the stress tensor and $\p_\mathrm{n} $
denotes the directional derivative along the outer normal vector
$\mathrm{n}(x), x\in\pO$. In the above Navier boundary condition, $\chi>0$ is
the reciprocal of the slip length which depends on the material of
the container.

In the current work, for general cut-off collision kernels, namely in the framework of \cite{LM}, we justify the NSF system. Regarding  the {\em weak} convergence results, our proof is  basically the same as in \cite{M-S} and \cite{SRM}: the boundary conditions of the limiting NSF system depend on the ratio of the accommodation coefficient and the Kundsen number, namely when $\frac{\alpha_\eps}{\eps}\rightarrow \infty$ as $\eps\rightarrow 0$, Dirichlet condition is derived, while when   $\frac{\alpha_\eps}{\eps}\rightarrow \sqrt{2\pi}\chi$, the Navier-slip boundary condition is derived. The main difference is that \cite{SRM} used the same renormalizations of \cite{Go-Sai05}, applicable for hard potentials, while in the current work, we use the renormalization of \cite{LM}, which works for more  general cut-off kernels, including soft potentials.

The  main novelty of the current work
is the treatment of the Dirichlet boundary condition case.
Indeed, we prove that when $\frac{\alpha_\eps}{\eps}\rightarrow \infty$,
the convergence is {\em strong}. Furthermore, as a consequence of this strong convergence, the first correction to the infinitesimal Maxwellian, which is a quadratic term obtained from the Chapman-Enskog expansion with the Navier-Stokes scaling, is rigorously justified. We point out that  in all the  previous
works mentioned above, the convergence is in $w\mbox{-}L^1$, unless
the initial data is well-prepared, i.e. is hydrodynamic and satisfies the Boussinesq and
incompressibility relations. This weak  convergence is caused by the
persistence of fast acoustic waves. In the Navier-Stokes regime, the
Reynold number $Re$ is order $O(1)$, then the von
K\'{a}m\'{a}n relation $\eps=\frac{Ma}{Re}$
implies that in the fluid limit $\eps\rightarrow 0$, the Mach number
$Ma$ must go to zero. As is well know physically, one
expects that as $Ma\rightarrow 0$, fast acoustic waves are
generated and carry the energy of the potential part of the flow. For
the periodic flows, or for some particular boundary conditions such
as Navier condition \eqref{navier-slip-bc}, these waves subsist
forever and their frequency   grows with  $\eps$. Mathematically,
this means that the convergence is only {\em weak}. This phenomenon
happens in many singular limits of fluid equations among which we only mention \cite{LM1, LM2}.

One of the ingredients of the convergence proof is the treatment of the
acoustic waves which are highly oscillating. A compensated
compactness type argument was used by Lions and Masmoudi \cite{LM3}
to prove that these acoustic waves have no contribution on the equation satisfied by
the weak limit. This argument was previously used in the
compressible incompressible limit \cite{LM2}.

In \cite{DGLM}, a striking phenomenon, namely the damping of acoustic waves
caused by  the Dirichlet  boundary condition was found by
Desjardins, Grenier, Lions, and Masmoudi in considering the
incompressible limit of the isentropic compressible Navier-Stokes equations. In
the case of a viscous flow in a bounded domain with Dirichlet
boundary condition, and  under a generic assumption on the domain
(related to the so-called Schiffer's conjecture and the Pompeiu
problem \cite{Dalmasso}), they showed  that the acoustic waves are
instantaneously (asymptotically) damped, due to the formation of a
thin boundary layer in time.  This layer is caused by a boundary
layer in space and  dissipates the energy carried by the acoustic
waves. From a mathematical point of view, strong convergence
was obtained.

Inspired by the idea of \cite{DGLM}, the current paper considers the much more involved kinetic-fluid coupled case.  We prove that if the accommodation coefficient is bigger than the Knudsen number, there is
no need for the argument in \cite{LM2} since we can prove that the
acoustic waves  are damped instantaneously. Our work is based on the
construction of  viscous and kinetic Knudsen  boundary layers of
size  $\sqrt{\eps }$  and $\eps$. The main idea is to use a family of test functions
which solve approximately a scaled stationary linearized  Boltzmann equation  and can
capture the propagation of the fast acoustic waves. These test
functions are constructed through considering a family of
approximate eigenfunctions of a {\em dual} operator with a {\em
dual} kinetic boundary condition with respect to the original
Boltzmann equation. The approximate eigenvalue is the sum of several terms with different order of $\eps$: the leading term is purely imaginary, which
describes the acoustic mode, and the real part of the next order term is strictly {\em negative} which gives the strict dissipation when applying the test functions to the
renormalized Boltzmann equation.

In contrast to \cite{DGLM}, the approximate eigenfunctions include interior part and two boundary layers:  fluid viscous layer and kinetic Knudsen layer, while in \cite{DGLM}, only a fluid
boundary layer was necessary. Another important difference is that a generic
assumption on the domain had to be made in \cite{DGLM} (in
particular there are modes which are not damped in the disc), while in the current work,
this assumption is not needed. The reason is that we deal with the full acoustic system, namely including the temperature. The NSF system  has also some dissipation in the temperature equation which is ignored
in the isentropic model. (in particular this dissipation property holds in the case of the ball). This was also considered in \cite{J-M-2013} in which we reinforced  the result of \cite{DGLM}.

When the accommodation coefficient $\alpha_\eps$ is asymptotically larger than the Knudsen number $\eps$ in the sense that $\frac{\alpha_\eps}{\eps}\rightarrow \infty$ as $\eps\rightarrow 0$, the fluid limit is the NSF equations with Dirichlet boundary condition. For example, we can assume
$\ale = \chi\eps^\beta$ with $0 \leq \beta <1$. We found that $\beta=\frac{1}{2}$ is a threshold in the sense that the kinetic-fluid coupled boundary layers behave differently for $0 \leq \beta <\frac{1}{2}$ and $\frac{1}{2} \leq \beta <1$, but for both cases the kinetic-fluid layers have damping effect. The current paper focuses on the threshold case $\beta=\frac{1}{2}$ and we leave the other cases for a separate paper due to the more complex construction of the boundary layers.

One of the difficulties  of  the construction happens in the case  the Laplace operator $-\Delta_{\!x}$ with Neumann boundary condition has multiple eigenvalues. As a consequence, the dimension of the null space of the the operator $\AA - i\lambda^k_0$ is greater than one, where $\AA$ denotes the acoustic operator, and $\tfrac{\D}{\D+2}[\lambda^k_0]^2$ are eigenvalues for $k\in\mathbb{N}$ (for details see Section 5.2). Thus, as each stage of the construction of boundary layers, the terms in the null space of $\AA - i\lambda^k_0$ can not be determined uniquely. To completely determine all the terms in the ansaza of boundary layers, we have to add some orthogonality conditions. Surprisingly, all these orthogonality conditions are consistent, at least for the threshold case $\beta=\frac{1}{2}$ treated in the current paper. Similar idea has been used in \cite{J-M-2013} which can be applied to the compressible-incompressible limit of the full Navier-Stokes-Fourier system in a bounded domain.

A key role is played by the linearized kinetic boundary layer equation in the coupling of viscous and kinetic layers. More specifically, its solvability provides the boundary conditions of the fluid variables in the interior and viscous boundary layers which satisfy the acoustic systems with source terms and second order ordinary differential equations respectively. This linearized kinetic boundary layer equation has been studied extensively (see \cite{BCN, CGS, GP, GPS, UYY}). Applying the boundary layer equations to construct the two layer eigenfunctions is the main novelty of the current paper. To the best of our knowledge these  two layer  eigenfunctions are new even in the applied literature.

The paper is organized as follows: the next section contains
preliminary material regarding the Boltzmann equation in a bounded
domain. We state the main theorems in Section 3 which include the
weak convergence for the Navier slip boundary and strong convergence
for the Dirichlet boundary. In Section 4, we list some differential
geometry properties of the boundary $\pO$ as a submanifold of
$\RD$. Section 5 provides an introduction to the acoustic modes while
Section 6 is about the analysis of the kinetic boundary layer
equation whose solvability provides the boundary conditions of the
fluid variables. In Section 7, we present the constructions of the
test functions used in the proof of the Main Theorem. The proof of
the main proposition on the boundary layers is given in Sections 8 and 9.
In Section 10, we establish the weak convergence result of
the main theorem. Section 11 contains the  proof of the strong
convergence in  the Dirichlet boundary case using the test functions constructed in
Section 7.

\section{Boltzmann Equation in Bounded Domain}

Here we introduce the Boltzmann equation in a bounded domain,
only so far as to set our notations, which are  essentially those  of
\cite{BGL2} and \cite{M-S}. More complete introduction to the
Boltzmann equation can be found in \cite{Cer, CIP, Glassey, Sone2}.

\subsection{Maxwell Boundary Condition}

We consider $\Omega$, a smooth bounded domain of $\RD$, and
$\mathcal{O}=\Omega\times\RD$, the space-velocity domain. Let
$\mathrm{n}(x)$ be the outward unit normal vector at $x\in\pO$ and  let $\mathrm{d}\sigma_x $ be
 the Lebesgue measure on the boundary
$\pO$. We define the outgoing and incoming sets $\Sigma_{+}$ and
$\Sigma_{-}$  by
\begin{equation}\nonumber
\Sigma_{\pm}=\{(x,v)\in\Sigma:\pm \mathrm{n}(x)\!\cdot\! v>0\}\quad\mbox{where}\quad\Sigma=\pO\times \RD\,.
\end{equation}

Denoted by $\gamma F$ the trace of $F$ over $\Sigma$, the boundary
condition takes the form of a balance between the values of the
outgoing and incoming parts of $\gamma F$, namely
$\gamma_{\pm}F=\mathbf{1}_{\Sigma_{\pm}}\gamma F$. In order to
describe the interaction between particles and the wall, Maxwell
\cite{Maxwell} proposed in 1879 the following phenomenological law
which splits into a local reflection and a diffuse
reflection
\begin{equation}\label{Max-bc}
\gamma_{-}F=(1-\alpha)L\gamma_{+}F +\alpha K \gamma_{+}F \quad\mbox{on}\quad \Sigma_{-}\,,
\end{equation}
where $\alpha\in[0\,,1]$ is a constant, called the ``accommodation coefficient.'' The local reflection operator $L$ is given by
\begin{equation}\label{reflect-L}
L\phi(x\,,v)=\phi(x\,,R_x v)\,,
\end{equation}
where $R_x v=v-2\left[\mathrm{n}(x)\!\cdot\! v\right]\mathrm{n}(x)$
is the velocity before the collision with the wall. The diffuse
reflection operator $K$ is given by
\begin{equation}\nonumber
K\phi(x\,,v)=\sqrt{2\pi}\tilde{\phi}(x)M(v)\,,
\end{equation}
where $\tilde{\phi}$ is the outgoing mass flux
\begin{equation}\nonumber
\tilde{\phi}(x)=\int\limits_{v\cdot
\mathrm{n}(x)>0}\phi(x\,,v)v\!\cdot\!\mathrm{n}(x) \,\mathrm{d}v\,,
\end{equation}
and $M$ is the absolute Maxwellian $ M(v)=\tfrac{1}{(2\pi)^{\D/2}}\exp\left(-\tfrac{1}{2}|v|^2\right)$,
that corresponds to the spatially homogeneous fluid state with
density and temperature equal to 1 and bulk velocity equals to 0. Furthermore, We notice that
\begin{equation}\nonumber
\int\limits_{v\cdot \mathrm{n}(x)>0}v\!\cdot\!\mathrm{n}(x)\sqrt{2\pi}M(v)\,\mathrm{d}v
=\int\limits_{v\cdot \mathrm{n}(x)<0}|v\!\cdot\!\mathrm{n}(x)|\sqrt{2\pi}M(v)\,\mathrm{d}v=1\,,
\end{equation}
which expresses the conservation of mass at the boundary. Here we take the temperature of the wall to be constant and equal to 1.

\subsection{Nondimensionalized Form of the Boltzmann Equation}

We consider a sequence of renormalized solutions $F_\eps(t,x,v)$ to the rescaled Boltzmann equation
\begin{equation}\label{BE-F}
\begin{aligned}
\eps\p_t\Fe+v\cdot\nabla_x\Fe=\frac{1}{\eps}\mathcal{B}(\Fe\,,\Fe)\quad&\mbox{on}\quad\mathbb{R}^+\times\mathcal{O}\,,\\
\Fe(0\,,x\,,v)=\Fein (x\,,v)\geq 0\quad&\mbox{on}\quad\mathcal{O}\,,\\
\gamma_{-}\Fe=(1-\alpha)L\gamma_{+}\Fe+\alpha
K \gamma_{+}\Fe\quad&\mbox{on}\quad \mathbb{R}^+\times\Sigma_{-}\,.
\end{aligned}
\end{equation}
The Boltzmann collision operator $\mathcal{B}$ acts only on the $v$ argument of
$F$ and is formally given by
\begin{equation}\nonumber
\mathcal{B}(F\,,F)=\iint\limits_{\mathbb{S}^{D-1}\times
\RD}(F'_1F'-F_1F)b(\omega\,,v_1-v)\,\mathrm{d}\omega\,\mathrm{d}v_1\,,
\end{equation}
where $v_1$ ranges over $\RD$ endowed with its Lebesgue measure
$\mathrm{d}v_1$, while $\omega$ ranges over the unit sphere
$\SD=\{\omega\in\RD:|\omega|=1\}$ endowed with its rotationally
invariant unit measure $\mathrm{d}\omega$. The $F'_1\,,F'\,,F_1\,,$
and $F$ appearing in the integrand designate $F(t\,,x\,,\cdot)$
evaluated at the velocities $v'_1\,,v'\,,v_1$ and $v$, respectively,
where the primed velocities are defined by
\begin{equation}\nonumber
v'_1=v_1-\omega[\omega\cdot(v_1-v)]\,,\quad v'=v+\omega[\omega\cdot(v_1-v)]\,,
\end{equation}
for any given $(\omega\,,v_1\,,v)\in\SD\times\RD\times\RD$. This
expresses the conservation of momentum and energy for particle pairs
after a collision, namely,
\begin{equation}\nonumber
v+v_1=v'+v'_1\,,\quad
|v|^2+|v_1|^2=|v'|^2+|v'_1|^2\,.
\end{equation}
The collision kernel $b$ is a positive, locally integrable function and has
the classical form
\begin{equation}\nonumber
b(\omega\,,v)=|v|\Sigma(|\omega\cdot\hat{v}|\,,|v|)\,,
\end{equation}
where $\hat{v}=v/|v|$ and $\Sigma$ is the specific differential
cross section. This symmetry implies that the quantity $\int
b(\omega\,,v)\,\mathrm{d}\omega$ is a function of $|v|$ only. The
DiPerna-Lions theory requires that $b$ satisfies
\begin{equation}\label{b-DL}
\lim\limits_{|v|\rightarrow\infty}\frac{1}{1+|v|^2}\iint\limits_{\SD\times
K}b(\omega\,,v_1-v)\,\mathrm{d}\omega\,\mathrm{d}v_1=0
\end{equation}
for any compact set
$K\subset\RD$. There are some additional assumptions on $b$ needed
in \cite{LM}. For the convenience
of the reader, we list these assumptions here.

A major role will be played by the attenuation coefficient $a(v)$,
which is defined as
\begin{equation}\nonumber
a(v)=\int\limits_{\RD}\bar{b}(v_1-v)M_1\,\mathrm{d}v_1=\iint\limits_{\SD\times\RD}b(\omega,v_1-v)\,\mathrm{d}\omega
M_1\,\mathrm{d}v_1\,.
\end{equation}
A few facts about $a(v)$ are readily evident from what we have already assumed. Because \eqref{b-DL} holds, one can show that
\begin{equation}\label{assum1}
\lim\limits_{|v|\rightarrow \infty}\frac{a(v)}{1+|v|^2}=0\,.
\end{equation}

Our {\em second assumption} regarding the
collision kernel $b$ is that $a(v)$ satisfies a lower bound of the form
\begin{equation}\label{assum2}
C_a(1+|v|)^\alpha\leq a(v)\,,
\end{equation}
for some
constant $C_a>0$ and $\alpha\in\mathbb{R}$. The {\em third
assumption} is that there exists $s\in(1\,,\infty]$ and
$C_b\in(0\,,\infty)$ such that
\begin{equation}\label{assum3}
\left(\,\,\int\limits_{\RD}\left|\frac{\bar{b}(v_1-v)}{a(v_1)a(v)}\right|^s
a(v_1)M_1\,\mathrm{d}v_1\right)^{\frac{1}{s}}\leq C_b\,.
\end{equation}

Another major role in what follows will be played by the linearized
around the global Maxwellian $M$ collision operator $\mathcal{L}$,
which is defined by
\begin{equation}\label{linearized B}
\mathcal{L}\tilde{g}=\iint\limits_{\SD\times\RD}(\tilde{g}+\tilde{g}_1-\tilde{g}'-\tilde{g}'_1)b(\omega\,,v_1-v)\,\mathrm{d}\omega
M_1\,\mathrm{d}v_1\,.
\end{equation}
One has the decomposition
\begin{equation}\nonumber
\frac{1}{a}\mathcal{L}=\mathcal{I}+\mathcal{K}^--2\mathcal{K}^+\,,
\end{equation}
where the loss operator $\mathcal{K}^-$ and the gain operator
$\mathcal{K}^+$ are defined by
\begin{equation}\nonumber
\mathcal{K}^-\tilde{g}=\frac{1}{a}\int\limits_{\RD}\tilde{g}_1\bar{b}(v_1-v)M_1\,\mathrm{d}v_1\,,
\end{equation}
\begin{equation}\nonumber
\mathcal{K}^+\tilde{g}=\frac{1}{a}\int\limits_{\RD}(\tilde{g}'+\tilde{g}'_1)b(\omega,v_1-v)\,\mathrm{d}\omega
M_1\,\mathrm{d}v_1\,.
\end{equation} The {\em fourth assumption} regarding the
collision kernel $b$ is that
\begin{equation}\label{assum4}
\mathcal{K}^+:L^2(aM\mathrm{d}v)\rightarrow L^2(aM\mathrm{d}v)\quad\mbox{is
compact}\,.
\end{equation}
Combining the gain operator assumption
$\eqref{assum4}$ and the loss operator assumption $\eqref{assum3}$, we
conclude that
\begin{equation}\nonumber
\frac{1}{a}\mathcal{L}: L^p(aM\mathrm{d}v)\rightarrow
L^p(aM\mathrm{d}v)\quad\mbox{is Fredholm}
\end{equation}
for every $p\in(1\,,\infty)$. From this Fredholm property we can define the
psuedo-inverse of $\mathcal{L}$, called $\mathcal{L}^{-1}$:
\begin{equation}\nonumber
\mathcal{L}^{-1}: L^p(a^{1-p}M\mathrm{d}v)\cap\mbox{Null}^\perp(\LL)\rightarrow L^p(aM\mathrm{d}v)\,.
\end{equation}
Moreover, $\mathcal{L}^{-1}$ is a bounded operator.

The {\em fifth assumption} regarding $b$ is that for every $\delta > 0$ there exists $C_\delta$ such that $\bar{b}$ satisfies
\begin{equation}\label{assum5}
\frac{\bar{b}(v_1-v)}{1+\delta\frac{\bar{b}(v_1-v)}{1+|v_1-v|^2}} \leq C_\delta(1+a(v_1))(1+a(v))\quad\mbox{for every}\quad\! v_1, v \in \RD\,.
\end{equation}

It is well known that the null space of the linearized Boltzmann operator $\LL$ is given by Null$(\LL)\equiv\mbox{span} \{ 1, v_1\,,\cdots\,,
v_{\D}\,, |v|^2 \}$. Let $\PP$ be the orthogonal projection from $L^2(M\dd v)$ onto Null$(\LL)$, namely,
\begin{equation}\label{projection-p}
\PP \tilde{g}=\< \tilde{g} \> + v\cdot \< \tilde{g} \> +(\tfrac{|v|^2}{2}-\tfrac{\D}{2})\< (\tfrac{1}{\D}|v|^2-1)\tilde{g} \>\,,
\end{equation}
where the notation $\<\cdot\>$ is defined below in \eqref{bracket}.  Furthermore, we define $\PP^\perp=\mathcal{I} -\PP$. The matrix-valued function
$\mathrm{A}(v)$ and the vector-valued function $\mathrm{B}(v)$
are defined by
\begin{equation}\label{define-AB}
\mathrm{A}(v)=v \otimes v-\tfrac{1}{\D}|v|^2 \mathrm{I}\,,\quad \mathrm{B}(v)=\tfrac{1}{2}|v|^2 v-\tfrac{\D+2}{2}v\,.
\end{equation}
We also define a scalar-valued function $\mathrm{C}(v)$ by
\begin{equation}\label{define-C}
  \mathrm{C}(v)= \tfrac{1}{4}|v|^4 - \tfrac{\D+2}{2}|v|^2 + \tfrac{\D(\D+2)}{4}\,.
\end{equation}
It is easy to see that each entry of $\mathrm{A}$, $\mathrm{B}$ and $\mathrm{C}$
are in $L^2(a^{-1}M\,\dd v)\cap\mbox{Null}^\perp(\LL)$. Furthermore, $\mathrm{C}$ is perpendicular to each entry of $\mathrm{A}$ and $\mathrm{B}$. We also introduce
$\widehat{\mathrm{A}}\in L^2(aM\dd v; \mathbb{R}^{\D\times \D})$ and
$\widehat{\mathrm{B}}\in L^2(aM\dd v; \mathbb{R}^{\D})$  by
\begin{equation}\label{A-B-hat}
\widehat{\mathrm{A}}=\LL^{-1}\mathrm{A}\,,\quad \widehat{\mathrm{B}}=\LL^{-1}\mathrm{B}\,.
\end{equation}

Next, for the sake of simplicity, we take the following normalizations:
\begin{equation}\nonumber
\begin{aligned}
&\iiint\limits_{\SD\times\RD\times\RD}b(\omega\,,v_1-v)\,\mathrm{d}\omega M_1\,\mathrm{d}v_1\,\mathrm{d}v=1\,,\\
&\int\limits_{\SD}\,\mathrm{d}\omega=1\,,\quad\int\limits_{\RD}M\mathrm{d}v=1\,,\quad\int\limits_\Omega\,\mathrm{d}x=1\,,
\end{aligned}
\end{equation}
associated with the domains $\SD\times\RD\times\RD$, $\SD$, $\RD$ and $\Omega$ respectively, and
\begin{equation}\label{renorm-ini}
\iint\limits_{\Omega\times\RD}F^{\mathrm{in}}_\eps\,\mathrm{d}x\,\mathrm{d}v=1\,,
\end{equation}
associated with the initial data $F^{\mathrm{in}}_\eps$.

Because $M\mathrm{d}v$ is a positive unit measure on $\RD$, we denote by $\langle\xi\rangle$ the average over
this measure of any integrable function $\xi=\xi(v)$,
\begin{equation}\label{bracket}
\langle\xi\rangle=\int\limits_{\RD}\xi(v)M\mathrm{d}v\,,
\end{equation}
and the inner product on $L^2(M\mathrm{d}v)$
\begin{equation}\nonumber
 \langle\xi\,,\eta\rangle =\int\limits_{\RD}\xi(v)\overline{\eta(v)}M\mathrm{d}v\,,
\end{equation}
where $\overline{\eta}$ denotes the complex conjugate of $\eta$. Moreover, we also use the  following average on the boundary
\begin{equation}\label{B-average}
\langle\xi\rangle_{\pO}=\int\limits_{\RD}\xi(v)\left[\mathrm{n}(x)\!\cdot\! v\right]\sqrt{2\pi}M\mathrm{d}v\,,
\end{equation}
from which we have
$\langle\mathbf{1}_{\Sigma_+}\rangle_{\pO}=-\langle\mathbf{1}_{\Sigma_-}\rangle_{\pO}=1\,.$
Because $\mathrm{d}\mu =b(\omega\,,v_1-v)\,\mathrm{d}\omega
M_1\,\mathrm{d}v_1M\,\mathrm{d}v$ is a positive unit measure on
$\SD\times\RD\times\RD$, we denote by
$\langle\!\langle\Xi\rangle\!\rangle$ the average over this measure
of any integrable function $\Xi=\Xi(\omega\,,v_1\,,v)$\,
\begin{equation}\nonumber
\langle\!\langle\Xi\rangle\!\rangle=\iiint\limits_{\SD\times\RD\times\RD}\Xi(\omega\,,v_1\,,v)\,\mathrm{d}\mu \,.
\end{equation}
The measure $\mathrm{d}\mu$ is invariant under the coordinate transformations
\begin{equation}\nonumber
(\omega\,,v_1\,,v)\mapsto(\omega\,,v\,,v_1)\,,\quad(\omega\,,v_1\,,v)\mapsto(\omega\,,v'_1\,,v')\,.
\end{equation}
These are called $d\mu \mbox{-}symmetries$.

\subsection{Navier-Stokes Scaling}

The incompressible NSF system can be formally derived from the
Boltzmann equation through a scaling in which the fluctuations of
the kinetic densities $F_\eps$ about the absolute Maxwellian M are
scaled to be of order $\eps$. More precisely, we take
\begin{equation}\label{define-g}
F_\eps=MG_\eps=M(1+\eps g_\eps)\,.
\end{equation}
Rewriting equation $\eqref{BE-F}$ for $G_\eps$ yields
\begin{equation}\label{BE-G}
\begin{aligned}
\eps\p_t\Ge+v\cdot\nabla_x\Ge=\frac{1}{\eps}\mathcal{Q}(\Ge\,,\Ge)\quad&\mbox{on}\quad\mathbb{R}^+\times\mathcal{O}\,,\\
\Ge(0\,,x\,,v)=\Gein(x\,,v)\quad&\mbox{on}\quad\mathcal{O}\,,\\
\gamma_{-}\Ge=(1-\alpha)L\gamma_{+}\Ge+\alpha \< \gamp
G_\eps\>_{\!\pO}\quad&\mbox{on}\quad \mathbb{R}^+\times\Sigma_{-}\,,
\end{aligned}
\end{equation}
where the collision kernel $\mathcal{Q}$ is now given by
\begin{equation}\nonumber
\mathcal{Q}(G\,,G)=\iint\limits_{\SD\times\RD}(G'_1G'-G_1G)b(\omega\,,v_1-v)\mathrm{d}\omega
M_1\,\mathrm{d}v_1\,.
\end{equation}
In terms of $g_\eps$ the system $\eqref{BE-F}$ finally reads
\begin{equation}\label{BE-g}
\begin{aligned}
\eps\p_tg_\eps+\Divv  g_\eps+\tfrac{1}{\eps}\mathcal{L}g_\eps=\mathcal{Q}(g_\eps\,,g_\eps)\quad&\mbox{on}\quad\mathbb{R}^+\times
\mathcal{O}\,,\\
g_\eps(0\,,x\,,v)=g_\eps^{\mathrm{in}}(x\,,v)\quad&\mbox{on}\quad\mathcal{O}\,,\\
\gamma_{-}g_\eps=(1-\alpha)L \gamma_{+}g_\eps +\alpha \langle\gamma_+g_\eps\rangle_\pO\quad&\mbox{on}\quad \mathbb{R}^+\times\Sigma_{-}\,.
\end{aligned}
\end{equation}
\subsection{{\em A Priori} Estimates}
Due to the presence of the boundary, the classical {\em a priori}
estimates for the Boltzmann equation, namely the entropy and energy
bounds, are modified. First, because all particles arriving at the
boundary are reflected or diffused, we have conservation of mass,
which can be written as
\begin{equation}\nonumber
\int_{\Omega}\<\Ge\>\,\mathrm{d}x=\int_{\Omega}\<\Gein\>\,\mathrm{d}x=1\,.
\end{equation}
Multiplying the equation \eqref{BE-G} by $\log(\Ge)$ and integrating in $x$ and $v$, we get formally
\begin{equation}\nonumber
\begin{aligned}
&\eps\p_t\int_\Omega\<\Ge\log(\Ge)-\Ge+1\>\,\mathrm{d}x+\iint_\Sigma (\Ge\log(\Ge)-\Ge+1)v\!\cdot\!\mathrm{n}(x)\,\mathrm{d}\sigma_x M\mathrm{d}v\\
=&\frac{1}{\eps}\int_\Omega\<\log(\Ge)\mathcal{Q}(\Ge\,,\Ge)\>\,\mathrm{d}x\,.
\end{aligned}
\end{equation}
By denoting $h(z)=(1+z)\log(1+z)-z$, for $z>-1$ and using that it is a convex function, we can compute the boundary term in the following way:
\begin{equation}\nonumber
\begin{aligned}
\widetilde{\mathcal{E}}_\eps(\gamp\Ge)=&\iint_\Sigma [\Ge\log(\Ge)-\Ge+1]v\!\cdot\!\mathrm{n}(x)\,\mathrm{d}\sigma_x M\mathrm{d}v\\
=&\Sigpo\left[h(\eps
\gamp\gep)-h\left((1-\ale)\eps\gamp\gep+\ale\<\eps\gamp\gep\>_{\pO}\right)\right]v\!\cdot\!\mathrm{n}(x)
\,M\mathrm{d}v\mathrm{d}\sigma_x \\
\geq&\Sigpo\left[h(\eps \gamp\gep)-(1-\ale)h(\eps\gamp\gep)-\ale
h(\<\eps\gamp\gep\>_\pO)\right]v\!\cdot\!
\mathrm{n}(x)\,M\mathrm{d}v\mathrm{d}\sigma_x \\
=&\frac{\ale}{\sqrt{2\pi}}\mathcal{E}(\gamp\Ge)\,,
\end{aligned}
\end{equation}
where $\mathcal{E}(\gamp\Ge)$, the so-called Darroz\`{e}s-Guiraud information, is given by
\begin{equation}\nonumber
\mathcal{E}(\gamp\Ge)=\int_\pO\left[\<h(\eps\gamp\gep)\>_\pO-h(\eps\<\gamp\gep\>_\pO)\right]\,\mathrm{d}\sigma_x
\,.
\end{equation}
Jensen's inequality implies that $\mathcal{E}(\gamp\Ge)\geq 0$.
Noticing that $\widetilde{\mathcal{E}}_\eps(\gamp\Ge)\geq \frac{\ale}{\sqrt{2\pi}}\mathcal{E}(\gamp\Ge)$, we get the entropy inequality
\begin{equation}
H(\Ge(t))+\int^t_0\left(\frac{1}{\eps^2}R(\Ge(s))+\frac{1}{\eps}\widetilde{\mathcal{E}}_\eps(\gamp\Ge(s))\right)\,\mathrm{d}s\leq
H(\Gein)\,,
\end{equation}
where $H(G)$ is the relative entropy functional
\begin{equation}\nonumber
H(G)=\int_\Omega\<G\log(G)-G+1\>\,\mathrm{d}x\,,
\end{equation}
and $R(G)$ is the entropy dissipation rate functional
\begin{equation}\nonumber
R(G)=\int_\Omega
\left\langle\!\!\!\left\langle\frac{1}{4}\log\left(\frac{G'_1G'}{G_1G}\right)(G'_1G'-G_1G)
\right\rangle\!\!\!\right\rangle\,\mathrm{d}x\,.
\end{equation}

\subsection{DiPerna-Lions-(Mischler) Solutions}

We will work in the setting of renormalized solutions which were initially constructed by DiPerna and Lions
\cite{D-P} over the whole space $\RD$ for any initial data
satisfying natural physical bounds. Recently, their result was extended
to the case of a bounded domain by Mischler \cite{Misch1,
Misch2, Misch3} with general Maxwell boundary conditions
$\eqref{Max-bc}$.

The DiPerna-Lions-(Mishler) theory does not yield solutions that are known to
solve the Boltzmann equation in the usual weak sense. Rather, it
gives the existence of a global weak solution to a class of formally
equivalent initial value problems that are obtained by multiplying
$\eqref{BE-G}$ by $\Gamma'(G_\eps)$:
\begin{equation}\label{BE-Gamma}
\begin{aligned}
(\eps\p_t+\Divv )\Gamma(\Ge)=\frac{1}{\eps}\Gamma'(\Ge)\mathcal{Q}(\Ge\,,\Ge)\quad&\mbox{on}\quad \mathbb{R}^+\times\mathcal{O}\,,\\
\Ge(0\,,\cdot\,,\cdot)=\Gein\geq 0\quad&\mbox{on}\quad\mathcal{O}\,.
\end{aligned}
\end{equation}
Here the admissible function $\Gamma:[0\,,\infty)\rightarrow \mathbb{R}$ is continuously differentiable and for some constant $C_\Gamma<\infty$
its derivative satisfies
\begin{equation}
|\Gamma'(z)|\sqrt{1+z}\leq C_\Gamma\,.
\end{equation}

The weak formulation of the renormalized Boltzmann equation \eqref{BE-Gamma} is given by
\begin{equation}\label{weak-BE-G}
\begin{aligned}
\eps&\int\limits_\Omega\langle\Gamma(\Ge(t_2))Y\rangle\,\mathrm{d}x-\eps\int\limits_\Omega\langle\Gamma(\Ge(t_1))Y\rangle\,\mathrm{d}x\\
&-\int^{t_2}_{t_1}\int\limits_\Omega
\langle\Gamma(\Ge)\Divv  Y\rangle\,\mathrm{d}x\,\mathrm{d}t
+\int^{t_2}_{t_1}\int\limits_\pO\langle\Gamma(\gamma\Ge)Y[
\mathrm{n}(x)\cdot v]\rangle\,\mathrm{d}\sigma_x \,\mathrm{d}t\\
&=\frac{1}{\eps}\int^{t_2}_{t_1}\int\limits_\Omega\langle
\Gamma'(\Ge)\mathcal{Q}(\Ge\,,\Ge)Y \rangle\,\mathrm{d}x\,\mathrm{d}t\,,
\end{aligned}
\end{equation}
for every $Y\in C^1\cap
L^\infty(\bar{\Omega}\times\RD)$ and every
$[t_1\,,t_2]\subset[0\,,\infty]$. Moreover, the boundary condition is also understood in the renormalized sense:
\begin{equation}\label{weak-b-BE}
\Gamma(\gamn\Ge)=\Gamma\left((1-\alpha)L \gamp\Ge+\alpha\widetilde{ \Fe}\right)\quad\mbox{on}\quad\mathbb{R}^+\times\Sigma_-\,,
\end{equation} where the equality holds almost everywhere and in the sense of distribution.

\begin{Prop}
{\em (Renormalized solutions in bounded domain \cite{Misch3}) }  Let $b$ satisfy the
condition $\eqref{b-DL}$. Given any initial data $\Gein$ satisfying
\begin{equation}\label{entropy-class}
\iint\limits_{\mathcal{O}}G_\eps^{\mathrm{in}}(1+|v|^2+|\log
 G_\eps^{\mathrm{in}}|)\,M\mathrm{d}v\,\mathrm{d}x<+\infty\,,
\end{equation} there exists at least one
$\Ge\geq 0$ in $C([0\,,\infty);\,L^1(\,M\mathrm{d}v\,\mathrm{d}x))$ such that
$\eqref{weak-BE-G}$ and $\eqref{weak-b-BE}$ hold for all admissible
functions $\Gamma$. Moreover, $\Ge$ satisfies the following global
entropy inequality for all $t>0$:
\begin{equation}\label{entropy-inequality}
H(\Ge(t))+\frac{1}{\eps^2}\int^t_0
R(\Ge(s))\,\mathrm{d}s+\frac{1}{\eps}\int^t_0\widetilde{\mathcal{E}}_\eps(\gamma_+\Ge(s))\leq
H(\Gein)\,.
\end{equation}

\end{Prop}


\section{Statement of the Main Results}

In this section we state our main results on justifying the incompressible NSF limits with different boundary conditions depending on the quotient between the accommodation coefficients $\alpha_\eps$ and the Knudsen number $\eps$.

\subsection{Dirichlet Boundary Condition}
The main theorem of this paper is the following strong convergence
 to the NSF system with Dirichlet boundary
condition when the accommodation coefficient $\ale$ is much larger
than the Knudsen number $\eps$, i.e. $\frac{\ale}{\eps}\rightarrow \infty$
as $\eps \rightarrow 0$.
\begin{theorem}\label{Diri-limit}
{\em (Dirichlet Boundary Condition)\/} Let $b$ be a collision kernel
that satisfies conditions \eqref{assum1}, \eqref{assum2},
\eqref{assum3}, \eqref{assum4} and \eqref{assum5}. Let $G^{\mathrm{in}}_\eps$ be
any family of non-negative measurable functions of $(x,v)$ satisfying \eqref{entropy-class} and the renormalization
\eqref{renorm-ini}. Let $g^{\mathrm{in}}_\eps$ be the associated
family of fluctuations given by $G^{\mathrm{in}}_\eps=1+\eps
g^{\mathrm{in}}_\eps$. Assume that the families $G^{\mathrm{in}}_\eps$ and $g^{\mathrm{in}}_\eps$
satisfy
\begin{equation}\label{boundedness-entropy}
H(G^{\mathrm{in}}_\eps) \leq C^{\mathrm{in}} \eps^2\,,
\end{equation}
and
\begin{equation}\label{distribution}
\lim\limits_{\eps\rightarrow 0} \left( \< g^{\mathrm{in}}_\eps \>\,,
\< v g^{\mathrm{in}}_\eps \> \,, \< (\tfrac{|v|^2}{\D}-1)
g^{\mathrm{in}}_\eps \> \right)=(\rho^{\mathrm{in}},
\mathrm{u}^{\mathrm{in}}, \theta^{\mathrm{in}})\,,
\end{equation}
in the sense of distributions for some $(\rho^{\mathrm{in}},
\mathrm{u}^{\mathrm{in}}, \theta^{\mathrm{in}})\in L^2(\dd x\,;
\mathbb{R} \times \mathbb{R}^\D \times \mathbb{R})$. Let $G_\eps$ be any family of DiPerna-Lions renormalized solutions to the
Boltzmann equation \eqref{BE-G} that have $G^{\mathrm{in}}_\eps$ as initial values, and the accommodation
coefficient $\ale$ satisfies
\begin{equation}
\alpha_\eps= \sqrt{2\pi}\chi \sqrt{\eps}\,.
\end{equation}

Then the family of fluctuations $g_\eps$ given by \eqref{define-g} is relatively compact in  $L^1_{loc}(\dd
t;L^1(\sigma M\dd v\dd x))\,.$ Every limit point $g$ of $g_\eps$ has
the infinitesimal Maxwellian form
\begin{equation}\label{limit-g}
g=v\!\cdot\!
\mathrm{u}+\left(\tfrac{1}{2}|v|^2-\tfrac{\D+2}{2}\right)\theta\,,
\end{equation}
where $(\mathrm{u},\theta)\in C([0,\infty);L^2(\dd
x\,;\mathbb{R}^\D\times \mathbb{R})) \cap L^2(\dd t\,; H^1(\dd
x\,;\mathbb{R}^\D\times \mathbb{R}))$ with mean zero over $\Omega$,
and it satisfies the NSF system with Dirichlet
boundary condition \eqref{incom-boussinesq}, \eqref{incom-NSF}, and
\eqref{dirichlet-bc}, where kinematic viscosity $\nu $ and thermal
conductivity $\kappa$ are given by
\begin{equation}\label{mu-kappa}
\nu =\tfrac{1}{(\D-1)(\D+2)}\<\AHat\!:\!\LL\AHat\>\,,\quad \kappa=\tfrac{1}{\D}\<\BHat\!\cdot\!\LL\BHat\>\,.
\end{equation}
The initial data is given by
\begin{equation}\label{initial-data}
\mathrm{u}^0=\mathbb{P} \mathrm{u}^{\mathrm{in}}\,,\quad
\theta^0=\tfrac{\D}{\D+2}\theta^{\mathrm{in}}-\tfrac{2}{\D+2}\rho^{\mathrm{in}}\,.
\end{equation}
Here the operator $\mathbb{P}$ is the Leray's projection on the space of divergence free vector fields. Moreover, every subsequence $g_{\eps_k}$ of $g_\eps$ that converges to $g$ as $\eps_k\rightarrow 0$ also satisfies
\begin{equation}\label{moments-Diri}
\begin{aligned}
\< vg_{\eps_k}\>\rightarrow \mathrm{u}\quad&\hbox{in}\!\!\quad L^p_{loc}(\dd t;L^1(\dd x;\RD))\,,\\
\<(\tfrac{1}{\D}|v|^2-1)g_{\eps_k}\>\rightarrow
\theta\quad&\hbox{in}\!\!\quad L^p_{loc}(\dd t;L^1(\dd
x;\mathbb{R})) \quad\mbox{for every}\quad 1\leq p <\infty\,.
\end{aligned}
\end{equation}
Furthermore, $\frac{1}{\eps}{\mathcal{P}}^\perp g_\eps$ is
relatively compact in $w\mbox{-}L^1_{loc}(\dd t; w\mbox{-}L^1(\sigma
M \dd v \dd x))$. For every subsequence $\eps_k$ so that
$g_{\eps_k}$ converges to $g$,
\begin{equation}\label{P-Perp}
\begin{aligned}
\frac{1}{\eps}{\mathcal{P}}^\perp g_{\eps_k} \rightarrow
&\tfrac{1}{2}\mathrm{A} : \mathrm{u} \otimes \mathrm{u}+
\mathrm{B}\!\cdot\! \mathrm{u}\theta +
\tfrac{1}{2}\mathrm{C}\theta^2\\
& -\widehat{\mathrm{A}} : \grad \mathrm{u}- \widehat{\mathrm{B}}\!\cdot\!
\grad \theta\,,\quad \mbox{in}\!\!\quad w\mbox{-}L^1_{loc}(\dd t;
w\mbox{-}L^1(\sigma M \dd v \dd x))\,,
\end{aligned}
\end{equation}
as $\eps_k\rightarrow 0$, where $\mathrm{A}, \mathrm{B}, \mathrm{C}$ and $\AHat, \BHat$ are defined in \eqref{define-AB}, \eqref{define-C} and \eqref{A-B-hat}.

\end{theorem}
\noindent{\bf Remark:} In the formal Chapman-Enskog expansion,
\begin{equation}\nonumber
g_\eps=g+ \eps \mathcal{P}^\perp g_1 + \eps \mathcal{P}g_1 + \eps^2 g_2 + \cdots\,,
\end{equation}
where $g$ is given by \eqref{limit-g} and $\mathcal{P}^\perp g_1$ is the righthand side term in \eqref{P-Perp}. In previous works \cite{Go-Sai04, Go-Sai05,
LM}, under the assumptions \eqref{boundedness-entropy} and
\eqref{distribution}, the convergence to \eqref{limit-g} and
\eqref{moments-Diri} are only in $w\mbox{-}L^1$. So
the convergence to the quadratic term \eqref{P-Perp}, which is the
first correction to the infinitesimal Maxwellian that one obtains
from the Chapman-Enskog expansion with the Navier-Stokes scaling,
could not be obtained. In Theorem \ref{Diri-limit}, by showing the
acoustic waves are instantaneously damped, we justify not only the
strong convergence to the leading order term $g$, but also weak
convergence to the kinetic part of the next order corrector \eqref{P-Perp}.

\subsection{Navier Boundary Condition}
The second result is about Navier boundary condition. For this case, although the coupled viscous boundary layer and the
Knudsen layer still have dissipative effect, however, the damping
happens a longer time scale $O(1)$. Consequently, unlike the Dirichlet
boundary condition case, the fast acoustic waves can be damped, but
{\em not instantaneously}. Nevertheless, we can show the weak
convergence result, thus justify the NSF limit with slip Navier boundary
condition, while the linear Stokes-Fourier limit was justified in \cite{M-S}.


\begin{theorem}\label{Navier-limit}
{\em (Navier Boundary Condition)\/} With the same assumptions with Theorem \ref{Diri-limit}, except that the accommodation coefficients satisfy
\begin{equation}\label{finite-lamda}
\frac{\ale}{\sqrt{2\pi}\eps}\rightarrow \chi\,,\quad\mbox{as}\quad
\eps\rightarrow 0\,.
\end{equation}

Then the family $g_\eps$ is relatively compact in
$w\mbox{-}L^1_{loc}(\dd t;w\mbox{-}L^1(\sigma M\dd v\dd x))\,.$
Every limit point $g$ of $g_\eps$ in $w\mbox{-}L^1_{loc}(\dd
t;w\mbox{-}L^1(\sigma M\dd v\dd x))$ has the infinitesimal
Maxwellian form as \eqref{limit-g} in which $(\mathrm{u},\theta)\in C([0,\infty);L^2(\dd x\,; \RD \times
\mathbb{R})) \cap L^2(\dd t; H^1(\dd x\,;\RD \times \mathbb{R}))$
is a Larey solution of the
NSF system with Navier boundary condition
\eqref{incom-boussinesq}, \eqref{incom-NSF}, and
\eqref{navier-slip-bc}, where kinematic viscosity $\nu $ and thermal
conductivity $\kappa$ are given by \eqref{mu-kappa}, the initial
data is given by \eqref{initial-data}.

Moreover, every subsequence $g_{\eps_k}$ of $g_\eps$ that converges to $g$ as $\eps_k\rightarrow 0$ also satisfies
\begin{equation}\label{moments-Navier}
\begin{aligned}
\mathbb{P}\< vg_{\eps_k}\>\rightarrow \mathrm{u}&\quad\hbox{in}\!\!\quad C([0,\infty)\,;\mathcal{D}'(\Omega\,;\RD))\,,\\
\<(\tfrac{1}{\D+2}|v|^2-1)g_{\eps_k}\>\rightarrow
\theta&\quad\hbox{in}\!\!\quad C([0,\infty)\,;w\mbox{-}L^1(\Omega\,;\mathbb{R}))\,.
\end{aligned}
\end{equation}

\end{theorem}

\noindent{\bf Remark:} For the Navier-slip boundary condition case, since the convergence is weak, the convergence \eqref{P-Perp}, i.e. the justification of the first correction to the infinitesimal Maxwellian in the Chapman-Enskog expansion can not be obtained.

\section{Geometry of the boundary $\pO$}

In this section, we collect some differential geometry properties
related to the boundary $\pO$ which can be considered as a $(\D-1)$
dimension Riemannian manifold with a metric induced from the
standard Euclidian metric of $\RD$. From the following classical result in geometry (for the proof, see \cite{Simon}), there is a tubular neighborhood $\Omega_{\delta} = \{ x \in \Omega:
\mbox{dist}(x, \pO) < \delta\}$ of $\pO$ such that the nearest point
projection map is well defined.

\begin{Lemma}\label{Projection}
If $\pO$ is a compact $C^k$ submanifold of dimension $\D-1$ embedded
in $\RD$, then there is $\delta = \delta_{\pO}>0$ and a map $\pi \in
C^{k-1}(\Omega^\delta\,; \RD)$
such that the following properties hold:

$(i)$: for all $x \in \Omega \subset \RD$ with
$\mathrm{dist}(x\,,\pO) < \delta\,;$
\begin{equation*}
\pi(x) \in \pO\,, \quad x - \pi(x) \in T^\perp_{\Pi(x)}(\pO)\,,\quad |
x - \pi(x)| = \mathrm{dist}(x\,, \pO)\,, and
\end{equation*}
\begin{equation*}
|z - x| > \mathrm{dist}(x\,,\pO)\quad \mbox{for any}\quad\! z \in
\pO\setminus\{\pi(x)\} \,;
\end{equation*}

$(ii)$:
\begin{equation*}
\pi(x + z) \equiv x\,, \quad \mbox{for}\quad\! x \in \pO\,, z \in
T_x(\pO)^\perp\,, |z| < \delta\,,
\end{equation*}

$(iii)$: Let $\mathrm{Hess}\Pi^x$ denote the Hessian of $\pi$ at
$x$, then
\begin{equation*}
\mathrm{Hess}\pi^x(V_1\,, V_2) = \mathrm{h}_x(V_1\,, V_2)\,,
\quad\mbox{for}\quad\! x\in\pO\quad\! V_1\,,V_2 \in T_x(\pO)\,,
\end{equation*}
where $\mathrm{h}_x$ is the second fundamental form of $\pO$ at $x$.
\end{Lemma}

The viscous boundary layer has significantly different behavior over the tangential and normal directions near the boundary. This inspire us to consider the following new coordinate system, which we call the curvilinear coordinate for the tubular neighborhood $\Omega^\delta$ defined in Lemma \ref{Projection}. Because $\pO$ is a $(\D-1)$ dimensional manifold, so locally
$\pi(x)$ can be represented as
\begin{equation}\label{Pai}
\pi(x) = (\pi^1(x)\,, \cdots\,, \pi^{\D-1}(x))\,.
\end{equation}

More precisely, the representation \eqref{Pai} could be understood in the following sense: we can introduce a new coordinate system $(\xi^1\,,\cdots\,, \xi^\D)$ by
a homeomorphism which locally defined as $\xi: \xi(x)=(\xi'(x)\,,\xi^\D(x))$ where $\xi'= (\xi^1\,,\cdots\,, \xi^{\D-1})$, such that $\xi(\pi(x))=(\xi',\, 0)$ and $\mathrm{d}(x)= \xi^\D$, where $\mathrm{d}(x)$ is the distance function to
the boundary $\partial\Omega$, i.e.
\begin{equation}\label{function-d}
\mathrm{d}(x) = \mathrm{dist}(x\,,\pO) = |x - \pi(x)|\,.
\end{equation}
For the simplicity of notation, we denote $``\xi'(x)=\pi(x)"$ which is the meaning of \eqref{Pai}.

It is easy to see that $\grad \mathrm{d}$ is perpendicular to the
level surface of the distance function $\mathrm{d}$, i.e. the set
$S^z = \{ x\in \Omega: \mathrm{d}(x) = z\}$. In particular, on the
boundary, $\grad \mathrm{d}$ is perpendicular to $S_0 = \pO$.
Without loss of generality, we can normalize the distance
function so that $\grad \mathrm{d}(x) = -\mathrm{n}(x)$ when $x \in
\pO$. By the definition of the projection $\Pi$, we have
\begin{equation}\label{Proj-Tangent}
\pi(x + t \grad\mathrm{d}(x)) = \pi(x)\quad \mbox{for}\quad\!\!\! t\quad\!\!\!\mbox{small}\,,
\end{equation}
and consequently, $\grad \pi^\alpha \cdot \grad \mathrm{d} = 0$, for
$\alpha = 1\,,\cdots\,, \D-1$. In particular, for $t$ small enough,
$\grad \pi^\alpha(x) \in T_x(\pO)$ when $x\in \pO$.

Next, we calculate the induced Riemannian metric from $\RD$ on
$\pO$. In a local coordinate system, these Riemannian metric can be
represented as
\begin{equation}\nonumber
g = g_{\alpha \beta} \mathrm{d}\pi^\alpha \otimes \mathrm{d}
\pi^\beta\,,
\end{equation}
where $g_{\alpha \beta}= \< \tfrac{\p}{\p{\pi^\alpha}}\,,
\tfrac{\p}{\p{\pi^\beta}}\>$. Noticing that $\tfrac{\p}{\p{x^i}} =
\tfrac{\p\pi^\alpha}{\p{x^i}}\tfrac{\p}{\p{\pi^\alpha}}$, and
$\<\tfrac{\p}{\p{x^i}}, \tfrac{\p}{\p{x^j}} \> =\delta_{ij}$, the
metric $g_{\alpha \beta}$ can be determined by
\begin{equation}\nonumber
g_{\alpha
\beta}\tfrac{\p\pi^\alpha}{\p{x^i}}\tfrac{\p\pi^\beta}{\p{x^i}}
=1\,.
\end{equation}


\section{Acoustic Modes}\label{Acoustic-Mode}
\subsection{Acoustic Operator $\AA$}
Recall that the Leray's projection $\PBbb$ on the space of
divergence-free vector fields and $\mathbb{Q}$ on the space of
gradients are defined by
\begin{equation}\nonumber
\PBbb=\IBbb-\mathbb{Q}\,,
\end{equation}
where  $\mathbb{Q} \mathrm{u} = \grad q$ and $q$ solves
\begin{equation}\label{Leray-Proj}
\begin{aligned}
\Delta_{\!x} q &= \DIV \mathrm{u}\quad\mbox{in}\quad \Omega\,,\\
\grad q\!\cdot\!\mathrm{n}&= \mathrm{u}\!\cdot\!\mathrm{n}
\quad\mbox{on}\quad\partial\Omega\,,\quad\mbox{and}\quad\int_\Omega
q\,\mathrm{d}x=0\,.
\end{aligned}
\end{equation}

We define Hilbert spaces
\begin{equation}\nonumber
\begin{aligned}
\mathbb{H}&=\left\{ U=(\rho\,, \mathrm{u}\,, \theta)\in L^2(\dd{x};\mathbb{C}\times\mathbb{C}^{\mathrm{D}}\times\mathbb{C})\right \}\,,\\
\mathbb{V}&=\left\{ U\in \mathbb{H}: \int_\Omega |\grad U|^2\,\dd
x<\infty \right\}\,,
\end{aligned}
\end{equation}
endowed with inner product
\begin{equation}\label{innerproduct}
\<U_1\,,U_2\>_\mathbb{H}=\int_\Omega(\rho_1\overline{\rho}_2+\mathrm{u}_1\cdot\overline{\mathrm{u}}_2
+\tfrac{\D}{2}\theta_1\overline{\theta}_2)\,\dd
x\,,
\end{equation}
where $\overline{f}$ denotes the complex conjugate of the complex-valued function $f$. Next, we define the acoustic operator $\AA$:
\begin{equation}\label{acoustic-A}
\AA\begin{pmatrix}\rho\\ \mathrm{u}\\
\theta\end{pmatrix}=\begin{pmatrix}\DIV \mathrm{u}\\ \nabla_x(\rho+\theta)\\
\tfrac{2}{\D}\DIV \mathrm{u}\end{pmatrix}\,,
\end{equation}
over the domain
\begin{equation}\nonumber
\mathrm{Dom}(\AA)=\left\{ U=(\rho\,, \mathrm{u}\,, \theta)\in \mathbb{V}:
\mathrm{u}\!\cdot\!\mathrm{n}=0 \ \hbox{on} \ \partial
\Omega\right\}\,.
\end{equation}

The null space of $\mathcal{A} $ and its orthogonal with respect to
the inner product \eqref{innerproduct} are characterized as
\begin{equation}\label{null-A}
\mbox{Null}(\mathcal{A})=\{(-\varphi\,,\mathrm{w}\,,\varphi)\in
\mathbb{V}: \DIV \mathrm{w}=0\quad\!\mbox{and}\quad\!
\mathrm{w}\!\cdot\! \mathrm{n}=0\quad\!\mbox{on}\quad\!\pO\}\,,
\end{equation}
and
\begin{equation}\label{null-orthA}
\mbox{Null}(\mathcal{A} )^\perp=\{(\rho\,,\mathrm{u}\,,\theta)\in
\mathbb{V}:\theta=\tfrac{2}{\D}\rho\,,\mathrm{u}=\grad\phi\,,\quad\!\!\mbox{for
some}\quad\!\!\phi\in H^1{(\Omega)}\}\,,
\end{equation}
respectively. Because Null$(\AA)$ includes the incompressibility and Boussinesq relations,
we call it {\em incompressible} regime. We will see in the next subsection that Null$(\AA)^\perp$ is spanned by the eigenspaces of the acoustic operator $\AA$, so we call it {\em acoustic} regime.

For any $U=(\rho\,,\mathrm{u}\,,\theta)\in \mathbb{H}$, we can
define $\Pi$ and $\Pi^\perp$ the projections to the incompressible
regime $\mbox{Null}(\mathcal{A})$ and acoustic regime
$\mbox{Null}(\mathcal{A} )^\perp$ respectively as follows:
\begin{equation}\nonumber
\begin{aligned}
\Pi U&=\left(\tfrac{2}{\D+2}\rho-\tfrac{\D}{\D+2}\theta\,,\mathbb{P}\mathrm{u}\,,\tfrac{\D}{\D+2}\theta
-\tfrac{2}{\D+2}\rho\right)\,,\\
\Pi^\perp U&=\left(\tfrac{\D}{\D+2}(\rho+\theta)\,,\mathbb{Q}
\mathrm{u}\,,\tfrac{2}{\D+2}(\rho+\theta)\right)\,.
\end{aligned}
\end{equation}

\subsection{Eigenspaces of $\AA$}

The eigenvalues and eigenvectors of the acoustic operator $\AA$ in a
bounded domain can be constructed from those of the Laplace operator
with Neumann boundary condition in the following way: Let $\tfrac{\D}{\D+2}[\lambda^k]^2\,,\lambda^k
>0\,, k\in\mathbb{N}$ be the nondecreasing sequence of eigenvalues of the Laplace operator $-\Delta_{\mathrm{N}}$ with
homogeneous Neumann boundary condition, and
$\Psi^k$ be the corresponding orthonormal basis of $L^2(\Omega)$ eigenfunctions:
\begin{equation}\label{laplac-neu}
-\Delta_{\!
x}\Psi^k=\tfrac{\D}{\D+2}[\lambda^k]^2\Psi^k
\quad\mbox{in}\quad \Omega,\quad\grad\Psi^k \!\cdot\!
\mathrm{n}=0\,\quad\mbox{on}\quad\partial\Omega\,.
\end{equation}
More specifically,
\begin{equation}\nonumber
0 < \lambda^1 \leq \lambda^2 \leq \cdots \leq \lambda^k \rightarrow +\infty\,,\quad \mbox{as}\quad\! k \rightarrow \infty\,.
\end{equation}

Let $\tau$ denote either $+$ or $-$, and $\lambda^{\tau,k} =
\tau\lambda^k$. It can be verified that $i\lambda^{\tau,k}$ are non-zero eigenvalues of $\AA$ and
\begin{equation}\label{eigen-A}
U^{\tau,k}= \sqrt{\tfrac{\D+2}{2\D}}\left(\tfrac{\D}{\D+2}\Psi^k\,,\tfrac{\nabla_x\Psi^k}{
i\lambda^{\tau,k}}\,, \tfrac{2}{\D+2}\Psi^k \right)^\top
\end{equation}
are the corresponding normalized eigenvectors, i.e.
\begin{equation}\label{acoustic-equ}
\AA U^{\tau,k}= i\lambda^{\tau,k} U^{\tau,k}\,,
\end{equation}
and furthermore, $U^{\tau,k}$ span Null$(\mathcal{A} )^\perp$ under the inner product \eqref{innerproduct}. Consequently we have an orthonormal basis of
the acoustic modes, i.e.
\begin{equation}\nonumber
\mbox{Null}(\mathcal{A} )^\perp
=\overline{\mbox{Span}\left\{U^{\tau,k}|k\in
\mathbb{N}\,, \tau=\pm\right\}}^{L^2}\,.
\end{equation}

Moreover, we can use the components of $U^{\tau,k}$  to construct the infinitesimal Maxwellians $g^{\tau,k}$ which are in the null space of $\LL$:
\begin{equation}\label{acoustic-modes}
g^{\tau,k}= \sqrt{\tfrac{\D+2}{2\D}}\left\{\tfrac{\D}{\D+2}\Psi^k+v\!\cdot
\tfrac{\nabla_x\Psi^k}{ i\lambda^{\tau,k}}+
\tfrac{2}{\D+2}\Psi^k
(\tfrac{|v|^2}{2}-\tfrac{\D}{2})\right\}\,.
\end{equation}
These infinitesimal Maxwellians will be the building blocks of the
approximate eigenfunctions of $\frac{1}{\eps}\LL -
v\!\cdot\! \grad$.

\subsection{Conditions on $\Psi^k$}
Note that $\Psi^k$, $k \geq 1$ are solutions to the Neumann boundary condition equation \eqref{laplac-neu}, so some {\em orthogonality condition} is required for the eigenfunctions associated to the eigenvalues with multiplicity greater
than 1. Assume that $\lambda^2$ is an eigenvalue of \eqref{laplac-neu} and denote by $\mathrm{H}_0=\mathrm{H}_0(\lambda)$ the eigenspace associated to $\lambda^2$, i.e.
\begin{equation}\label{Hspace}
     \mathrm{H}_0(\lambda)= \{\Psi \in \mathrm{Dom}(-\Delta_{\!x}): -\Delta_{\!x} \Psi= \lambda^2 \Psi\quad\!\! \mbox{in}\quad\!\! \Omega, \tfrac{\p \Psi}{\p \mathrm{n}}=0 \quad\!\! \mbox{on}\quad\!\! \pO\}
\end{equation}
where $\mathrm{Dom}(-\Delta_{\!x})= H^2(\Omega)\cap \{\Psi | \tfrac{\p \Psi}{\p \mathrm{n}}=0 \quad\!\! \mbox{on}\quad\!\! \pO\}$ denotes the domain of $-\LAP$ with Neumann boundary condition. On the finite dimensional space $\mathrm{H}_0(\lambda)$, we can define a quadratic form $Q_1$. Its associated bilinear form that we still denote $Q_1$ and a symmetric operator $L_1= L^\lambda_1$ by
\begin{equation}\label{L1def}
   Q_1(\Psi,\Phi) = \int_\Omega L_1(\Psi)\Phi\,\mathrm{d}x= \int_\Omega L_1(\Phi)\Psi\,\mathrm{d}x\,.
\end{equation}
The eigenspace $\mathrm{H}_0(\lambda)$ is endows with an orthogonality condition
\begin{equation}\label{condition0}
     Q_1(\Psi^k,\Psi^l)=0\,,\quad \mbox{if}\quad\! \Psi^k, \Psi^l\in \mathrm{H}_0(\lambda)\quad\!\mbox{and}\quad\! k\neq l\,.
\end{equation}
This condition means that the eigenvectors $\Psi^k$ for $\lambda^k=\lambda$ are orthogonal for the symmetric operator $L^\lambda_1$. Of course, since $L^2(\Omega)$ is the direct sum of the spaces $\mathrm{H}_0(\lambda)$ for different $\lambda's$. From the definition of $L^\lambda_1$ on each eigenspace $\mathrm{H}_0(\lambda)$, we can define an operator $L_1$ on $L^2(\Omega)$ which leaves each eigenspace $\mathrm{H}_0(\lambda)$ invariant.
But this is not necessary, so we will think of $L_1= L^\lambda_1$ as acting on $\mathrm{H}_0(\lambda)$ for a fixed multiple eigenvalue $\lambda$.

The orthogonality condition \eqref{condition0} turns out to be enough for the construction of the boundary layer if the eigenvalues of $L_1$ are
simple, namely, if $\lambda^k_1 \neq \lambda^l_1$ for all $k \neq l$ such that $\lambda^k_0=\lambda^l_0=\lambda$. However, if $\lambda_1$ is an eigenvalue of $L_1$ with multiplicity greater than or equal to $2$, then we need an extra orthogonality condition. Let $\mathrm{H}_1=\mathrm{H}_1(\lambda_1)$ be defined by
\begin{equation}\label{Hspace1}
    \mathrm{H}_1= \{ \Psi\in \mathrm{H}_0 : L_1 \Psi = \lambda_1 \Psi \}\,.
\end{equation}
 On the finite dimensional space $\mathrm{H}_1$, there exists a quadratic form $Q_2$ and a symmetric operator $L_2$ (see the definition below), the extra condition is
 \begin{equation}\label{condition1}
     Q_2(\Psi^k,\Psi^l)=0\,,\quad \mbox{if}\quad\! \Psi^k, \Psi^l\in \mathrm{H}_1(\lambda)\quad\!\mbox{and}\quad\! k\neq l\,.
\end{equation}
 This condition is enough if $L_2$ has only simple eigenvalue on the vector space $H_1$. This process can be continued inductively.

 Let us now explain more precisely the condition we have to impose on the eigenvectors of $-\LAP$. We can construct recursively, on each eigenspace $\mathrm{H}_0(\lambda)$ of $-\LAP$, a sequence of symmetric operators $L_q, q \in \mathbb{N}$ in the following way: Let $L_0= - \LAP$, we define $L_1$ on each one of the eigenspace $\mathrm{H}_0(\lambda)$ of $L_0$ by \eqref{L1def}. Assume that the operators $L_p$ were constructed for $p
 \leq q-1, q \geq 2$ in such a way that each operator $L_p$ leaves invariant the eigenspaces of the operators $L_{p'}$ for $p' < p$. Now, to construct $L_q$, it is enough to construct $L_q$ on each eigenspace $\mathrm{H}_1(\lambda_1)\cap \mathrm{H}_2(\lambda_2)\cap \cdots \cap \mathrm{H}_{q-1}(\lambda_{q-1})$, where $\lambda_1, \lambda_2, \cdots, \lambda_{q-1}$ are eigenvalues of $L_1, L_2, \cdots, L_{q-1}$ respectively. This is done by constructing a quadratic form $Q_q$ on each space $\mathrm{H}_1(\lambda_1)\cap \mathrm{H}_2(\lambda_2)\cap \cdots \cap \mathrm{H}_{q-1}(\lambda_{q-1})$ and defining $L_q$ by
 \begin{equation}\nonumber
    Q_q(\Psi, \Phi)= \int_\Omega L_q(\Psi)\Phi\,\mathrm{d}x\,,\quad\!\mbox{for all}\quad\! \Psi,\Phi\in \mathrm{H}_1(\lambda_1)\cap \mathrm{H}_2(\lambda_2)\cap \cdots \cap \mathrm{H}_{q-1}(\lambda_{q-1})\,.
 \end{equation}
The precise construction of the quadratic form $Q_q$ on the space $\mathrm{H}_1(\lambda_1)\cap \mathrm{H}_2(\lambda_2)\cap \cdots \cap \mathrm{H}_{q-1}(\lambda_{q-1})$ will be done in the proof.

Let $N\in \mathbb{N}$ be an integer. This is the integer that will appear in the order of the approximation in the Proposition \ref{main-prop} . The eigenvectors $\Psi^k$ for $\lambda^k_0=\lambda$ should be chosen in such a way that they are eigenvectors for all the operators $L_n$ at least for $n \leq N+2$. This implies that they are orthogonal to all the operators $L_n$ for $n \leq N+2$, which means that
 \begin{equation}\label{condition-n}
     Q_n(\Psi^k,\Psi^l)=\int_\Omega L_n(\Psi^k)\Psi^l=0\,,
\end{equation}
if $\Psi^k, \Psi^l\in \mathrm{H}_1(\lambda_1)\cap \mathrm{H}_2(\lambda_2)\cap \cdots \cap \mathrm{H}_{n-1}(\lambda_{n-1})$ and $ k\neq l$.

\noindent{\bf Remark: } The precise construction of the quadratic form $Q_q$ will be done in the proof of Proposition \ref{main-prop}.


\subsection{The operator $\AA-
i \lambda^{\tau,k}$}

Later on, in the construction of the boundary layers, for each acoustic mode $k \geq 1$ and
$\tau= +$ or $-$, we will frequently solve the following linear hyperbolic system for $V^{\tau,k}=(\rho^{\tau,k},
\mathrm{v}^{\tau,k}, \theta^{\tau,k})^\top$:
\begin{equation}\label{A-Lambda}
\begin{aligned}
(\AA- i \lambda^{\tau,k})V^{\tau,k} & = i \mu^{\tau,k}
U^{\tau,k} +
F^{\tau,k}\,,\\
\mathrm{v}^{\tau,k}\!\cdot\! \mathrm{n} & = g^{\tau,k}\quad
\mbox{on}\quad\! \pO\,.
\end{aligned}
\end{equation}
where $\mu^{\tau,k}$, $F^{\tau,k}$ and $g^{\tau,k}$ are given, and $U^{\tau,k}$ is defined in
\eqref{eigen-A}.

\noindent{\bf Remark:} Strictly speaking, \eqref{A-Lambda} is not
rigorous because  $\mathrm{v}^{\tau,k}\!\cdot\!\mathrm{n}$ is
non-zero, so $V^{\tau,k}$ is not in the domain of $\AA$. For
notational simplicity, we still use $\AA$ in \eqref{A-Lambda} and
later on, just mean the expression of $\AA$ in \eqref{acoustic-A}
regardless of the domain.

To solve the system \eqref{A-Lambda}, the main difficulty is that the kernel of $\AA- i
\lambda^{\tau,k}$ is nontrivial. It will be more involved when the
eigenvalues have multiplicity greater than $1$. It can be characterized that the kernel and the orthogonal of $\AA- i
\lambda^{\tau,k}$ with respect to the inner product \eqref{innerproduct} are
\begin{equation}\nonumber
\mathrm{Ker}\big(\AA- i \lambda^{\tau,k} \big) =
\mathrm{Span}\big\{U^{\tau,l}: \mbox{for all}\quad\!\! l\in\mathbb{N}\quad\!\! \mbox{such that} \quad\!\!\lambda^l = \lambda^k \big\}\,,
\end{equation}
and
\begin{equation}\nonumber
\begin{aligned}
\mathrm{Ker}\big(\AA- i \lambda^{\tau,k}\big)^\perp = &
\mathrm{Span}\big\{U^{\delta,l}: \mbox{for all}\quad\!\!\delta=\pm\quad\!\!\mbox{and}\quad\!\! l\in\mathbb{N}\quad\!\! \mbox{such that}\quad\!\!\lambda^l \neq \lambda^k \big\}\\
& \oplus \mathrm{Span}\big\{U^{-\tau,l}: \lambda^l =
\lambda^k \big\}\oplus \mbox{Null}(\AA)\,.
\end{aligned}
\end{equation}

Next, we define a bounded pseudo inverse of $\AA- i \lambda^{\tau,k}$
\begin{equation}\nonumber
\big(\AA- i \lambda^{\tau,k}\big)^{-1}:
\mathrm{Ker}\big(\AA- i \lambda^{\tau,k}\big)^\perp
\longrightarrow \mathrm{Ker}\big(\AA- i
\lambda^{\tau,k}\big)^\perp\,,
\end{equation}
by
\begin{equation}\label{pseudo-1}
\big(\AA- i \lambda^{\tau,k}\big)^{-1}
U^{\delta,l}=\tfrac{1}{i\lambda^{\delta,l}- i \lambda^{\tau,
k}}U^{\delta,l}\,,\quad\!\mbox{for any}\quad\!
U^{\delta,l}\quad\!\mbox{with}\quad\! \lambda^l \neq \lambda^k\,,
\end{equation}
\begin{equation}\label{pseudo-2}
\big(\AA- i \lambda^{\tau,k}\big)^{-1}
U^{-\tau,l}=\tfrac{1}{-2i\lambda^{\tau,k}}U^{-\tau,l}\,,\quad\!\mbox{for
any}\quad\! U^{-\tau,l}\quad\!\mbox{with}\quad\! \lambda^l =
\lambda^k\,,
\end{equation}
and
\begin{equation}\label{pseudo-3}
\big(\AA- i \lambda^{\tau,k}\big)^{-1} (\rho, \mathrm{v}, -\rho)^T
= \tfrac{1}{i \lambda^{\tau,k}} (\rho, \mathrm{v}, -\rho)^T\,,
\end{equation}
for any $(\rho, \mathrm{v}, -\rho)^T\in \mbox{Null}(\AA)$ and
$\tau,\delta \in \{ +,-\}$. It is obvious that this pseudo inverse operator
is a bounded operator. Consequently, the solutions to the system \eqref{A-Lambda} are stated in the following lemma.

\begin{Lemma}\label{Solve-A}
For each fixed acoustic modes $k \geq 1$ and $\tau \in \{+,-\}$, the solvability conditions of the system \eqref{A-Lambda} are:

$(i)$ If $\lambda^{k}$ is a simple eigenvalue of \eqref{laplac-neu}, then the only solvability condition is that $i \mu^{\tau,k}$ must satisfy
\begin{equation}\label{imu}
i \mu^{\tau,k} = \int_{\pO}g^{\tau,k}\Psi^k
\dd \sigma_{\! x} - \< F^{\tau,k} |
U^{\tau,k}\>\,.
\end{equation}
Under this condition, the solutions to \eqref{A-Lambda} $V^{\tau,k}$ can be solved uniquely as
\begin{equation}\label{VK-1}
V^{\tau,k} = V^{\tau,k}_1\,,
\end{equation}
where $V^{\tau,k}_1 \in \mathrm{Ker}\big(\AA- i\lambda^{\tau,k}\big)^\perp$.

$(ii)$ If $\lambda^{k}$ is not a simple eigenvalue, then besides \eqref{imu}, further compatibility condition is needed: $F^{\tau,k}$ must satisfy:
\begin{equation}\label{compatability-KL}
\int_{\pO}g^{\tau,k}\Psi^l \dd \sigma_{\!x} = \< F^{\tau,k} |
U^{\tau,l}\>\,,\quad \mbox{for}\quad\! \lambda^l= \lambda^k \quad
\mbox{with}\quad\! k \neq l\,.
\end{equation}
For this case,  under these two conditions \eqref{imu}-\eqref{compatability-KL}, the solutions to \eqref{A-Lambda}
$V^{\tau,k}$ can be determined modulo $\mathrm{Ker}\big(\AA- i
\lambda^{\tau,k}\big)$. In other words, $V^{\tau,k}$ can be uniquely
represented as
\begin{equation}\label{VK}
V^{\tau,k}  = \sum\limits_{\lambda^k
= \lambda^l} \< V^{\tau,k} \,|\, U^{\tau,l}\> U^{\tau,l} + V^{\tau,k}_1\,,
\end{equation}
where $V^{\tau,k}_1 \in \mathrm{Ker}\big(\AA- i
\lambda^{\tau,k}\big)^\perp$.

\end{Lemma}

\begin{proof}

For any $g^{\tau,k} \in H^{\frac{1}{2}}(\pO)$, there exists
$\tilde{\mathrm{v}}^{\tau,k}\in H^1(\Omega\,;\mathbb{R}^\D) $, such that
$\gamma \tilde{\mathrm{v}}^{\tau,k}\!\cdot\!\mathrm{n} =
g^{\tau,k}$, where $\gamma$ is the usual trace operator from
$H^1(\Omega\,;\mathbb{R}^\D)$ to $H^{\frac{1}{2}}(\pO)$. We define
\begin{equation}\label{V-tilde}
\widetilde{V}^{\tau,k}= V^{\tau,k}- (0,
\tilde{\mathrm{v}}^{\tau,k}, 0)^T\,.
\end{equation}
Then $\widetilde{V}^{\tau,k}$ has zero the normal velocity on the
boundary $\pO$, thus is in the domain of $\AA$. From
\eqref{A-Lambda}, $\widetilde{V}^{\tau,k}$ satisfies
\begin{equation}\label{A-V}
(\AA- i \lambda^{\tau,k}) \widetilde{V}^{\tau,k} = - (\AA- i
\lambda^{\tau,k})(0, \tilde{\mathrm{v}}^{\tau,k}, 0)+ i
\mu^{\tau,k} U^{\tau,k} + F^{\tau,k}\,.
\end{equation}

The solvability of \eqref{A-V} is that the righthand side must be in
$\mathrm{Ker}\big(\AA- i \lambda^{\tau,k}\big)^\perp$. Thus, the inner
product of \eqref{A-V} with $U^{\tau,k}$ is zero, which gives \eqref{imu}, while the inner product
with $U^{\tau,l}$ with $\lambda^k = \lambda^l, k\neq l$ gives
\eqref{compatability-KL}. Under these conditions, by applying the
psedu inverse operator $\big(\AA- i \lambda^{\tau,k}\big)^{-1}$ defined in \eqref{pseudo-1}-\eqref{pseudo-3},
we can uniquely solve $\widetilde{V}^{\tau,k}$ in
$\mathrm{Ker}\big(\AA- i \lambda^{\tau,k}\big)^\perp$, denoted by
$\widetilde{V}^{\tau,k}_1$. However, the projection of $\widetilde{V}^{\tau,k}$ on
$\mathrm{Ker}\big(\AA- i \lambda^{\tau,k}\big)$ is {\em not} determined. In other words,
\begin{equation}\nonumber
\widetilde{V}^{\tau,k}= \widetilde{V}^{\tau,k}_1 + \sum\limits_{\lambda^k
= \lambda^l} \< \widetilde{V}^{\tau,k} \,|\, U^{\tau,l}\> U^{\tau,l}\,.
\end{equation}

Using \eqref{V-tilde}, we get \eqref{VK},
where
\begin{equation}\nonumber
V^{\tau,k}_1 = \widetilde{V}^{\tau,k}_1 + (0, \tilde{\mathrm{v}}^{\tau,k},
0)^T -  \sum\limits_{\lambda^k = \lambda^l} \< (0,
\tilde{\mathrm{v}}^{\tau,k}, 0)^T \,|\, U^{\tau,l}\> U^{\tau,l}\,.
\end{equation}

In \eqref{VK}, the projection of $V^{\tau,k}$ on
$\mathrm{Ker}\big(\AA- i \lambda^{\tau,k}\big)$, i.e. the first
term in the righthand side of \eqref{VK}, can {\em not} be
determined. It is easy to see that the projection of $V^{\tau,k}$ on
$\mathrm{Ker}\big(\AA- i \lambda^{\tau,k}\big)^\perp$, i.e. $V^{\tau,k}_1$, is uniquely determined,
although the lifting of the trace $g^{\tau,k}$ is not unique.

\end{proof}

\section{Analysis of the Kinetic Boundary Layer Equation}
In this section, we collect three results in kinetic equations which
will be frequently used in this paper. The first two results are
standard in kinetic theory:
\begin{Lemma}\label{Kinetic-Sol}
The solvability condition for the linear kinetic equation $\LL g=f$
is
\begin{equation}\label{sol-condition}
\< f\,,\zeta(v)\>=0\,,\quad\mbox{for}\quad
\zeta\in\mbox{Span}\{1\,,v\,,|v|^2\}\,.
\end{equation}
\end{Lemma}
The second result we will use is quoted from Lemma 4.4 in \cite{BGL2}.
\begin{Lemma}\label{mu-kappa-identity}
The components of $\< \mathrm{A}\otimes \widehat{\mathrm{A}}\> $ and
$\< \mathrm{B}\cdot\widehat{\mathrm{B}}\>$ satisfy the following
identities:
\begin{equation}\nonumber
\begin{aligned}
\< \mathrm{A}_{ij}\otimes \widehat{\mathrm{A}}_{kl}\>&=\nu  \left(
\delta_{ik}\delta_{jl}+\delta_{il}\delta_{jk}
-\tfrac{2}{\D}\delta_{ij}\delta_{kl}\right)\,,\\
\< \mathrm{B}_i\widehat{\mathrm{B}}_j\> &= \tfrac{\D+2}{2}\kappa
\delta_{ij}\,,
\end{aligned}
\end{equation}
where $\nu $ and $\kappa$ are given by \eqref{mu-kappa}.
\end{Lemma}

The next result is about the linear kinetic boundary layer equation which
will be used to determine the boundary conditions of the fluid
variables. We define the kinetic boundary layer operator $\LL^{BL}$, reflection boundary operator $L^{\mathcal{R}}$ and diffusive boundary operator $L^{\mathcal{D}}$ acting on functions $\{g^\bb(x,v, \xi): (x,v,\xi)\in \Omega^\delta \times \RD \times \mathbb{R}_+\}$ as follows:
\begin{equation}\label{L-BL}
 \LL^{BL}g^\bb := -(v\!\cdot\!\grad\mathrm{d}) \p_{\xi}g^\bb + \LL g^\bb\,,
\end{equation}
where $\LL$ is the linearized Boltzmann operator defined in \eqref{linearized B}.
\begin{equation}\nonumber
 L^{\mathcal{R}} g^\bb := \gamma_+ g^\bb - L \gamma_- g^\bb\,,\quad\mbox{and}\quad\!   L^{\mathcal{D}}g^\bb := \sqrt{2\pi}\chi\left[ \<\gamma_- g^\bb \>_{\p\Omega} - L \gamma_- g^\bb \right]\,.
\end{equation}

\begin{Lemma}\label{BC}
Considering the following linear kinetic boundary layer equation of $g^\bb(x,v,\xi)$ in
half space:
\begin{equation}\label{B-L-Equation}
\begin{aligned}
 \LL^{BL}g^\bb & = S^{\bb}\,, \quad \mbox{in}\quad\xi >0\,,\\
g^\bb &\longrightarrow 0\,,\quad\mbox{as}\quad \xi\rightarrow
\infty\,,
\end{aligned}
\end{equation}
with boundary condition
\begin{equation}\label{BC-beta>0}
 L^\mathcal{R}g^\bb = H^\bb\,,\quad \mbox{on}\quad\xi
= 0\,,\quad v\cdot \mathrm{n}>0\,.
\end{equation}
In the above equations, the boundary source term $H^\bb$ is taken of the following form:
\begin{equation}\label{h}
H^\bb=-L^{\mathcal{R}} g + L^{\mathcal{D}} f\,,
\end{equation}
where $ g$ and $f$ are of the forms:
\begin{equation}\label{g}
\begin{aligned}
g=&\rho_g + \mathrm{u}_g\cdot v + \theta_g \left( \tfrac{|v|^2}{2}-\tfrac{\D}{2}\right) \\
-& (\p_\zeta \mathrm{u}^\b\!\otimes\! \mathrm{n} :\widehat{\mathrm{A}}+
\p_\zeta \theta^{\b} \mathrm{n}\!\cdot\! \widehat{\mathrm{B}})+(\p_{\pi^\alpha} \tilde{\mathrm{u}}^\b \!\otimes\! \grad\pi^\alpha :\widehat{\mathrm{A}}+
\p_{\pi^\alpha} \tilde{\theta}^{\b} \grad\pi^\alpha\!\cdot\! \widehat{\mathrm{B}})\\
+&(\grad
\mathrm{u}^\Int :\widehat{\mathrm{A}} + \grad \theta^\Int \cdot
\widehat{\mathrm{B}}) +S_g\,,
\end{aligned}
\end{equation}
and
\begin{equation}\label{f}
f=\rho_f + \mathrm{u}_f\cdot v + \theta_f \left(
\tfrac{|v|^2}{2}-\tfrac{\D}{2}\right)  +S_f\,,
\end{equation}
and where $S_g\,, S_f\in\mbox{Null}(\LL)^\perp$ are source terms.

Then there exists a solution $g^\bb(x,v,\xi)$ of the equation \eqref{B-L-Equation} if and only if the following boundary
conditions are satisfied by the fluid variables:

$(i)$ On the boundary $\pO$, the normal components of velocity is
       \begin{equation}\label{BC-N}
               \mathrm{u}_g\cdot \mathrm{n} = \int^\infty_0\< S^{\bb}\>\, \dd \xi\,.
        \end{equation}

$(ii)$ On the boundary $\pO$, the tangential components of
velocities and temperature satisfy
\begin{equation} \label{BC-D}
\begin{aligned}
\left [\mathrm{u}_f\right ]^{\mathrm{tan}}  = &
\tfrac{\nu}{\chi}  \left [\p_{\zeta}\mathrm{u}^\b \right ]^{\mathrm{tan}} - \tfrac{\nu}{\chi}  \left[2
d(\mathrm{u}^\Int)\cdot \mathrm{n}\right]^{\mathrm{tan}}- \tfrac{\nu}{\chi}\nabla_{\!\pi}[\tilde{\mathrm{u}}^\b\!\cdot\!\mathrm{n}]\\
+&  \left[ \int_{v\cdot \mathrm{n}>0}(L^{\mathcal{D}}S_f) v(v\!\cdot\! \mathrm{n}) M\,\dd v \right]^{\mathrm{tan}}
 - \tfrac{1}{\chi} \< (v\!\cdot\! \mathrm{n}) v S_g \> ^{\mathrm{tan}}+ \tfrac{1}{\chi}\int^\infty_0\< S^{\bb} v\>^{\mathrm{tan}} \, \dd \xi\,,
\end{aligned}
\end{equation}
and
\begin{equation} \label{BC-theta}
\begin{aligned}
\theta_f   = &\tfrac{\D+2}{\D+1}\tfrac{\kappa}{\chi} \p_\zeta \theta^{\b} -\tfrac{\D+2}{\D+1}\tfrac{\kappa}{\chi} \p_\mathrm{n}\theta^\Int  + \tfrac{\sqrt{2\pi}}{2(\D+1)} \mathrm{u}_f\cdot \mathrm{n}
+\tfrac{\sqrt{2\pi}}{\D+1}\int_{v\cdot \mathrm{n} >0}(L^{\mathcal{D}}S_f)
|v|^2 (v\!\cdot\! \mathrm{n})M\,\dd v\\
& - \tfrac{1}{(\D+1)\chi}\<(v\!\cdot\!\mathrm{n})|v|^2 S_g \> + \tfrac{\D+2}{\D+1}\tfrac{1}{\chi}\int^\infty_0\< S^{\bb} (\tfrac{|v|^2}{\D+2}-1)\>\, \dd \xi\,,
\end{aligned}
\end{equation}
where kinematic viscosity $\nu $ and thermal conductivity $\kappa$
are given by \eqref{mu-kappa}, $\mathrm{u}^\mathrm{tan}$ denotes the
tangential components of the vector $\mathrm{u}$, and $\nabla_\pi$ denotes the tangential derivative.
\end{Lemma}

\begin{proof}

The solvability conditions of the linear boundary layer equation
\eqref{B-L-Equation} with boundary condition \eqref{BC-beta>0} are
given by
\begin{equation}\label{BL-SOL}
\int_{v\cdot \mathrm{n} >0} H^\bb\eta(v)(v\cdot \mathrm{n}) M\, \dd
v=-\int^\infty_0\< S^{\bb} \eta\>\, \dd \xi\,,
\end{equation}
for all $\eta(v)\in$Null$(\LL)$ satisfying the condition:
\begin{equation}\label{condition-v}
\eta(R_x v)=\eta(v)\,.
\end{equation}

It is obvious that $1$ and $|v|^2$ satisfy \eqref{condition-v}. If $\eta(v) = \sum^\D_1 a_i v_i$ satisfies \eqref{condition-v}, then necessarily
\begin{equation}\label{perpendicular}
(v\cdot \mathrm{n})\sum^\D_{i=1} a_i \mathrm{n}_i = 0\,,
\end{equation}
which implies that the vector $\mathrm{a} = (a_1\,,\cdots\,,
a_{\D})^\top$ is perpendicular to the outer normal vector $\mathrm{n}$.

The formula \eqref{BC-N} can be derived by taking $\eta=1$ in
\eqref{BL-SOL}. Simple calculations show that
\begin{equation}\nonumber
\begin{aligned}
\int_{v\cdot \mathrm{n} >0} H^\bb(v) (v\cdot \mathrm{n}) M\,\dd v&=-\< v \gamma g\>\cdot \mathrm{n}\\
&=-\mathrm{u}_g\cdot \mathrm{n}-\< v S_g\> \cdot \mathrm{n}\,.
\end{aligned}
\end{equation}
Note that $S_g \in$Null$(\LL)^\perp$, hence \eqref{BC-N} follows.

To prove \eqref{BC-D}, by taking $\eta= \sum ^\D_1 a_i v_i$ in
\eqref{BL-SOL}, we have
\begin{equation}\nonumber
\int_{v\cdot \mathrm{n} >0} H^\bb (a_i v_i)(\mathrm{n}_j v_j)  M\, \dd
v=-\int^\infty_0\< (a_i v_i) S^{\bb} \>\, \dd \xi\,.
\end{equation}
In other words,
\begin{equation}\nonumber
\begin{aligned}
&-\int_{v \cdot \mathrm{n}>0}(L^\mathcal{R}g) (a_i v_i)(\mathrm{n}_j v_j)M\,\dd v+\int_{v \cdot \mathrm{n}>0}(L^\mathcal{D}f) (a_i v_i)(\mathrm{n}_j v_j)M\,\dd v\\
=& -\int^\infty_0\< S^{\bb} (a_i v_i)\>\, \dd \xi\,.
\end{aligned}
\end{equation}
Simple calculations yield that
\begin{equation}\nonumber
\int_{v \cdot \mathrm{n}>0}(L^\mathcal{R}g) (a_i v_i)(\mathrm{n}_j v_j)M\,\dd v = \<  v_i v_j \gamma g\>
a_i \mathrm{n}_j = \<  \mathrm{A}_{ij} \gamma g\>
a_i \mathrm{n}_j\,.
\end{equation}
Using the definition of the viscosity $\nu $ in \eqref{mu-kappa} and
Lemma \ref{mu-kappa-identity}, we have
\begin{equation}\nonumber
\begin{aligned}
& \int_{v \cdot \mathrm{n}>0}(L^\mathcal{R}g) (a_i v_i)(\mathrm{n}_j v_j)M\,\dd v \\
 = & -\nu [\p_\zeta \mathrm{u}^\b]\cdot \mathrm{a} +\nu
\left[(\grad \mathrm{u}^\Int +(\grad \mathrm{u}^\Int)^\top) \cdot
\mathrm{n}\right] \cdot \mathrm{a}+ \nu\nabla_{\pi}[\tilde{\mathrm{u}}\cdot\mathrm{n}]\cdot\mathrm{a} +\< S_g \mathrm{A} \mathrm{n}\>
\cdot \mathrm{a}\,.
\end{aligned}
\end{equation}
Next, it can be calculated that
\begin{equation}\label{R-f}
L^\mathcal{D} f  =
\big(\tfrac{\D+1}{2}\theta_f-\tfrac{\sqrt{2\pi}}{2}\mathrm{u}_f\cdot
\mathrm{n}\big) - \mathrm{u}_f\!\cdot\! v + 2(\mathrm{u}_f\!\cdot\!\mathrm{n})\mathrm{n}\!\cdot\! v -\tfrac{1}{2}\theta_f|v|^2\\
  + L^\mathcal{D}S_f\,,
\end{equation}
where $R_x \mathrm{u}_f=\mathrm{u}_f-2(\mathrm{u}_f\cdot
\mathrm{n})\mathrm{n}$ is the reflection of $\mathrm{u}_f$ with
respect to the normal $\mathrm{n}(x)$. Thus
\begin{equation}\nonumber
\int_{v\cdot \mathrm{n} >0} (L^\mathcal{D} f)  (a_i v_i)(v\cdot\mathrm{n})  M\,\dd v
=- \tfrac{1}{\sqrt{2\pi}}(R_x \mathrm{u}_f)\cdot \mathrm{a} +
\left[\int_{v \cdot \mathrm{n} >0} (L^\mathcal{D} S_f)  \mathrm{A}\cdot \mathrm{n} M\,\dd v \right] \cdot
\mathrm{a}\,.
\end{equation}

To prove \eqref{BC-theta}, taking $\eta= |v|^2$ in
\eqref{BL-SOL}, we have
\begin{equation}\nonumber
\int_{v\cdot \mathrm{n} >0} H^\bb |v|^2(v\cdot \mathrm{n}) M\, \dd
v=-\int^\infty_0\<|v|^2 S^{\bb} \>\, \dd \xi\,.
\end{equation}
It can be calculated that
\begin{equation}\nonumber
\int_{v\cdot \mathrm{n} >0}(L^\mathcal{R}g) (v\cdot \mathrm{n})|v|^2M\,\dd v=2 \< \gamma g \mathrm{B}\>
\cdot \mathrm{n} +(\D+2) \< (v\cdot \mathrm{n}) \gamma g\>\,.
\end{equation}
Using the definition of the thermal conductivity $\kappa$ in
\eqref{mu-kappa} and Lemma \ref{mu-kappa-identity}, we have
\begin{equation}\nonumber
\begin{aligned}
&\int_{v \cdot \mathrm{n}>0}L^\mathcal{R}g (v\cdot \mathrm{n})|v|^2M\,\dd v\\ =&-(\D+2)\kappa \p_\zeta
\theta^b +(\D+2)\kappa \grad\theta^\Int \cdot \mathrm{n} +(\D+2)
\mathrm{u}_g\cdot \mathrm{n}+\< (v\cdot \mathrm{n}) |v|^2 S_g \>\,.
\end{aligned}
\end{equation}
Furthermore,
\begin{equation}\nonumber
\begin{aligned}
& \int_{v \cdot \mathrm{n}>0} (L^\mathcal{D} f)  (v\!\cdot\!\mathrm{n}) |v|^2 M\, \dd v\\
 = & \tfrac{1}{2}\mathrm{u}_f\cdot \mathrm{n} -
\tfrac{\D+1}{\sqrt{2\pi}}\theta_f + \int_{v\cdot \mathrm{n}>0} (L^\mathcal{D} S_f)  (v \cdot \mathrm{n})
|v|^2 M\,\dd v\,.
\end{aligned}
\end{equation}

As showed in \eqref{perpendicular}, the vector $\mathrm{a}$ is
perpendicular to $\mathrm{n}$. Thus the resulting vector of inner
product with $\mathrm{a}$ is the tangential part. Finally, noticing
$[R_x \mathrm{u}_f]^{\mathrm{tan}}=[\mathrm{u}_f]^{\mathrm{tan}}$,
we then finish the proof of the Lemma \ref{BC}.
\end{proof}

\section{Approximate Eigenfunctions-Eigenvalues}

\subsection{Motivation} We define the operators $\LL_\eps$ and $\LL^*_\eps$ as
\begin{equation}\nonumber
 \LL_\eps := \frac{1}{\eps}\LL -v\!\cdot\!\grad\,, \qquad \LL^*_\eps := \frac{1}{\eps}\LL +v\!\cdot\!\grad\,.
\end{equation}
Formally, $\LL_\eps$ and $\LL^*_\eps$ are ``dual'' in the following sense:
\begin{equation}\label{formal-dual}
\<\LL^*_\eps g^*\,, g \> = \< g^* \,, \LL_\eps g \>\,,
\end{equation}
provided that $g^*$ satisfies the Maxwell reflection boundary condition
\begin{equation}\label{Maxwell-BC}
\gamma_{-}g^*=(1-\alpha)L \gamma_{+}g^*+\alpha \langle\gamma_+g^*\rangle_\pO\quad \mbox{on}\quad \Sigma_{-}\,,
\end{equation}
and $g$ satisfies the dual boundary condition
\begin{equation}\label{dual-BC}
\gamma_{+}g=(1-\alpha)L \gamma_{-}g +\alpha \langle\gamma_-g\rangle_\pO\quad \mbox{on}\quad \Sigma_{+}\,.
\end{equation}

If $g_\eps$ is the fluctuation defined in \eqref{define-g}, then $g_\eps$ obeys the scaled Boltzmann equation \eqref{BE-g} in which $\LL^*_\eps g_\eps$ appears and $g_\eps$ satisfies the boundary condition \eqref{Maxwell-BC}. Then from \eqref{formal-dual}, $\LL_\eps g^{BL}_\eps$ appears in the weak formulation of the Boltzmann equation if we take $g^{BL}_\eps$ as a test function. Thus, it is natural to construct eigenfunctions and eigenvalues of $\LL_\eps$ satisfying the dual boundary condition \eqref{dual-BC}. Specifically, we consider the kinetic eigenvalue problem:
\begin{equation}\label{eigen-Lep}
\mathcal{L}_\eps g^{BL}_\eps=-i\lambda^{BL}_{\eps} g^{BL}_\eps\,,
\end{equation}
with $g^{BL}_\eps$ satisfying the dual Maxwell boundary condition \eqref{dual-BC},
where the accommodation coefficient $\alpha$ takes the value $\ale = \sqrt{2\pi}\chi \sqrt{\eps}$. By doing so, formally the equation \eqref{BE-g} becomes an ordinary differential equation of $b_{\eps}=\int_{\Omega}\< g_{\eps}\,, g^{BL}_{\eps}\>\,\mathrm{d}x$:
\begin{equation}\nonumber
  \tfrac{\mathrm{d}}{\mathrm{d}t}b_{\eps} + \tfrac{i\lambda^{BL}_{\eps}}{\eps}b_{\eps} = c_{\eps}\,.
\end{equation}

To solve the eigenvalue problem \eqref{eigen-Lep} and \eqref{dual-BC}, a key observation is that the solutions must include interior and two boundary layer terms: the {\em fluid viscous boundary layer} with thickness $\sqrt{\eps}$, and the {\em kinetic Knudsen layer} with thickness $\eps$. We make the ansatz of $g^{BL}_\eps$ and $\lambda^{BL}_{\eps}$  as
\begin{equation}\label{1-ansatz-half}
g^{BL}_{\eps}=\sum\limits_{m\geq 0}\left[
g^{\Int}_{m}(x,v) +
g^{\b}_{m}(\pi(x),\tfrac{\dd(x)}{\sqrt{\eps}},v)\right] \eps^\frac{m}{2} + \sum\limits_{m\geq 1}g^{\bb}_{m}(\pi(x),\tfrac{\dd(x)}{\eps},v)\eps^\frac{m}{2} \,,
\end{equation}
and
\begin{equation}\label{1-lambda-half}
\lambda^{BL}_{\eps}=\sum\limits_{m \geq
0}\lambda_{m}\eps^
{\frac{m}{2}}\,.
\end{equation}

Each $g^{\b}_{m}$ and $g^{\bb}_{m}$ are defined in $\Omega^\delta$, the
$\delta\mbox{-}$tubular neighborhood of $\pO$ in $\Omega$, where
$\delta>0$ is the small number defined in Lemma \ref{Projection},
the projection $\pi$ is defined in \eqref{Pai}. After rescaling by
$\sqrt{\eps}$ and $\eps$ respectively,
\begin{equation}\nonumber
g^{\b}_{m}, g^{\bb}_{m}: (\pO \times \mathbb{R}^+) \times \RD
\longrightarrow \mathbb{R}\,.
\end{equation}
Both $g^{\b}_{m}$ and $g^{\bb}_{m}$ will vanish in the
outside of $\Omega^\delta$. Thus $g^{\b}_{m}$ and $g^{\bb}_{m}$
are required to be rapidly decreasing to 0 in the $\zeta$ and $\xi$
respectively, which are defined by $\zeta=\frac{\mathrm{d}(x)}{\sqrt{\eps}}$
and $\xi=\frac{\mathrm{d}(x)}{\eps}$.

In the ansatz \eqref{1-ansatz-half}, $g^{BL}_\eps$ consists three types of terms: the interior terms $g^{\Int}_{m}$,
the fluid viscous boundary layer terms $g^{\b}_{m}$, and the kinetic Knudsen layer terms
$g^{\bb}_{m}$. They are coupled through the boundary condition \eqref{dual-BC}.

\subsection{Statement of the Proposition}
Now we state the proposition which can be considered as a kinetic analogue of the
Proposition 2 in \cite{DGLM}.
\begin{Prop}\label{main-prop}
Let $\Omega$ be a  $C^2$ bounded domain of $\RD$ and the accommodation coefficient $\alpha_\eps= \sqrt{2\pi}\chi\sqrt{\eps}$. Then, for every acoustic mode $k
\geq 1$, non-negative integer $N$, and each $\tau \in \{+\,, -\}$, there exists approximate
eigenfunctions $g^{\tau,k}_{\eps,N}$ and eigenvalues
$-i\lambda^{\tau,k}_{\eps,N}$ of $\LL_\eps$ , and error terms
$R^{\tau,k}_{\eps,N}$ and $r^{\tau,k}_{\eps,N}$ respectively,
such that
\begin{equation}\label{eigen-Lep-N}
\LL_\eps
g^{\tau,k}_{\eps,N}=-i\lambda^{\tau,k}_{\eps,N}g^{\tau,k}_{\eps,N}+R^{\tau,k}_{\eps,N}\,,
\end{equation}
and $g^{\tau,k}_{\eps,N}$ satisfy the approximate dual Maxwell boundary condition:
\begin{equation}\label{bc-ki}
L^\mathcal{R} g^{\tau,k}_{\eps,N}=\sqrt{\eps} L^\mathcal{D} g^{\tau,k}_{\eps,N}+r^{\tau,k}_{\eps,N}\quad
\mbox{on}\quad\! \Sigma_{+}\,.
\end{equation}
Moreover, there exits complex numbers $\lambda^{\tau,k}_1$,  such that $i\lambda^{\tau,k}_{\eps,N}$ has the following expansions:
\begin{equation}\label{lambda-1}
i\lambda^{\tau,k}_{\eps,N}=
i\lambda^{\tau,k}_0+i\lambda^{\tau,k}_1\sqrt{\eps}+O(\eps)\,,\quad
 with\quad \mathrm{Re}(i\lambda^{\tau,k}_1)< 0\,.
\end{equation}
Furthermore, for all $1<r, p\leq\infty$, we have error estimates:
\begin{equation}\label{error1}
\|R^{\tau,k}_{\eps,N}\|_{L^r(\mathrm{d}x,L^p(a^{1-p}M\mathrm{d}v))}=O(\sqrt{\eps}^{N-1})\,,
\end{equation}
 and
\begin{equation}\label{error2}
\|g^{\tau,k}_{\eps,N}-g^{\tau,k,\Int}_{0}\|_{L^r(\mathrm{d}x,L^p(a^{1-p}M\mathrm{d}v))}=O
(\eps^{\frac{1}{2r}})\,.
\end{equation}
where $g^{\tau,k,\Int}_{0}$ is defined in
\eqref{acoustic-modes}. We also have the boundary error estimates:
\begin{equation}\label{r-error}
\|r^{\tau,k}_{\eps,N}\|_{L^r(\mathrm{d}\sigma_x,L^p(a^{1-p}M\mathrm{d}v))}
=O\left(\sqrt{\eps}^{N+1}\right)\,.
\end{equation}
\end{Prop}

\subsection{Main Idea of the Proof}

For each non-negative integer $m$, $g^\Int_m$ is decomposed as hydrodynamic part $\PP g^\Int_m$, i.e. the projection on Null$(\LL)$, and kinetic part $\PP^{\!\perp} g^\Int_m$, the projection on Null$(\LL)^{\!\perp}$. The hydrodynamic part is given by
\begin{equation}\nonumber
\PP g^\Int_m=
\rho^{\Int}_m+v\!\cdot\!\mathrm{u}^{\Int}_m+\big(\cvd\big)\theta^\Int_m = (1,v,\tfrac{|v|^{2}}{2}-\tfrac{\D}{2})U^{\Int}_{m}\,,
\end{equation}
where $U^{\Int}_m=
(\rho^{\Int}_m,\mathrm{u}^{\Int}_m,\theta^\Int_m
)^\top$ is called the fluid variables of $g^\Int_m$. It can be shown that the coefficients of $\PP^{\!\perp} g^\Int_m$ are in terms of  $U^{\Int}_{m'}$ and their derivatives for $m'<m$. Thus, we need only to solve $U^\Int_n$ for all integers $n\leq m$ to determine $g^\Int_m$. Similar notations will be used for $g^{\b}_m$, and for the same reason we also need only to solve $U^\b_m$.

We put the ansatz  into the equation \eqref{eigen-Lep}, then collect the same order terms. The leading order term $g^{\Int}_0$ is hydrodynamic, which means $g^\Int_0$ is completely determined by $U^\Int_0$. we can derive that $U^\Int_0$ satisfies the equation
\begin{equation}\label{eigen-00}
\AA U^\Int_0 = i \lambda_0 U^\Int_0\,.
\end{equation}
For \eqref{eigen-00} there are two cases:
\begin{itemize}
\item  {\bf Case 1}: $\lambda_0\neq 0$, by comparing the
equations \eqref{eigen-00} and \eqref{acoustic-equ},
$ i\lambda_0$ is an eigenvalue of the acoustic operator
$\AA$, i.e. $\lambda_0= \lambda^{\tau,k}$ and
$U^\Int_0= U^{\tau,k}$, where $k \geq 0$ is the acoustic modes, and $\tau$ denotes either $+$ or $-$. Starting from here, we can construct the boundary layer $g^{BL}_\eps$ which we call the boundary layer in the {\bf acoustic modes}.
\item {\bf Case 2}: $\lambda_0=0$,
which implies that $U^\Int_0\in\mbox{Ker}(\AA)$, i.e. $\rho^\Int_0+ \theta^\Int_0=0$ and $\DIV \mathrm{u}^\Int_0=0$. Starting from here, we can construct the boundary layer $g^{BL}_\eps$ which we call the boundary layer in the {\bf incompressible modes}.
\end{itemize}
Because the main goal of the current paper is about how the acoustic waves and the boundary layers interact in the incompressible Navier-Stokes limit of the Boltzmann equation, we only consider the {\bf Case 1},
the kinetic-fluid boundary layers in the acoustic modes for
each $k\geq 0$ and each $\tau$. Consequently, we add superscript $\tau,k$ for each term in the ansatz.
In the forthcoming paper \cite{J-M-2}, we will investigate the higher order acoustic limit of the Boltzmann equation, where we need to analyze the boundary layers in incompressible modes, i.e. {\bf Case 2}.

The basic strategy to solve all terms in the ansatz  is the following: (for the simplicity of notation, we don't write the upper index $\tau$)

\noindent {\bf 1,} $g^{k,\bb}_m$ satisfies the linear kinetic boundary layer equation \eqref{B-L-Equation}. Applying Lemma \ref{BC}, the solvability conditions for $g^{k,\bb}_m$ give the normal boundary condition $[\mathrm{u}^{k,\Int}_m+ \mathrm{u}^{k,\b}_m]\cdot\mathrm{n}$ and the tangential boundary condition $[\mathrm{u}^{k,\b}_{m-1}-\tfrac{\nu}{\chi}\partial_{\!\zeta}\mathrm{u}^{k,\b}_{m-1}]^{\mathrm{tan}} +[\mathrm{u}^{k,\Int}_{m-1}]^{\mathrm{tan}}$ and $\theta^{k,\b}_{m-1}-\tfrac{\D+2}{\D+1}\tfrac{\kappa}{\chi}\partial_{\!\zeta}\theta^{k,\b}_{m-1}+ \theta^{k,\Int}_{m-1}\,;$

\noindent {\bf 2,} $U^{k,\b}_m$ satisfies the ODE system like \eqref{U-b-m-1}, where the normal boundary acoustic operator $\AA^{\mathrm{d}}$ is defined in \eqref{A-d}. Solving $U^{k,\b}_m$ includes two steps: First, projecting the ODE system of $U^{k,\b}_m$ on Null$^\perp(\AA^{\mathrm{d}})$ to get the first order ODE satisfied by $\rho^{k,\b}_m + \theta^{k,\b}_m$ and $\mathrm{u}^{k,\b}_m\cdot\grad\mathrm{d}$ (thus only need one boundary condition at $\zeta=\infty$). The next, projecting on Null$(\AA^{\mathrm{d}})$, to solve $\mathrm{u}^{k,\b}_{m-1}\cdot\grad\pi$ and $\theta^{k,\b}_{m-1}$ which satisfy second ODE. Thus besides the condition at $\zeta=\infty$, another boundary condition at $\zeta=0$ is needed. This is given by the solvability condition of $g^{k,\bb}_{m}$ in {\bf 1}.

\noindent {\bf 3,} $U^{k,\Int}_m$ and $i \lambda^{k}_m$ can be solved by applying Lemma \ref{Solve-A}, where only the boundary condition in the normal
direction $\mathrm{u}^{k,\Int}_m\!\cdot\!\mathrm{n}$ is needed which can be known from $\mathrm{u}^{k,\b}_m\!\cdot\!\mathrm{n}$ since their summation $[\mathrm{u}^{k,\Int}_m+ \mathrm{u}^{k,\b}_m]\cdot\mathrm{n}$ is found in $(1)$, and $\mathrm{u}^{k,\b}_m\!\cdot\!\mathrm{n}$ is already known in {\bf 2};

\noindent {\bf 4,} If the multiplicity of the eigenvalue $\lambda^{k}_0$ is greater than 1, applying Lemma \ref{Solve-A}, $U^{k,\Int}_m$ can be only solved modulo $\mathrm{Ker}(\AA - i\lambda^{k}_0)$. The components of $U^{k,\Int}_m$ in $\mathrm{Ker}(\AA - i\lambda^{k}_0)$ will be solved by applying Lemma \eqref{Solve-A} again in later rounds.

\section{Proof of Proposition \ref{main-prop}: Construction of Boundary Layers}

In this section, we construct the kinetic-fluid boundary layers corresponding to the accommodation coefficient $\ale = \sqrt{2\pi}\chi \sqrt{\eps}$. As mentioned in the previous section, we make the ansatz for $g^{BL}_{\eps}$ and $\lambda^{BL}_{\eps}$ as in \eqref{1-ansatz-half} and \eqref{1-lambda-half}. Then we formally
plug the ansatz \eqref{1-ansatz-half} into  the equation \eqref{eigen-Lep} and then collect the terms with the same order of $\eps$ in the interior, the viscous boundary layer, and the  Knudsen layer respectively. The following calculations will be frequently used:
\begin{equation}\nonumber
\begin{aligned}
v\!\cdot\!\grad g^\b(\pi(x), \tfrac{\mathrm{d}(x)}{\sqrt{\eps}})&= (v\!\cdot\!\grad \pi^\alpha)\p_{\pi^\alpha}g^\b + \tfrac{1}{\sqrt{\eps}} (v\!\cdot\!\grad \mathrm{d})\p_\zeta g^\b\,,\\
v\!\cdot\!\grad g^\bb(\pi(x), \tfrac{\mathrm{d}(x)}{\eps})&=
(v\!\cdot\!\grad \pi^\alpha)\p_{\pi^\alpha}g^\bb +
\tfrac{1}{\eps} (v\!\cdot\!\grad \mathrm{d})\p_\xi g^\bb\,.
\end{aligned}
\end{equation}

\subsection{Normal Boundary Acoustic Operator}
A key role played in the analysis of the viscous boundary layer is the so-called normal boundary acoustic operator $\AA^\mathrm{d}$ of the viscous boundary fluid variables $U^\b = (\rho^\b\,, \mathrm{u}^\b\,, \theta^\b)^\top$:
\begin{equation}\label{A-d}
\mathcal{A}^{\dd} U^\b :=
\begin{pmatrix}
\partial_{\!\zeta}(\mathrm{u}^\b\!\cdot\!\nabla_{\!x}\mathrm{d})\\
\partial_{\!\zeta}(\rho^\b+\theta^\b)\nabla_{\!x}\mathrm{d}\\
\tfrac{2}{\D}\partial_{\!\zeta}(\mathrm{u}^\b\!\cdot\!\nabla_{\!x}\mathrm{d})
\end{pmatrix}\,.
\end{equation}
The null space of $\AA^\b$ and its orthogonal space are
\begin{equation}\nonumber
\begin{aligned}
 \mathrm{Null}(\AA^\mathrm{d}) = \{(\rho^\b\,, \mathrm{u}^\b\,, \theta^\b)^\top \in L^2(\mathrm{d}x\,; \Omega^\delta): \rho^b + \theta^\b= 0\,, \mathrm{u}^\b\cdot \grad\mathrm{d}=0\}\,,\\
 \mathrm{Null}^\perp(\AA^\mathrm{d}) = \{(\rho^\b\,, \mathrm{u}^\b\,, \theta^\b)^\top \in L^2(\mathrm{d}x\,; \Omega^\delta): \theta^\b = \tfrac{2}{\D}\rho^\b\,, \mathrm{u}^\b\cdot \grad\pi=0 \}\,, \end{aligned}
\end{equation}
where the orthogonality is with respect to the inner product endowed on $L^2(\mathrm{d}x\,; \Omega_\delta)$:
\begin{equation}\nonumber
 \langle \widetilde{U}^\b\,, U^\b \rangle_{L^2(\Omega_\delta)} := \int_{U^{\delta}}(\tilde{\rho}^\b \overline{\rho^b} + \tilde{\mathrm{u}}^\b\cdot \overline{\mathrm{u}^\b} + \tfrac{\D}{2} \tilde{\theta}^\b \overline{\theta^b})\,\mathrm{d}x\,.
\end{equation}
The projections from $L^2(\Omega_\delta)$ to Null$(\AA^\mathrm{d})$ and Null$(\AA^\mathrm{d})^\perp$ are defined as
\begin{equation}\label{U-b-decomposition}
\begin{aligned}
 U^\b & = \Pi^\b U^\b + (\mathrm{I}-\Pi^\b) U^\b\\
      & := \begin{pmatrix}
             \tfrac{2}{\D+2}\rho^b-\tfrac{\D}{\D+2}\theta^\b \\
             (\mathrm{u}^\b\cdot \grad \pi^\alpha)\grad\pi^\alpha\\
             \tfrac{\D}{\D+2}\theta^\b - \tfrac{2}{\D+2}\rho^b
          \end{pmatrix}
          + \begin{pmatrix}
             \tfrac{\D}{\D+2}(\rho^\b+ \theta^\b) \\
             (\mathrm{u}^\b\cdot \grad \mathrm{d})\grad \mathrm{d}\\
             \tfrac{2}{\D+2}(\rho^\b+ \theta^\b)
          \end{pmatrix}\,.
\end{aligned}
\end{equation}

We remark that the normal boundary acoustic operator $\AA^\mathrm{d}$ and its Null and Null orthogonal spaces appear in other places to play a key role. For example, the zero viscosity limit of compressible NSF equations with boundary, see \cite{Ding-Jiang}.

\subsection{Preparations}

Before we start the induction, we solve $g^{\Int}_{0}$, $g^{\b}_{0}$
and $i\lambda_{0}$. First, the terms of order $O(\eps^{-1})$ in the interior and viscous boundary layers
in the equation $\eqref{eigen-Lep}$ yield
\begin{equation}\nonumber
\LL g^{\Int}_{0}=0\quad\mbox{and}\quad \LL
g^{\b}_{0}=0\,,
\end{equation}
which imply that $g^{\Int}_{0}$ and
$g^{\b}_{0}$ are hydrodynamic, i.e.
\begin{equation}\nonumber
g^{\Int}_{0}(x,v)=\rho^{\Int}_{0}+v\!\cdot\!\mathrm{u}^{\Int}_{0}+\big(\cvd\big)\theta^{\Int}_{0}= (1,v,\cvd)U^{\Int}_{0}\,,
\end{equation}
and
\begin{equation}\nonumber
\begin{aligned}
g^{\b}_{0}(\pi(x),\zeta,v)&=\rho^{\b}_{0}+v\!\cdot\!
\mathrm{u}^{\b}_{0}+\big(\cvd\big)\theta^{\b}_{0}= (1,v,\cvd)U^{\b}_{0}\,.
\end{aligned}
\end{equation}
We denote above expressions as $g^\Int_0 = I_0(U^\Int_0)$ and $g^\b_0= B_0(U^\b_0)$. Here as operators, $I_0=B_0$, we use different notations to emphasize that one is for interior variable, the other for viscous boundary variable.  The fluid variables $U^{\Int}_{0}$ and $U^{\b}_{0}$ are to be determined. To solve them we need to know the equations satisfied by them and their boundary conditions. It is easy to know from the order $O(\sqrt{\eps}^{-1})$ of the interior part that $g^{\Int}_{1}$ is also hydrodynamic, i.e. $g^{\Int}_{1}= I_0(U^\Int_1)\,.$

\subsection{Induction: Round 0}

Now we start our induction arguments, each round includes considering the kinetic boundary layer, viscous boundary layer and interior terms alternately.

{\bf \underline{Step 1: Order $O(\sqrt{\eps}^{-2})$ in the kinetic boundary layer.}}

The order $O(\sqrt{\eps}^{-2})$ in the boundary condition
$\eqref{dual-BC}$ gives
\begin{equation}\label{BC-g0-half}
L^\mathcal{R}
\tilde{g}_{0}=0\,,
\end{equation}
where we use the notation
$\tilde{g}=g^{\Int}+g^{\b}$.
Because $g^{\Int}_{0}$ and $g^{\b}_{0}$ are
hydrodynamic, then
\begin{equation}\label{normal-00}
\big[\mathrm{u}^{\Int}_{0}+\mathrm{u}^{\b}_{0}\big]\cdot\mathrm{n}=0\quad \mbox{on}\quad\! \pO\,.
\end{equation}

{\bf \underline{Step 2: Order $O(\sqrt{\eps}^{-1})$ in the viscous boundary layer.}}

Next, the order $O(\sqrt{\eps}^{-1})$ in the viscous boundary layer reads
\begin{equation}\label{12-b}
\begin{aligned}
&\LL g^{\b}_{1}=\VGRAD \mathrm{d}\partial_\zeta g^{\b}_{0}\\
= & \partial_{\!\zeta}\mathrm{u}^{\b}_{0}\!\otimes\!\!\grad\mathrm{d}:\!\mathrm{A}
+\partial_{\!\zeta}\theta^{\b}_{0}\grad
\mathrm{d}\!\cdot\!\mathrm{B}+ \A^\dd U^{\b}_{0}\!\cdot\!(1,v,\cvd)\,.
\end{aligned}
\end{equation}

Lemma \ref{Kinetic-Sol} implies that the solvability condition for equation $\eqref{12-b}$ is $\A^\dd U^{\b}_{0}=0$, i.e. $U^\b_0$ lies in the kernel of the normal boundary acoustic operator:
\begin{equation}\nonumber
\partial_{\!\zeta}(\rho^{\b}_{0}+\theta^{\b}_{0})\grad  \mathrm{d}=0\quad \mbox{and}\quad\!
\partial_{\!\zeta}( \mathrm{u}^{\b}_{0}\!\cdot\! \grad  \mathrm{d})=0\,.
\end{equation}
from which we deduce that $\rho^{\b}_{0}+\theta^{\b}_{0}$ is constant in $\zeta$. Since $\rho^{\b}_{0}+\theta^{\b}_{0}\rightarrow 0 $ as
$\zeta \rightarrow \infty$, then
\begin{equation}\label{rho-add-theta-0}
\rho^{\b}_{0}+\theta^{\b}_{0}=0\,.
\end{equation}
Similarly, we have
\begin{equation}\label{u-normal-0}
\mathrm{u}^{\b}_{0}\!\cdot\! \grad  \mathrm{d}=0\,,
\end{equation}
which also gives that on the boundary $\partial\Omega$,
\begin{equation}\label{bc-b-0}
\mathrm{u}^{\b}_{0}(x,\zeta\!=\!0)\!\cdot\!\mathrm{n}=0\,.
\end{equation}
Combining \eqref{normal-00} with the condition \eqref{bc-b-0}, we deduce that
\begin{equation}\label{bc-int-0}
\mathrm{u}^{\Int}_{0}\!\cdot\!\mathrm{n}=0\quad \mbox{on}\quad\! \pO\,.
\end{equation}
Under these conditions, $g^{\b}_{1}$ can be expressed as
\begin{equation}\label{gb-10}
\begin{aligned}
g^{\b}_{1}& =B_0(U^\b_1) + B_1(U^\b_0) \\
& := (1,v,\cvd)U^{\b}_{1}+
\{\partial_{\!\zeta}\mathrm{u}^{\b}_{
0}\!\otimes\!\grad\mathrm{d}\!:\!\AHat+\partial_{\!\zeta}\theta^{\b}_{0}\grad
\mathrm{d}\!\cdot\!\BHat\}\,.
\end{aligned}
\end{equation}
Note that $B_1$ is a linear operator.

{\bf \underline{Step 3: Order $O(\sqrt{\eps}^{0})$ in the interior.}}

To find the equation satisfied by
$U^{\Int}_{0}$, we consider the order $O(\sqrt{\eps}^0)$
in the interior part:
\begin{equation}\label{int-00}
\LL g^{\Int}_{2}=\VGRAD g^{\Int}_{0}-
i\lambda_{0}g^{\Int}_{0}\,.
\end{equation}
Using the boundary condition \eqref{bc-int-0} which means that
$U^{\Int}_{0}$ is in the domain of the acoustic
operator $\AA$, the solvability of \eqref{int-00} gives
\begin{equation}\label{sol-int-00}
\AA
U^{\Int}_{0}=i\lambda_{0}U^{\Int}_{0}\,,
\end{equation}
which is a first order linear hyperbolic system with the boundary condition \eqref{bc-int-0}. If
$i\lambda_{0}=0$, \eqref{sol-int-00} is the so-called acoustic system whose solutions $U^{\Int}_{0}$ satisfies the incompressible and Boussinesq relations.
This case will be treated in a separate paper \cite{J-M-2}.

If $i\lambda_{0}\neq 0$,
then from the discussions in section \ref{Acoustic-Mode}, especially \eqref{acoustic-equ}, we know that
 the system \eqref{sol-int-00} has a family of solutions: for each $k\in \mathbb{N}$,
$U^{\Int}_{0}$ are eigenvectors of $\AA$, and can be
constructed from the eigenvectors of the Laplace operator with
Neumann boundary condition, i.e.
\begin{equation}\label{sol-int-0}
U^{\Int}_{0}=U^{\tau,k}_{0}\,,\quad\mbox{and}\quad
\lambda_{0}=\tau \lambda^k\,,\quad k=1,2,\cdots
\end{equation}
where $\tau$ denotes the signs either $+$ or $-$, see the discussion of the spectrum of $\AA$ in section
\ref{Acoustic-Mode}, in particular \eqref{laplac-neu},
\eqref{eigen-A} and \eqref{acoustic-equ}. Consequently, every term in the ansatz \eqref{1-ansatz-half} and \eqref{1-lambda-half} depends on the choice of $k\in \mathbb{N}$ and $\tau$.  \\

\noindent{\bf Remark:} \eqref{sol-int-0} is the building-block of
the construction of the approximate eigenvector
$g^{k}_{\eps,N}$ and eigenvalues
$\lambda^{k}_{\eps,N}$ for any $k, N \in \mathbb{N}$. It means that the leading order term of $(g^{k}_{\eps,N}, \lambda^{k}_{\eps,N})$ is in acoustic modes. For this reason, we call $g^{k}_{\eps,N}$ the boundary layer in {\em acoustic} regime.

Now we can represent
$g^{\Int}_{2}$ as
\begin{equation}\label{g-int-20}
\begin{aligned}
g^{\Int}_{2}&=I_0(U^{\tau,k,\Int}_2) + I_2(U^{\tau,k,\Int}_0)\\
& := U^{\tau,k,\Int}_{2}\!\cdot\!(1,v,\cvd)+ \{\grad \mathrm{u}^{\tau,k,\Int}_{0}\!:\!\widehat{\mathrm{A}}(v)+\grad \theta^{\tau,k,\Int}_{0}\!\cdot\!\widehat{\mathrm{B}}(v)\}\,,
\end{aligned}
\end{equation}
where $\mathrm{u}^{\tau,k,\Int}_{0}=\sqrt{\tfrac{\D+2}{2\D}}\frac{\grad
\Psi^k}{i\lambda^{\tau,k}_{0}}$ and
$\theta^{\tau,k,\Int}_{0}=\sqrt{\frac{2}{\D(\D+2)}}\Psi^k$, and $\Psi^k$ is defined in \eqref{laplac-neu}.\\

\noindent{\bf Remark:} For the notational simplicity, from now on we drop $\tau$ on the subscript, unless specifically mentioned.

\subsection{Induction: Round 1}

Now we move to round 1 which studies the kinetic boundary layer, viscous boundary layer and interior terms {\em alternately}.

{\bf \underline{Step 1: Order $O(\sqrt{\eps}^{-1})$ in the kinetic boundary layer.}}

The order $O(\sqrt{\eps}^{-1})$ of the kinetic boundary layer in the ansatz gives that  $g^{k,\bb}_{1}(x,v,\xi)$ obeys the following linear kinetic boundary layer equation in $\Omega^{\delta}\times \mathbb{R}^{\D}\times \mathbb{R}_{+}$ : (recalling the definition of $\LL^{BL}$ in \eqref{L-BL})
\begin{equation}\label{BL-Equation-10}
\begin{aligned}
 \LL^{BL} g^{k,\bb}_{1} & = 0\,, \quad \mbox{in}\quad\xi >0\,,\\
g^{k,\bb}_{1}&\longrightarrow 0\,,\quad\mbox{as}\quad \xi\rightarrow
\infty\,,
\end{aligned}
\end{equation}
with boundary condition at $\xi=0$
\begin{equation}\label{BC-bb-10}
L^\mathcal{R} g^{k,\bb}_{1} = H^{k,\bb}_{1}\,,\quad \mbox{on}\quad\xi
= 0\,,\quad v\cdot \mathrm{n}(x)>0\,,
\end{equation}
where $H^{k,\bb}_{1}$ is of the form: for $x\in \pO$, $v\cdot \mathrm{n}(x)>0,$
\begin{equation}\label{h-bb-1}
\begin{aligned}
H^{k,\bb}_{1}(x,v)&=-L^\mathcal{R}\tilde{g}^{k}_{1}+ L^\mathcal{D}\tilde{g}^{k}_{0} \\
& := \tilde{h}^\bb_0(U^{\Int}_{1},U^{\b}_{1}) + \tilde{h}^\bb_1(U^{\Int}_{0},U^{\b}_{0})\,.
\end{aligned}
\end{equation}
Here $\tilde{h}^\bb_0(U^{\Int}_{1},U^{\b}_{1}) $ is a linear function of $U^\Int_1$ and $U^\b_1$, and $\tilde{h}^\bb_1(U^{\Int}_{0},U^{\b}_{0})$ is a linear function of $U^\Int_0$ and $U^\b_0$. More specifically,
\begin{equation}\label{H0}
\begin{aligned}
\tilde{h}^\bb_0(U^{\Int}_{1},U^{\b}_{1}):=& -L^\mathcal{R}(I_0(U^\Int_1)+ B_0(U^\b_1))\\
=& -2(v\!\cdot\!\mathrm{n})(\tilde{\mathrm{u}}^k_1\!\cdot\!\mathrm{n})\,,
\end{aligned}
\end{equation}
and
\begin{equation}\label{H1}
\begin{aligned}
\tilde{h}^\bb_1(U^{\Int}_{0},U^{\b}_{0}) := & -L^\mathcal{R} B_1(U^\b_0)+ L^\mathcal{D} (I_0(U^\Int_0)+ B_0(U^\b_0))\\
=&L^\mathcal{R}(\partial_{\!\zeta}\mathrm{u}^{k,\b}_{0}\!\otimes\!\mathrm{n}\!:\!\AHat
+\partial_{\!\zeta}\theta^{k,\b}_{0}\mathrm{n}\!\cdot\!\BHat)\\
& - v\cdot [\tilde{\mathrm{u}}^k_0\!\cdot\!\grad\pi^\alpha]\grad\pi^\alpha+ \big(\tfrac{\D+1}{2}- \tfrac{|v|^2}{2}\big)\tilde{\theta}^k_0  + \big(\mathrm{n}\cdot v-\tfrac{\sqrt{2\pi}}{2}\big)(\tilde{\mathrm{u}}^k_0\!\cdot\!\mathrm{n})\,.
\end{aligned}
\end{equation}

The boundary conditions for the tangential components of $\mathrm{u}^{k,\b}_0$ and $\theta^{k,\b}_0$ can be derived from the solvability condition of the above kinetic boundary layer equations. The formulas \eqref{BC-D} and \eqref{BC-theta} of Lemma \ref{BC} give that on the boundary
$\pO$,
\begin{equation}\label{BC u-theta-0}
[\mathrm{u}^{k,\b}_{0}-\tfrac{\nu}{\chi}\partial_{\!\zeta}\mathrm{u}^{k,\b}_{0}]^{\mathrm{tan}}=-[\mathrm{u}^{k,\Int}_{0}]^{\mathrm{tan}}
\quad \mbox{and}\quad
\theta^{k,\b}_{0}-\tfrac{\D+2}{\D+1}\tfrac{\kappa}{\chi}\partial_{\!\zeta}\theta^{k,\b}_{0}=-\theta^{k,\Int}_{0}\,,
\end{equation}
from which we can deduce the boundary conditions for
$[\mathrm{u}^{k,\b}_{0}]^{\mathrm{tan}}$ and $\theta^{k,\b}_{0}$ because
$\mathrm{u}^{k,\Int}_{0}$ and
$\theta^{k,\Int}_{0}$ are already determined, thus their boundary values are known. Furthermore, from \eqref{BC-N} of Lemma \ref{BC}, we also have
\begin{equation}\label{normal-10-half}
[\mathrm{u}^{k,\Int}_{1}+ \mathrm{u}^{k,\b}_{1}]\!\cdot\!\mathrm{n}=0\,,\quad \mbox{on}\quad\! \pO\,,
\end{equation}
which implies $\tilde{h}^\bb_0(U^{\Int}_{1},U^{\b}_{1})=0$. Furthermore, from the boundary condition \eqref{BC u-theta-0}
\begin{equation}\label{H-1-u}
\begin{aligned}
\tilde{h}^\bb_1(U^{\Int}_{0},U^{\b}_{0})= & L^\mathcal{R}(\partial_{\!\zeta}\mathrm{u}^{k,\b}_{
0}\!\otimes\!\mathrm{n}\!:\!\AHat+\partial_{\!\zeta}\theta^{k,\b}_{0}\mathrm{n}\!\cdot\!\BHat)\\
 & - \tfrac{\nu}{\chi}v\cdot[\partial_{\!\zeta}\mathrm{u}^{k,\b}_{0}]^{\mathrm{tan}} + \big(\tfrac{\D+1}{2}- \tfrac{|v|^2}{2}\big)\tfrac{\D+2}{\D+1}\tfrac{\kappa}{\chi}\partial_{\!\zeta}\theta^{k,\b}_{0}\,,
\end{aligned}
\end{equation}
which implies that although {\em formally}, $g^{k,\bb}_{1}$ depends on $g^{k,\Int}_{1}$ and $g^{k,\b}_{1}$ which have not been fully determined at this stage, it in fact depends only on the boundary values of  $U^{k,\b}_{0}$, thus once we solve $U^{k,\b}_{0}$, we can solve $g^{k,\bb}_{1}$ {\em completely}. This will be finished at the end of the {\bf \underline{Step 2}}, see \eqref{H-1-sol}.

{\bf \underline{Step 2: Order $O(\sqrt{\eps}^{0})$ in the viscous boundary layer.}}

The equations satisfied by $[\mathrm{u}^{k,\b}_{0}]^{\mathrm{tan}}$ and
$\theta^{k,\b}_{0}$ can be derived by considering the order
$O(\sqrt{\eps}^0)$ of the viscous boundary layer:
\begin{equation}\label{1-b}
\LL g^{k,\b}_{2}=\VGRAD
\mathrm{d}\partial_{\!\zeta}g^{k,\b}_{1}+ v\!\cdot\!\grad
\pi^\alpha\p_{\pi^\alpha} g^{k,\b}_{0}
-i\lambda^{k}_{0}g^{k,\b}_{0}\,.
\end{equation}

Lemma \ref{Kinetic-Sol} implies that the projection of the righthand
side of \eqref{1-b} must be in the null space of $\LL$. We first use
the expression \eqref{gb-10} of $g^{k,\b}_{1}$ to calculate
the projection of
$\VGRAD\mathrm{d}\partial_{\!\zeta}g^{k,\b}_{1}$ onto
Null$(\LL)$. It is easy to see that components of
$\PP(\partial_i\mathrm{d}\partial_m\mathrm{d}
\partial^2_{\zeta\zeta}(\mathrm{u}^{k,\b}_{0})_lv_i\widehat{\mathrm{A}}_{ml})$ on $1$ and $\frac{|v|^2}{2}-\frac{\D}{2}$
are zeros, here $\mathcal{P}$ is defined in \eqref{projection-p}. Applying Lemma \ref{mu-kappa-identity}, we get
\begin{equation}\nonumber
\PP\left(\partial_i\mathrm{d}\partial_m\mathrm{d}\partial^2_{\zeta\zeta}(\mathrm{u}^{k,\b}_{0})_lv_i\widehat{\mathrm{A}}_{ml}\right)
=\left[\nu \partial^2_{\zeta\zeta}
\mathrm{u}^{k,\b}_{0}+\nu (1-\tfrac{2}{\D})\partial^2_{\zeta\zeta}
\big(\mathrm{u}^{k,\b}_{0}\!\cdot\!\grad \mathrm{d}\big)\grad
\mathrm{d}
\right]\cdot v\,.
\end{equation}
Similarly, the components of
$\PP(\partial_i\mathrm{d}\partial_j\mathrm{d}\partial^2_{\zeta\zeta}\theta^{k,b}_{0}v_i\widehat{\mathrm{B}}_j)$
on $1$ and $v$ are zeros. Then applying Lemma \ref{mu-kappa-identity} again, we have
\begin{equation}\nonumber
\PP\left(\partial_i\mathrm{d}\partial_j\mathrm{d}\partial^2_{\zeta\zeta}\theta^{k,b}_{0}v_i\widehat{\mathrm{B}}_j\right)
=\tfrac{\D+2}{\D}\kappa\partial^2_{\zeta\zeta}\theta^{k,b}_{0} \big(\cvd\big)\,,
\end{equation}
where kinematic viscosity $\nu $ and thermal conductivity $\kappa$ are given by \eqref{mu-kappa}. Based on above calculations the solvability conditions for $\eqref{1-b}$ are a system of second order ordinary differential equations in $\zeta$:
\begin{equation}\label{U-b-0}
 -\AA^{\dd}
U^{k,\b}_{1} = (\AA^\pi
+\mathcal{D}^{\dd}-i\lambda^{k}_{0})U^{k,\b}_{0}\,,
\end{equation}
where the tangential acoustic operator $\AA^\pi$ and the normal diffusive operator $\mathcal{D}^\dd$ are
defined as
\begin{equation}\label{D-d}
\mathcal{A}^\pi U^\b :=
\begin{pmatrix}
\mathrm{div}_\pi(\mathrm{u}^\b\!\cdot\!\grad\pi)\\
\p_{\pi^\alpha}(\rho^\b+\theta^\b) \grad
\pi^\alpha\\
\tfrac{2}{\D}\mathrm{div}_\pi(\mathrm{u}^\b\!\cdot\!\grad\pi)
\end{pmatrix}\,,\quad
\mathcal{D}^{\dd}U^\b :=
\begin{pmatrix}
0\\
\nu  \del^2_{\!\zeta\!\zeta}\mathrm{u}^\b
+\nu (1-\tfrac{2}{\D})\del^2_{\!\zeta\!\zeta}(\mathrm{u}^\b\!\cdot\!\grad\mathrm{d})\grad  \mathrm{d}\\
\tfrac{\D+2}{\D}\kappa \del^2_{\!\zeta\!\zeta}\theta^\b
\end{pmatrix}\,,
\end{equation}
for $U^\b=(\rho^\b\,,\mathrm{u}^\b\,,\theta^\b)^\top$. Here we use the notation $\mathrm{div}_\pi(\mathrm{u}^\b\!\cdot\!\grad\pi)=\p_{\pi^\alpha}(\mathrm{u}^\b\!\cdot\!\grad\pi^\alpha)$. Recall the normal acoustic operator $\mathcal{A}^{\dd}$ is defined in \eqref{A-d}.

The ODE system \eqref{U-b-0} can be solved as follows: projecting the system \eqref{U-b-0} on Null$(\AA^\mathrm{d})$ and Null$(\AA^\mathrm{d})^\perp$ respectively, the projection on Null$(\AA^\mathrm{d})$ gives the first order equations of $\rho^\b_1+ \theta^b_1$ and $\mathrm{u}^\b_1\cdot\grad\mathrm{d}$ which can be solved by using the vanishing condition at $\zeta =\infty$, while the projection on Null$(\AA^\mathrm{d})^\perp$ gives the second order equations of $\mathrm{u}^\b_0\cdot\grad\pi$ and $\theta^\b_0$ which can be solved by using the vanishing condition at $\zeta =\infty$ and the Robin boundary condition at $\zeta=0$, i.e. \eqref{BC u-theta-0}.

{\bf 1. \underline{Solve $\rho^{k,\b}_{1} +\theta^{k,\b}_{1}$:} }  We first project the system \eqref{U-b-0} on Null$(\AA^\mathrm{d})^\perp$, the $\mathrm{u}\mbox{-}$component of is
\begin{equation}\label{rho-add-theta-1}
\del_{\!\zeta}(\rho^{k,\b}_{1} +\theta^{k,\b}_{
1})=0\,,\quad\mbox{hence}\quad \rho^{k,\b}_{1}
+\theta^{k,\b}_{1}=0\,.
\end{equation}
The $\rho\mbox{-}$component (or equivalently the $\theta\mbox{-}$component) of the projection on Null$(\AA^\mathrm{d})^\perp$ is
\begin{equation}\label{u-normal-1}
 -\partial_{\!\zeta}(\mathrm{u}^{k,\b}_1\cdot\grad\mathrm{d})= \mathrm{div}_\pi(\mathrm{u}^{k,\b}_0\cdot \grad\pi)+ \kappa \partial^2_{\!\zeta\zeta}\theta^{k,\b}_0\,,
\end{equation}
to solve which we need to first solve $\mathrm{u}^{k,\b}_0\cdot \grad\pi$ and $\theta^{k,\b}_0$. Note that in the derivation of \eqref{rho-add-theta-1} and \eqref{u-normal-1} the relations \eqref{rho-add-theta-0} and \eqref{u-normal-0} are used.

{\bf 2. \underline{Solve $[\mathrm{u}^{k,\b}_{0}]^{\mathrm{tan}}$}:} We next project the system \eqref{U-b-0} on Null$(\AA^\mathrm{d})$, the $\mathrm{u}\mbox{-}$component  gives the equation for $\mathrm{u}^{k,\b}_{0}\!\cdot\! \grad \pi^\alpha$:
\begin{equation}\label{equation-u-normal-0}
\begin{aligned}
(\nu \del^2_{\!\zeta\!\zeta}-
i\lambda^{k}_{0})[\mathrm{u}^{k,\b}_{0}\!\cdot\! \grad \pi^\alpha]&=0\,,\\
[\mathrm{u}^{k,\b}_{0}-\tfrac{\nu}{\chi}\partial_{\!\zeta}\mathrm{u}^{k,\b}_{0}](\zeta=0)\!\cdot\! \grad \pi^\alpha&=-\mathrm{u}^{k,\Int}_{0}(\pi(x))\!\cdot\! \grad \pi^\alpha\,,\\
\lim\limits_{\zeta\rightarrow\infty}\mathrm{u}^{k,\b}_{0}\!\cdot\!
\grad \pi^\alpha&=0\,.
\end{aligned}
\end{equation}
where the boundary condition in the second line of \eqref{equation-u-normal-0}
follows from the fact that $\mathrm{u}^{k,\b}_{0}\!\cdot\!
\grad \pi^\alpha$ is the tangential components of
$\mathrm{u}^{k,\b}_{0}$ because $\grad \pi^\alpha$ is
tangential to $\pO$, see the arguments after \eqref{Proj-Tangent}. The solution to ODE $\eqref{equation-u-normal-0}$ is
\begin{equation}\label{solution-u-0}
\mathrm{u}^{k,\b}_{0}(\pi(x),\zeta)\!\cdot\! \grad
\pi^\alpha=\tfrac{1}{c_\chi\sqrt{\nu}-1}(\mathrm{u}^{k,\Int}_{0}(\pi(x))\!\cdot\! \grad
\pi^\alpha)\exp\big(-(1+ \tau i)\sqrt{\tfrac{\lambda^k_{0}}{2\nu
}}\zeta\big)\,,
\end{equation}
where $\tau = +$ or $-$, $c_\chi=-\tfrac{1+\tau i}{2\chi}\sqrt{2\lambda^k_{0}}$. We denote the solution \eqref{solution-u-0} by
\begin{equation}\label{solution-u-tilde}
\mathrm{u}^{k,\b}_{0}(\pi(x),\zeta)\!\cdot\! \grad\pi = \widetilde{Z}^{\b,\mathrm{u}}_{0}(\zeta,U^{k,\Int}_{0})\,,
\end{equation}
where $\widetilde{Z}^{\b,\mathrm{u}}_{0}(\zeta, \cdot)$ is a linear function. Note that in the righthand side of \eqref{solution-u-0}, $U^{k,\Int}_0$ should be understood as its value on the boundary, i.e. $U^{k,\Int}_0(\pi(x))$.

{\bf 3. \underline{Solve $\theta^{k,\b}_0$}: }  The $\rho\mbox{-}$component (or equivalently the $\theta\mbox{-}$component) of the projection Null$(\AA^\mathrm{d})$ yields
the equation for $\theta^{k,\b}_{0}$:
\begin{equation}\label{equation-theta-0}
\begin{aligned}
(\kappa \del^2_{\!\zeta\!\zeta}-i\lambda^{k}_{0})\theta^{k,\b}_{0}&=0\,,\\
[\theta^{k,\b}_{0}-\tfrac{\D+2}{\D+1}\tfrac{\kappa}{\chi}\partial_{\!\zeta}\theta^{k,\b}_{0}](\zeta=0)&=-\theta^{k,\Int}_{0}(\pi(x))\,,\\
\lim\limits_{\zeta\rightarrow\infty}\theta^{k,\b}_{0}&=0\,.
\end{aligned}
\end{equation}
The solution to $\eqref{equation-theta-0}$ is
\begin{equation}\label{solution-theta-0}
\theta^{k,\b}_{0}(\pi(x),\zeta)=\tfrac{1}{c_\chi \sqrt{\tilde{\kappa}}-1}\theta^{k,\Int}_{0}(\pi(x))\exp\left(-(1 + \tau
i)\sqrt{\tfrac{\lambda^k_{0}}{2\kappa}}\zeta\right)\,,
\end{equation}
where $\tilde{\kappa}=(\tfrac{\D+2}{\D+1})^2\kappa$. We denote the solution \eqref{solution-theta-0} by
\begin{equation}\label{solution-theta-tilde}
\theta^{k,\b}_{0}(\pi(x),\zeta)= \widetilde{Z}^{\b,\theta}_{0}(\zeta,U^{k,\Int}_{0})\,,
\end{equation}
where $\widetilde{Z}^{\b,\theta}_{0}(\zeta, \cdot)$ is a linear function.

{\bf 4. \underline{Solve $\mathrm{u}^{k,\b}_{1}\!\cdot\!\grad\mathrm{d}$}: }
Now the equation \eqref{u-normal-1} becomes
\begin{equation}\label{equation-ud-1}
\p_{\!\zeta}(\mathrm{u}^{k,\b}_{1}\!\cdot\!\grad
\mathrm{d})=-\p_{\pi^\alpha}(\mathrm{u}^{k,\b}_{0}\!\cdot\!\grad
\pi^\alpha)-i\lambda^{k}_{0}\theta^{k,\b}_{0}\,.
\end{equation}
By integrating the equation $\eqref{equation-ud-1}$ from $\zeta$ to $\infty$, it gives
\begin{equation}\nonumber
\mathrm{u}^{k,\b}_{1}\!\cdot\! \grad \mathrm{d}=\widetilde{Z}^\b_{1}(\zeta,U^{k,\Int}_{0})\,,
\end{equation}
where $\widetilde{Z}^\b_{1,0}(\zeta, \cdot)$ is linear. In particular, letting $\zeta=0$ gives the value of $\mathrm{u}^{k,\b}_{1}\!\cdot\!
\mathrm{n}$ on the boundary $\partial\Omega$:
\begin{equation}\label{normal-b-10}
\begin{aligned}
-\mathrm{u}^{k,\b}_{1}\!\cdot\!\mathrm{n}&=\tfrac{1-\tau
i}{\sqrt{2\lambda^k_{0}}}\left( \mathrm{div}_\pi(\mathrm{u}^{k,\Int}_{0}\!\cdot\!\grad \pi)\tfrac{\sqrt{\nu}}{c_\chi\sqrt{\nu}-1}+  \tau i\lambda^{k}_{0} \theta^{k,\Int}_{0}\tfrac{\sqrt{\kappa}}{c_\chi\sqrt{\tilde{\kappa}}-1}\right)\\
&= Z^\b_{1}(U^{k,\Int}_{0})= \widetilde{Z}^\b_{1}(0,U^{k,\Int}_{0})\,,
\end{aligned}
\end{equation}
where $Z^\b_{1}(\cdot)$ is a linear function. Consequently, \eqref{normal-10-half} gives the boundary value $$\mathrm{u}^{k,\Int}_{1}\cdot\mathrm{n}=Z^\b_{1}(U^{k,\Int}_{0})\,.$$

Finally we can represent $g^{k,\b}_{2}$ from \eqref{1-b}:
\begin{equation}\label{solution-1-b}
\begin{aligned}
g^{k,\b} _{2}= &(1,v,\cvd)U^{k,\b}_{2} +\partial_\zeta \mathrm{u}^{k,\b}_{
1} \otimes\!\grad\mathrm{d}:\!\AHat+\partial_\zeta\theta^{k,\b}_{
1}\grad\mathrm{d}\!\cdot\!\BHat\\
&+\p_{\pi^\alpha}
\mathrm{u}^{k,\b}_{0}\otimes \grad\pi^\alpha:\AHat+\p_{\pi^\alpha}\theta^{k,\b}_{
0}\grad \pi^\alpha\cdot\BHat\\
&+\LL^{-1}\PP^\perp\left(\partial_i\mathrm{d}\partial_j\mathrm{d}
\partial^2_{\zeta\zeta}(\mathrm{u}^{k,\b}_{0})_k v_i\widehat{\mathrm{A}}_{jk}+\partial_i\mathrm{d}\partial_j\mathrm{d}
\partial^2_{\zeta\zeta}\theta^{k,b}_{0}v_i\widehat{\mathrm{B}}_j\right) \\
= & B_0(U^\b_2) + B_1(U^\b_1) + B_2(U^\b_0)\,,
\end{aligned}
\end{equation}
where
$$B_2(U^\b_0):= \LL^{-1}\mathcal{P}^\perp\left((v\cdot\grad\mathrm{d})\partial_{\!\zeta}B_1(U^\b_0)+ (v\cdot\grad\pi^\alpha)\partial_{\!\pi^\alpha}B_0(U^\b_0) \right)\,.$$

{\bf \underline{Solve $g^{k,\bb}_{1}$}: }
Now we can represent \eqref{h-bb-1} as,
\begin{equation}\label{H-1-sol}
\begin{aligned}
H^\bb_1(U^{k}_0)&= h^{\bb}_1(U^{k,\Int}_0) \\
& =L^\mathcal{R}(\partial_{\!\zeta}\widetilde{Z}^{\b,\mathrm{u}}_0(0,U^k_0)\grad\pi\!\otimes\!\mathrm{n}\!:\!\AHat+
\partial_{\!\zeta}\widetilde{Z}^{\b,\theta}_0(0,U^k_0)\mathrm{n}\!\cdot\!\BHat)\\
 & - \tfrac{\nu}{\chi}v\cdot\partial_{\!\zeta}\widetilde{Z}^{\b,\mathrm{u}}_0(0,U^k_0)\grad\pi + \left(\tfrac{\D+1}{2}- \tfrac{|v|^2}{2}\right)\tfrac{\D+2}{\D+1}\tfrac{\kappa}{\chi}\partial_{\!\zeta}\widetilde{Z}^{\b,\theta}_0(0,U^k_0)\,,
\end{aligned}
\end{equation}
which is completely determined. Thus we can solve $g^{k,\bb}_{1}$ which we denote by $$g^{k,\bb}_{1}(\pi(x),\xi,v)= K_{1}(\xi,v,U^{k,\Int}_{0}(\pi(x)))\,,$$ where $K_1(\xi,v,\cdot)$ is a linear function.

We summarize that in {\bf \underline{Step 2}} by considering the order $O(\sqrt{\eps}^0)$ in the viscous boundary layer, we determine:
\begin{itemize}
  \item $\rho^{k,\b}_{1}+\theta^{k,\b}_{1}$;
  \item $\mathrm{u}^{k,\b}_0\cdot \grad \pi$ and $\theta^{k,\b}_0$, thus $g^{k,\b}_0$;
  \item $\mathrm{u}^{k,\b}_1\cdot \grad \mathrm{d}$ and hence the boundary value of $\mathrm{u}^{k,\b}_1\cdot \mathrm{n}$ when we take $\zeta=0$, and consequently $\mathrm{u}^{k,\Int}_{1}\cdot\mathrm{n}$ which will be used in {\bf \underline{Step 3}};
  \item expression of $g^{k,\b}_2$;
  \item $g^{k,\bb}_{1}$.
\end{itemize}

{\bf \underline{Step 3: Order $O(\sqrt{\eps}^{1})$ in the interior.}}

The order $O(\sqrt{\eps}^{1})$ in the interior part yields
\begin{equation}\label{int-10}
\LL g^{k,\Int}_{3}=v\!\cdot\!\grad g^{k,\Int}_{1}
-i \lambda^{k}_{0}g^{k,\Int}_{1}-i
\lambda^{k}_{1}g^{k,\Int}_{0}\,,
\end{equation}
and the solvability condition of which is
\begin{equation}\label{sol-int-10}
\begin{aligned}
(\AA - i \lambda^{k}_{0}) U^{k,\Int}_{1} & = i
\lambda^{k}_{1} U^{k,\Int}_{0}\,,\quad\mbox{in}\quad\! \Omega\,,\\
 \mathrm{u}^{k,\Int}_{1}\cdot\mathrm{n}& =Z^\b_{1}(U^{k,\Int}_{0})\,,\quad \mbox{on}\quad\! \pO\,.
\end{aligned}
\end{equation}
To solve \eqref{sol-int-10}, we apply Lemme \ref{Solve-A}. The formula \eqref{imu} gives
\begin{equation}\label{l12}
i\lambda^{k}_{1}=\int_{\partial\Omega}[\mathrm{u}^{k,\Int}_{1}\cdot
\mathrm{n}]\Psi^k \,\mathrm{d}\sigma_x = \int_{\partial\Omega} Z^\b_{1}(U^{k,\Int}_{0})\Psi^k \,\mathrm{d}\sigma_x \,.
\end{equation}
Note that
$\grad \Psi^k  =
g^{\gamma\beta}\tfrac{\p \Psi^k}{\p\pi^\beta}
\tfrac{\p}{\p\pi^\gamma}$, and $\grad \pi^\alpha=
g^{\alpha\delta} \tfrac{\p}{\p\pi^\delta}$, we have
\begin{equation}\nonumber
\begin{aligned}
\int_\pO \p_{\pi^\alpha}(\grad\Psi^k\!\cdot\!\grad\pi^\alpha)
\Psi^k \, \mathrm{d}\sigma_x
=& -\int_\pO g_{\gamma
\delta}g^{\alpha\delta}g^{\beta\gamma}\tfrac{\p\Psi^k}{\p\pi^\alpha}\tfrac{\p\Psi^k}{\p\pi^\beta}\,\mathrm{d}\sigma_x\\
=&
-\int_\pO |\nabla_{\!\pi} \Psi^k |^2\mathrm{d}\sigma_x\,,
\end{aligned}
\end{equation}
where $\nabla_{\!\pi}$ is the tangential gradient on $\pO$. Thus
\begin{equation}\label{negative10-half}
i\lambda^{k}_{1}=\Lambda_1\int_{\partial\Omega} |\nabla_{\!\pi}\Psi^k |^2\,\mathrm{d}\sigma_x  + \Lambda_2
\int_{\partial\Omega}
\tfrac{2}{\D+2}(\lambda^k_{0})^2
|\Psi^k |^2 \,\mathrm{d}\sigma_x\,,
\end{equation}
where
\begin{equation}\nonumber
\Lambda_1= -\tfrac{\sqrt{\nu}}{\sqrt{2(\lambda^k_{0})^3}}\tfrac{(2a+1)+\tau i}{(a+1)^2+a^2}\sqrt{\tfrac{\D+2}{\D}}\,,\quad\! \Lambda_2 = - \tfrac{\sqrt{\kappa}}{\sqrt{2(\lambda^k_{0})^3}}\tfrac{(2b+1)+\tau i}{(b+1)^2+b^2}\sqrt{\tfrac{\D+2}{\D}}\,,
\end{equation}
\begin{equation}\nonumber
a=\tfrac{\sqrt{2\lambda^k_{0}\nu}}{2\chi}\,, \quad b=\tfrac{\sqrt{2\lambda^k_{0}\kappa}}{2\chi}\tfrac{\D+2}{\D+1}\,.
\end{equation}
From the expression \eqref{negative10-half}, $i\lambda^{\tau,k}_{1}$
has an important property ({\bf no matter $\tau= +$ or $-$!}):
\begin{equation}\label{dissip-half}
\mathrm{Re}(i\lambda^{\tau,k}_{1})<0\,.
\end{equation}

\begin{proof} From \eqref{negative10-half}, we can only conclude $\mathrm{Re}(i\lambda^{k}_{1})\leq 0$. The strict negativity comes from the following argument. Indeed, assume that $  \mathrm{Re}(i\lambda^{k}_{1}) = 0
$. This would imply that $ \grad \Psi^k  = 0  $ and $
\Psi^k  = 0  $ on the boundary $\pO$. Hence, extending $
\Psi^k $ by $0$ outside of $\Omega$ and denoting by $  \tilde
\Psi^k $ this extension, we see that $  \tilde
\Psi^k $ is an eigenvector of $ -\Delta_{\!x} $ on the whole
space with compact support which is impossible.  Hence
\eqref{dissip-half} holds.
\end{proof}

\noindent{\bf Remark:} The above formula $\eqref{negative10-half}$ and the strict inequality \eqref{dissip-half} is
crucial in this paper, because it gives dissipativity, which will be
used later in proving the damping of the acoustic waves in the
Navier-Stokes limit.\\

\underline{Case 1}: If $\frac{\D}{\D+2}[\lambda^k_0]^2$ is a {\em simple} eigenvalue of $-\Delta_{\!x}$ with Neumann boundary condition, see \eqref{laplac-neu}.  By Lemma \ref{Solve-A}, \eqref{l12} is the only solvability condition under which the system \eqref{sol-int-10} can be solved uniquely as
\begin{equation}\label{U-int-10-siple}
U^{k,\Int}_{1}  = Z^\Int_{1}(U^{k,\Int}_{0})\,,
\end{equation}
where $Z^\Int_{1}(U^{k,\Int}_{0})\in\mathrm{Null}(\A)^\perp$. Note that the system \eqref{sol-int-10} is linear and the boundary data $Z^\bb_1$ is also linear in $U^{k,\Int}_{0}$. So $Z^\Int_{1}(U^{k,\Int}_{0})$ also linearly depends on $U^{k,\Int}_{0}$, i.e. $Z^\Int_{1}(\cdot)$ is a linear function.

\underline{Case 2}: If the eigenvalues $\lambda^k_0$ is {\em not simple}, an additional compatibility condition is needed, which is given by the formula \eqref{compatability-KL}:
\begin{equation}\label{orth-cond-1}
\int_{\partial\Omega}Z^\b_{1}(U^{k,\Int}_{0})\Psi^l \,\mathrm{d}\sigma_x = 0\,,\quad \mbox{if}
\quad \lambda^k_{0}=\lambda^l_{0}\quad
\mbox{and}\quad k \neq l\,.
\end{equation}
Specifically, this condition reads as
\begin{equation}\nonumber
\Lambda_1 \int_{\partial\Omega} \nabla_{\!\pi}\Psi^k\!\cdot\!\nabla_{\!\pi}\Psi^l\,\mathrm{d}\sigma_x +
\Lambda_2 \int_{\partial\Omega}
\tfrac{2}{\D+2}(\lambda^k_{0})^2
\Psi^k \Psi^l\,\mathrm{d}\sigma_x
=0\,, \quad \mbox{if}
\quad \lambda^k_{0}=\lambda^l_{0}\quad
\mbox{and}\quad k \neq l\,.
\end{equation}
We can define the quadratic form $Q_1$ and the symmetric operator $L_1$ on $\mathrm{H}_0(\lambda)$ as
\begin{equation}\label{Q1}
Q_1(\Psi^k,\Psi^l) = \int_{\pO} Z^\b_{1}(U^{k,\Int}_{0})\Psi^l\,\mathrm{d}\sigma_x\,,
\end{equation}
and
\begin{equation}\label{L-1}
L_1 \Psi^k =  i \lambda^k_{1} \Psi^k\,,
\end{equation}
and the orthogonality condition \eqref{orth-cond-1} is
\begin{equation}\label{orth-cond-1-Q}
Q_1(\Psi^k,\Psi^l)=0\,,\quad\mbox{if}\quad\! \Psi^k,\Psi^l\in \mathrm{H}_0(\lambda)\quad\mbox{and}\quad\! l\neq k\,.
\end{equation}
Under these conditions, applying Lemma
\ref{Solve-A}, we solve $U^{k,\Int}_{1}$ modulo
$\mathrm{Ker}\big(\AA - i \lambda^{k}_{0}\big)$, i.e.
\begin{equation}\label{U-int-10}
U^{k,\Int}_{1}  = Z^\Int_{1}(U^{k,\Int}_{0})+ P_0 U^{k,\Int}_{1}\,,
\end{equation}
 where $P_0 U^{k,\Int}_{1}$ is defined as
\begin{equation}\label{P0U1}
P_0 U^{k,\Int}_{1} = \sum\limits_{l\neq k\,, \lambda^l_{0} = \lambda^k_{0}} a^{kl}_{1} U^{l,\Int}_{0}\,,
\end{equation}
where $a^{kl}_{1}= \< U^{k,\Int}_{1} \,|\, U^{l,\Int}_{0} \>$ will be determined later. Finally we can represent $g^{k,\Int}_{3}$ as
\begin{equation}\label{g-int-30}
\begin{aligned}
g^{k,\Int}_{3}
&= (1,v,\cvd)U^{k,\Int}_{3}+\grad \mathrm{u}^{k,\Int}_{1}\!:\!\widehat{\mathrm{A}}(v)+\grad \theta^{k,\Int}_{1}\!\cdot\!\widehat{\mathrm{B}}(v)\\
& = I_0(U^{k}_3) + I_2(P_0U^{k}_1) + I_3(U^{k,\Int}_{0})\,.
\end{aligned}
\end{equation}
If $\lambda^{k}_{0}$ is a $simple$ eigenvalue, the $P_{0} U^{k}_{1}$ term vanishes. Thus we finish the {\bf Round 1} in the induction.\\

\noindent{\bf Remark:} The orthogonality condition \eqref{orth-cond-1} are given on the eigenfunctions of $-\Delta_{x}$ with Neumann boundary condition with respect to the eigenvalue $[\lambda^k_0]^2$. Usually, the eigenfunctions with Neumann boundary conditions are determined up to some constants, so not unique. \eqref{orth-cond-1} is only used to determine the eigenfunctions corresponding to multiple eigenvalue, and it does not give any new assumption on the geometry of the domain $\Omega$.

\subsection{Induction: Round 2}

Now we move to the second round in the induction, which also includes three steps by considering terms in the kinetic, viscous boundary layers and interior alternatively.

{\bf \underline{Step 1: Order $O(\sqrt{\eps}^{0})$ in the kinetic boundary layer.}}

The order $O(\sqrt{\eps}^{0})$ of the kinetic boundary layer in the ansatz gives that $g^{k,\bb}_{2}$ satisfies the linear
boundary layer equation
\begin{equation}\label{BL-Equation-20}
\begin{aligned}
\LL^{BL} g^{k,\bb}_{2}  = 0\,, \quad &\mbox{in}\quad\xi >0\,,\\
g^{k,\bb}_{2}\longrightarrow 0\,,\quad &\mbox{as}\quad \xi\rightarrow
\infty\,,
\end{aligned}
\end{equation}
with boundary condition at $\xi=0$
\begin{equation}\label{BC-bb-20}
L^{\mathcal{R}} g^{k,\bb}_{2} = H^{k,\bb}_{2}\,,\quad \mbox{on}\quad\xi
= 0\,,\quad v\cdot \mathrm{n}>0\,,
\end{equation}
where $H^{k,\bb}_{2}$ is of the form:
\begin{equation}\label{h-bb-2}
\begin{aligned}
H^{k,\bb}_{2}&=-L^{\mathcal{R}}\tilde{g}^{k}_{2}+ L^{\mathcal{D}} (\tilde{g}^{k}_{1} + g^{k,\bb}_1)\\
& = \tilde{h}^\bb_0(U^{\Int}_{2},U^{\b}_{2}) + \tilde{h}^\bb_1(U^{\Int}_{1},U^{\b}_{1}) + \tilde{h}^\bb_2(U^{\Int}_{0},U^{\b}_{0})\,.
\end{aligned}
\end{equation}
Here
\begin{equation}\label{H2}
\tilde{h}^\bb_2(U^{k,\Int}_{0}, U^{\b}_{0})= -L^{\mathcal{R}}(I_2(U^\Int_0)+ B_2(U^\b_0))
 + L^{\mathcal{D}}(B_1(U^\b_0)+ K_1(U^\b_0))\,.
\end{equation}
Comparing with \eqref{h-bb-1}, we note that the first two terms of \eqref{h-bb-2} are the same as \eqref{h-bb-1} replacing by arguments with subscripts higher than 1. In other words,
\begin{equation}
\tilde{h}^\bb_0(U^{\Int}_{2},U^{\b}_{2}):= -2(v\!\cdot\!\mathrm{n})(\tilde{\mathrm{u}}^k_2\!\cdot\!\mathrm{n})\,,
\end{equation}
and
\begin{equation}\label{H-1-1}
\tilde{h}^\bb_1(U^{\Int}_{1},U^{\b}_{1}) := L^\mathcal{R}(\partial_{\!\zeta}\mathrm{u}^{k,\b}_{1}\!\otimes\!\mathrm{n}\!:\!\AHat
+\partial_{\!\zeta}\theta^{k,\b}_{1}\mathrm{n}\!\cdot\!\BHat)
 - v\cdot [\tilde{\mathrm{u}}^k_1\!\cdot\!\grad\pi^\alpha]\grad\pi^\alpha+ \big(\tfrac{\D+1}{2}- \tfrac{|v|^2}{2}\big)\tilde{\theta}^k_1\,.
\end{equation}
Note that comparing to \eqref{H1}, we have used the boundary condition $\mathrm{u}^{k}_{1}\cdot\mathrm{n}=0\,.$

Next, the formulas \eqref{BC-D} and \eqref{BC-theta} give the boundary conditions
\begin{equation}\label{bc-b-10-u-half}
[\mathrm{u}^{k,\b}_{1}-\tfrac{\nu}{\chi}\partial_{\!\zeta}\mathrm{u}^{k,\b}_{1}]^{\mathrm{tan}}=-[\mathrm{u}^{k,\Int}_{1}]^{\mathrm{tan}} + V^\mathrm{u}_{1}(U^{k,\Int}_{0})\,,
\end{equation}
and
\begin{equation}\label{bc-b-10-theta-half}
\theta^{k,\b}_{1,0}-\tfrac{\D+2}{\D+1}\tfrac{\kappa}{\chi}\partial_{\!\zeta}\theta^{k,\b}_{1}=-\theta^{k,\Int}_{1} + V^\theta_{1}(U^{k,\Int}_{0})\,,
\end{equation}
where
\begin{equation}\label{V-u-theta}
\begin{aligned}
 V^{\mathrm{u}}_{1}(U^{k,\Int}_{0})&= -\tfrac{\nu}{\chi}[2d(\mathrm{u}^{k,\Int}_{0})\!\cdot\!\mathrm{n}]^{\mathrm{tan}}+\left[ \int_{v\cdot \mathrm{n}>0}  L^\mathcal{D} S_1(v\!\cdot\! \mathrm{n})v M\,\dd v \right]^{\mathrm{tan}}\,,\\
 V^{\theta}_{1}(U^{k,\Int}_{0})&= \tfrac{\sqrt{2\pi}}{\D+1}\int_{v\cdot \mathrm{n} >0} L^\mathcal{D} S_1(v\!\cdot\! \mathrm{n})
|v|^2M\,\dd v\,,
\end{aligned}
\end{equation}
and $S_{1}=-(\p_{\zeta}
\mathrm{u}^{k,\b}_{0}\!\otimes\!\mathrm{n}:\widehat{\mathrm{A}}+\p_{\zeta}\theta^{k,\b}_{0}
\mathrm{n} \cdot \widehat{\mathrm{B}}) + K_{1}(U^{k,\Int}_{0})$. Here we have used the facts
\begin{equation}\nonumber
\begin{aligned}
& \left[\mathrm{n}_i\mathrm{n}_j\mathrm{n}_l\partial^2_{\!\zeta\!\zeta}(\mathrm{u}^{k,\b}_{0})_k
\int_{\mathbb{R}^\D}\LL^{-1}\mathcal{P}^\perp(v_i\widehat{A}_{jk})v_lv_m\,M\mathrm{d}v\right]^{\mathrm{tan}}\\
& + \left[\mathrm{n}_i\mathrm{n}_j\mathrm{n}_l\partial^2_{\!\zeta\!\zeta}\theta^{k,\b}_{0}
\int_{\mathbb{R}^\D}\LL^{-1}\mathcal{P}^\perp(v_i\widehat{B}_{j})v_lv_m\,M\mathrm{d}v\right]^{\mathrm{tan}}\\
& = 0\,,
\end{aligned}
\end{equation}
and
\begin{equation}\nonumber
\begin{aligned}
& \mathrm{n}_i\mathrm{n}_j\mathrm{n}_l\partial^2_{\!\zeta\!\zeta}(\mathrm{u}^{k,\b}_{0})_k
\int_{\mathbb{R}^\D}\LL^{-1}\mathcal{P}^\perp(v_i\widehat{A}_{jk})v_l|v|^2\,M\mathrm{d}v\\
& + \mathrm{n}_i\mathrm{n}_j\mathrm{n}_l\partial^2_{\!\zeta\!\zeta}\theta^{k,\b}_{0}
\int_{\mathbb{R}^\D}\LL^{-1}\mathcal{P}^\perp(v_i\widehat{B}_{j})v_l|v|^2\,M\mathrm{d}v\\
& = 0\,,
\end{aligned}
\end{equation}
Thus, \eqref{H-1-1} is reduced to
\begin{equation}
\begin{aligned}
\tilde{h}^\bb_1&(U^{\Int}_{1},U^{\b}_{1})=  L^\mathcal{R}(\partial_{\!\zeta}\mathrm{u}^{k,\b}_{1}
\!\otimes\!\mathrm{n}\!:\!\AHat+\partial_{\!\zeta}\theta^{k,\b}_{1}\mathrm{n}\!\cdot\!\BHat)\\
 & - \tfrac{\nu}{\chi}v\cdot[\partial_{\!\zeta}\mathrm{u}^{k,\b}_{1}]^{\mathrm{tan}} -v\cdot V^\mathrm{u}_{1}(U^{k,\Int}_{0})+
 \big(\tfrac{\D+1}{2}- \tfrac{|v|^2}{2}\big)\tfrac{\D+2}{\D+1}\tfrac{\kappa}{\chi}\partial_{\!\zeta}\theta^{k,\b}_{1} + \big(\tfrac{\D+1}{2}- \tfrac{|v|^2}{2}\big)V^\theta_{1}(U^{k,\Int}_{0})\,.
\end{aligned}
\end{equation}

Furthermore, the formula \eqref{BC-N} gives the boundary condition on the normal direction:
\begin{equation}\label{normal-20-half}
\mathrm{u}^{k,\Int}_{2}\!\cdot\!\mathrm{n}= - \mathrm{u}^{k,\b}_{2}\!\cdot\!\mathrm{n}\,, \quad \mbox{on}\quad\! \pO\,,
\end{equation}
from which we know $\tilde{h}^\bb_0(U^{\Int}_{2},U^{\b}_{2})=0$. Thus from \eqref{h-bb-2}, to solve $g^{k,\bb}_{2}$ we don't need know $g^{k,\Int}_{2}$ and $g^{k,\b}_{2}$, although formally it does. Thus, once we solve $\mathrm{u}^{k,\b}_{1}\cdot\grad\pi$ and $\theta^{k,\b}_{1}$, we can solve $g^{k,\bb}_{2}$.

{\bf \underline{Step 2: Order $O(\sqrt{\eps}^{1})$ in the viscous boundary layer.}}

The equations of $\mathrm{u}^{k,\b}_{1}\cdot\grad\pi$ and
$\theta^{k,\b}_{1}$ can be found by analyzing the order
$O(\sqrt{\eps}^{1})$ of the viscous boundary layer in
\eqref{1-ansatz-half} which gives
\begin{equation}\label{b-10-half}
\LL g^{k,\b}_{3}=(\VGRAD
\mathrm{d})\partial_{\!\zeta}g^{k,\b}_{2}+ (v\!\cdot\!\grad
\pi^\alpha)\p_{\pi^\alpha} g^{k,\b}_{1}
-i\lambda^{k}_{0}g^{k,\b}_{1}-i\lambda^{k}_{1}g^{k,\b}_{0}\,.
\end{equation}
The solvability condition for \eqref{b-10-half} yields that
\begin{equation}\label{U-b-1}
 -\AA^{\dd}
U^{k,\b}_{2} = (\AA^\pi
+\mathcal{D}^{\dd}-i\lambda^{k}_{0})U^{k,\b}_{1}+ (\mathcal{F}_1- i\lambda^k_{1})U^{k,\b}_{0}\,,
\end{equation}
where the linear operator $\mathcal{F}_1(U^{k,\b}_{0})=(\mathcal{F}^\rho_1,\mathcal{F}^\mathrm{u}_1,\mathcal{F}^\theta_1)^\top(U^{k,\b}_{0})$ is defined as
$$
\mathcal{F}_1(U^{k,\b}_{0}) \!\cdot\!(1,v,\cvd) := \mathcal{P}\left\{v\cdot\grad\mathrm{d}\partial_{\!\zeta}B_2(U^{k,\b}_0)+ v\cdot\grad\pi\partial_{\pi}B_1(U^{k,\b}_0)\right\}\,.
$$
The precise form of $\mathcal{F}_1(U^{k,\b}_{0})$ is tedious and not easy to represent explicitly, however, is not of great importance for the later analysis. The $\rho $ component vanishes, while the $\mathrm{u}$ and $\theta$ components are linear functions of third order $\zeta$ derivatives of $\mathrm{u}^{k,\b}_0$ and $\theta^{k,\b}_0$ respectively.

Similar as in solving \eqref{U-b-0}, we derive the ODE satisfied by $\rho^{k,\b}_{2}+ \theta^{k,\b}_{2}$ from the $\mathrm{u}\mbox{-}$component of the projection of \eqref{U-b-1} on Null$(\AA^\mathrm{d})^\perp$:
\begin{equation}\label{rho-theta-2-0-half}
-\p_\zeta(\rho^{k,\b}_{2}+ \theta^{k,\b}_{2})
= \left\{\nu(2-\tfrac{2}{\D})\p^2_{\!\zeta\!\zeta}- i\lambda^k_{0}\right\}[\mathrm{u}^{k,\b}_{1}\!\cdot\!\grad \mathrm{d}]+ \mathcal{F}^{\mathrm{u}}_1(U^{k,\b}_{0})\!\cdot\!\grad\mathrm{d}\,.
\end{equation}
Note that both $\mathrm{u}^{k,\b}_{1}\!\cdot\!\grad \mathrm{d}$ and $\mathrm{u}^{k,\b}_{0}\!\cdot\!\grad \pi^\alpha$ (which is included in $\mathcal{F}^{\mathrm{u}}_1(U^{k,\b}_{0})\!\cdot\!\grad\mathrm{d}$) are known from the last round and linear in $U^{k,\Int}_{0}$, so the righthand side of \eqref{rho-theta-2-0-half} is known and a linear function of $U^{k,\Int}_{0}$. Integrating \eqref{rho-theta-2-0-half} from $\zeta$ to $\infty$ gives
\begin{equation}\label{rho-add-theta-20-half}
   \rho^{k,\b}_{2}+ \theta^{k,\b}_{2}= Y^\b_{2}(\zeta,U^{k,\Int}_{0})\,,
\end{equation}
where $Y^\b_{2}(\zeta,\cdot)$ is a linear function. We can also derive the ODEs satisfied by $\mathrm{u}^{k,\b}_{1}\!\cdot\!\grad\pi$ and $\theta^{k,\b}_{1}$: \begin{equation}\label{equation-u-1-0-half}
\left\{\nu\del^2_{\!\zeta\!\zeta}-i\lambda^{k}_{0}\right\}[\mathrm{u}^{k,\b}_{1}\!\cdot\!\grad\pi]= i \lambda^k_{1} [\mathrm{u}^{k,\b}_{0}\!\cdot\!\grad\pi] - \mathcal{F}^{\mathrm{u}}_1(U^{k,\b}_{0})\cdot \grad \pi\,,
\end{equation}
and
\begin{equation}\label{equation-theta-1-0-half}
\left\{\kappa\del^2_{\!\zeta\!\zeta}-i\lambda^{k}_{0}\right\}\theta^{k,\b}_{1}= i \lambda^k_{1} \theta^{k,\b}_{0}-( \tfrac{2}{\D+2}\mathcal{F}^\rho_1- \tfrac{\D}{\D+2}\mathcal{F}^\theta_1)(U^{k,\b}_{0})\,,
\end{equation}
with boundary conditions \eqref{bc-b-10-u-half} and \eqref{bc-b-10-theta-half} respectively. Because of the linearity of the above equations, we can solve \eqref{equation-u-1-0-half} and \eqref{equation-theta-1-0-half} as
\begin{equation}\label{sol-b-u-theta-1}
\begin{aligned}
   \mathrm{u}^{k,\b}_{1}\!\cdot\!\grad\pi& =  \widetilde{Z}^{\b,\mathrm{u}}_{0}(\zeta,P_0 U^{k,\Int}_{1})+ \widetilde{Z}^{\b,\mathrm{u}}_{1}(\zeta,U^{k,\Int}_{0})\,,\\
\theta^{k,\b}_{1} & =  \widetilde{Z}^{\b,\theta}_{0}(\zeta,P_0U^{k,\Int}_{1})+ \widetilde{Z}^{\b,\theta}_{1}(\zeta,U^{k,\Int}_{0})\,,
\end{aligned}
\end{equation}
 recalling $\widetilde{Z}^{\b,\mathrm{u}}_{0}$ and $\widetilde{Z}^{\b,\theta}_{0}$ are defined in \eqref{solution-u-0} and \eqref{solution-theta-0} respectively, and $\widetilde{Z}^{\b,\mathrm{u}}_{1}$ is the solution of
\begin{equation}\nonumber
\begin{aligned}
   (\nu\del^2_{\!\zeta\!\zeta}-i\lambda^{k}_{0}) \mathrm{u} &= i \lambda^k_{1} [\mathrm{u}^{k,\b}_{0}\!\cdot\!\grad\pi]- \mathcal{F}^{\mathrm{u}}_1(U^{k,\b}_{0})\cdot \grad \pi\,,\\
   [\mathrm{u}-\tfrac{\nu}{\chi}\partial_{\!\zeta}\mathrm{u}](\zeta=0) & = -Z^\Int_{1}(U^{k,\Int}_{0})_{\mathrm{u}}\!\cdot\!\grad\pi +  V^{\mathrm{u}}_{1}(U^{k,\Int}_{0})\!\cdot\!\grad\pi\,,\\
   \lim_{\zeta\rightarrow \infty}\mathrm{u} & = 0\,,
\end{aligned}
\end{equation}
and $\widetilde{Z}^{\b,\theta}_{1}$ is the solution of
\begin{equation}\nonumber
\begin{aligned}
   (\nu\del^2_{\!\zeta\!\zeta}-i\lambda^{k}_{0}) \theta &= i \lambda^k_{1} \theta^{k,\b}_{0}- \mathcal{F}^\theta_1(U^{k,\b}_{0})\,,\\
   [\theta-\tfrac{\D+2}{\D+1}\tfrac{\kappa}{\chi}\partial_{\!\zeta}](\zeta=0) & = -Z^\Int_{1}(U^{k,\Int}_{0})_{\theta} +  V^{\theta}_{1}(U^{k,\Int}_{0})\,,\\
   \lim_{\zeta\rightarrow \infty}\theta & = 0\,,
\end{aligned}
\end{equation}
where $Z^\Int_{1}(U^{k,\Int}_{0})_{\mathrm{u}}$ and $Z^\Int_{1}(U^{k,\Int}_{0})_\theta$ denote the $\mathrm{u}$ and $\theta$ components of $Z^\Int_{1}(U^{k,\Int}_{0})$ respectively.

From the $\rho\mbox{-}$component of the projection of \eqref{U-b-1} on Null$(\AA^\mathrm{d})^\perp$ we can also derive the equation for $\mathrm{u}^{k,\b}_{2}\!\cdot\!\grad\mathrm{d}$:
\begin{equation}\nonumber
-\p_\zeta[\mathrm{u}^{k,\b}_{2}\!\cdot\!\grad\mathrm{d}]=\mathrm{div}_{\!\pi}[\mathrm{u}^{k,\b}_{1}\!\cdot\!\grad\pi]+ i\lambda^k_{0}\theta^{k,\b}_{1}+ i\lambda^k_{1,0}\theta^{k,\b}_{0}\,.
\end{equation}
Integrating from $\zeta$ to $\infty$, we can solve $\mathrm{u}^{k,\b}_{2}\!\cdot\!\grad\mathrm{d}=\widetilde{Z}^\b_{1}(\zeta,P_0U^{k,\Int}_{1}) + \widetilde{Z}^\b_{2}(\zeta,U^{k,\Int}_{0})$, in particular, by taking $\zeta=0$
\begin{equation}\label{normal-b-20-half}
-\mathrm{u}^{k,\b}_{2}\!\cdot\!\mathrm{n} =  Z^\b_{1}(P_0U^{k,\Int}_{1}) + Z^\b_{2}(U^{k,\Int}_{0})= \mathrm{u}^{k,\Int}_{2}\!\cdot\!\mathrm{n} \,,
\end{equation}
where the second equality followed from \eqref{normal-20-half}.

Finally, $g^{k,\b}_{3}$ can be represented as
\begin{equation}\nonumber
\begin{aligned}
    g^{k,\b} _{3}= &(1,v,\cvd)U^{k,\b}_{3} + \partial_\zeta \mathrm{u}^{k,\b}_{
2}\otimes\grad\mathrm{d}\!:\!\AHat+\partial_\zeta\theta^{k,\b}_{
2}\grad\mathrm{d}\!\cdot\!\BHat\\
&+\p_{\pi^\alpha}
\mathrm{u}^{k,\b}_{1}\otimes \grad\pi^\alpha:\AHat+\p_{\pi^\alpha}\theta^{k,\b}_{
1}\grad \pi^\alpha\cdot\BHat\\
&+\LL^{-1}\PP^\perp\left(\partial_i\mathrm{d}\partial_j\mathrm{d}
\partial^2_{\zeta\zeta}(\mathrm{u}^{k,\b}_{1})_k v_i\widehat{\mathrm{A}}_{jk}+\partial_i\mathrm{d}\partial_j\mathrm{d}
\partial^2_{\zeta\zeta}\theta^{k,b}_{1}v_i\widehat{\mathrm{B}}_j\right)\\
&+ \LL^{-1}\PP^\perp\left( (\partial_i\mathrm{d}\partial_j\mathrm{\pi^\alpha}+ \partial_j\mathrm{d}\partial_i\mathrm{\pi^\alpha})(\partial^2_{\zeta\pi^\alpha}(\mathrm{u}^{k,\b}_{0})_k v_i\widehat{\mathrm{A}}_{jk}+\partial^2_{\zeta\pi^\alpha}\theta^{k,\b}_{0}v_i\widehat{B}_j)\right)\\
&+ \LL^{-1}\PP^\perp\left( \partial_i\mathrm{d}\partial_j\mathrm{d}\partial_k\mathrm{d}
[\partial^3_{\!\zeta\zeta\zeta}(\mathrm{u}^{k,\b}_{0})_lv_i\LL^{-1}\PP^\perp(v_j\widehat{\mathrm{A}}_{kl})+
\partial^3_{\!\zeta\zeta\zeta}\theta^{k,\b}_{0,0}v_i\LL^{-1}\PP^\perp(v_j\widehat{\mathrm{B}}_{k})] \right)\\
& - i\lambda^k_{0}\left( \p_{\!\zeta}\mathrm{u}^{k,\b}_{0}\otimes \grad\mathrm{d}\!:\!\LL^{-1}\widehat{\mathrm{A}}+ \p_{\!\zeta}\theta^{k,\b}_{0}\grad\mathrm{d}\!\cdot\!\LL^{-1}\widehat{\mathrm{B}}\right)\\
= & B_0(U^\b_3) + B_1(U^\b_2) + B_2(U^\b_1) + B_3(U^\b_0) \,,
\end{aligned}
\end{equation}
where $B_3(v,U^\b_0)$ can be represented as
\begin{equation}\nonumber
\begin{aligned}
&B_3(v,U^\b_0)\\
= &\LL^{-1}\mathcal{P}^\perp\left((v\cdot\grad\mathrm{d})\partial_{\!\zeta}B_2(v,U^\b_0)+ [(v\cdot\grad\pi^\alpha)\partial_{\!\pi^\alpha}- i\lambda^k_0]B_1(v,U^\b_0) \right)\,.
\end{aligned}
\end{equation}

After $g^{k,\b}_{1}$ is solved (modulo $P_0U^{k,\Int}_{1}$),  go back to the linear kinetic boundary layer equations \eqref{BL-Equation-20}-\eqref{BC-bb-20} of $g^{k,\bb}_{2}$. Straightforward calculations by regrouping terms show that $$H^{k,\bb}_2= h^\bb_1(P_0 U^k_1) + h^\bb_2(U^{k,\Int}_{0})\,,$$ where $h^\bb_2(\cdot)$ is linear and whose detailed expression we omit here. Using the linearity of the kinetic boundary layer equation and the boundary conditions, it is easy to solve that
\begin{equation}\nonumber
g^{k,\bb}_{2}= K_{1}(\xi,v,P_0U^{k,\Int}_{1}) + K_{2}(\xi,v,U^{k,\Int}_{0})\,,
\end{equation}
where $K_{2}(\xi,v,U^{k,\Int}_{0})$ is the solution of the kinetic boundary layer equations \eqref{BL-Equation-20}-\eqref{BC-bb-20} with the boundary condition
\begin{equation}\nonumber
(\gamma_+- L \gamma_-) K_2 = h^\bb_2(U^{k,\Int}_{0})\,,\quad \mbox{on}\quad\xi
= 0\,,\quad v\cdot \mathrm{n}>0\,.
\end{equation}
It is obvious that $K_{2}(\xi,v,\cdot)$ is linear.

{\bf \underline{Step 3: Order $O(\sqrt{\eps}^{2})$ in the interior.}}

The order $O(\eps)$ in the interior part of \eqref{1-ansatz-half} yields
\begin{equation}\label{int-20-half}
\LL g^{k,\Int}_{4}=v\!\cdot\!\grad g^{k,\Int}_{2}
-i \lambda^{k}_{0}g^{k,\Int}_{2}-i
\lambda^{k}_{1}g^{k,\Int}_{1}-i \lambda^{k}_{2}g^{k,\Int}_{0}\,,
\end{equation}
the solvability condition of which is
\begin{equation}\label{sol-int-20-half}
\begin{aligned}
(\AA - i \lambda^{k}_{0}) U^{k,\Int}_{2} &= i
\lambda^{k}_{1} U^{k,\Int}_{1}+ (i\lambda^k_{2}- \mathcal{D})U^{k,\Int}_{0}\quad\mbox{in}\quad\! \Omega\,,\\
\mathrm{u}^{k,\Int}_2\cdot\mathrm{n}&=Z^\b_{1}(P_0U^{k,\Int}_{1}) + Z^\b_{2}(U^{k,\Int}_{0})\quad\mbox{on}\quad\! \partial\Omega\,,
\end{aligned}
\end{equation}
where $\mathcal{D}$ is defined as
\begin{equation}\nonumber
\mathcal{D}U=
\begin{pmatrix}
0\\
\nu\mathrm{div}_{\!x}(\grad\mathrm{u}+ \grad\mathrm{u}^\top-\tfrac{2}{\D}\mathrm{div}_{\!x}\mathrm{u})\\
\tfrac{\D+2}{\D}\kappa\LAP\theta
\end{pmatrix}\,.
\end{equation}
\noindent{\bf Remark:} To apply Lemma \ref{Solve-A}, we require the following orthogonality condition for $U^{k,\Int}_m$:
\begin{equation}\label{orthgonality}
\langle U^{k,\Int}_m\,|\,U^{l,\Int}_0\rangle =
0\,,\quad \mbox{for all}\quad \! m \neq 0\quad\!
\mbox{or}\quad\! k\neq l\,.
\end{equation}

To solve \eqref{sol-int-20-half}, first we apply the formula \eqref{imu} of Lemma \ref{Solve-A} to deduce
\begin{equation}\label{lambda-20}
\begin{aligned}
i\lambda^{k}_{2}&=\int_{\partial\Omega}[\mathrm{u}^{k,\Int}_{2}\cdot
\mathrm{n}]\Psi^k \,\mathrm{d}\sigma_x \\
&= \int_{\partial\Omega} Z^\b_{2}(U^{k,\Int}_{0})\Psi^k \,\mathrm{d}\sigma_x+\int_{\partial\Omega} Z^\b_{1}(P U^{k,\Int}_{0})\Psi^k \,\mathrm{d}\sigma_x +\langle \mathcal{D} U^{k,\Int}_{0}|U^{k,\Int}_{0}\rangle\,,\\
&= \int_{\partial\Omega} Z^\b_{2}(U^{k,\Int}_{0})\Psi^k \,\mathrm{d}\sigma_x+\langle \mathcal{D} U^{k,\Int}_{0}|U^{k,\Int}_{0}\rangle\,.
\end{aligned}
\end{equation}

\underline{Case 1}: If $\lambda^k_0$ is a {\em simple} eigenvalue, then $P U^{k,\Int}_{0}=0$, thus $i\lambda^k_1$ is given by \eqref{lambda-20}.

Otherwise, $\lambda^k_0$ is not a {\em simple} eigenvalue, note that
\begin{equation}\label{lambda-20-vanish}
\begin{aligned}
 \int_{\partial\Omega} Z^\b_{1}(P U^{k,\Int}_{0})\Psi^k \,\mathrm{d}\sigma_x &= \sum\limits_{l\neq k\,, \lambda^l_{0} = \lambda^k_{0}} a^{kl}_{1}\int_{\partial\Omega} Z^\b_{1}(U^{l,\Int}_{0})\Psi^k \,\mathrm{d}\sigma_x\\
 & = 0\,,
\end{aligned}
\end{equation}
because of the orthogonality condition \eqref{Q1} and \eqref{orth-cond-1-Q}\,. The above two identities illustrate that no matter $\lambda^k_0$ is simple or not, $i\lambda^k_2$ is completely determined, which is given by \eqref{lambda-20}.

When $\lambda^k_0$ is not simple, the compatibility condition \eqref{compatability-KL} is needed, which gives
\begin{equation}\label{orth-cond-2-half}
i\lambda^{l}_{1}a^{kl}_{1}+\int_{\partial\Omega}Z^\b_{2}(U^{k,\Int}_{0})\Psi^l \,\mathrm{d}\sigma_x = i\lambda^{k}_{1}a^{kl}_{1}\quad \mbox{if}
\quad \lambda^k_{0}=\lambda^l_{0}\quad
\mbox{and}\quad k \neq l\,.
\end{equation}

\underline{Case 2}: If $\lambda^k_0$ is not a {\em simple} eigenvalue, but $i\lambda^k_{1}$ is a {\em simple} eigenvalue of $L_1$ which is defined in \eqref{L-1}, then for all $l\neq k$ with $i\lambda^l_{0}=i\lambda^k_{0}$, we have $i\lambda^l_{1}\neq i\lambda^k_{1}$. For this case $a^{kl}_{1}$ can be solved from \eqref{orth-cond-2-half} as
\begin{equation}\label{a-kl-10-half}
a^{kl}_{1}= \frac{1}{i\lambda^k_{1}-i\lambda^l_{1}}\int_{\partial\Omega} Z^\b_{2}(U^{k,\Int}_{0})\Psi^l \,\mathrm{d}\sigma_x \,.
\end{equation}
Thus $P_0(U^{k,\Int}_{1})$ is completely determined, and no additional conditions on $\mathrm{H}_0(\lambda)$ rather than \eqref{orth-cond-1}, or equivalently \eqref{orth-cond-1-Q} is needed. However,

\underline{Case 3}: If $\lambda^k_0$ is not a {\em simple} eigenvalue, and $i\lambda^k_{1}$ is also not a {\em simple} eigenvalue of $L_1$, we need more orthogonality condition on
\begin{equation}\nonumber
\mathrm{H}_1= \mathrm{H}(\lambda_1)=\{\Psi\in \mathrm{H}_0: L_1 \Psi= i\lambda^k_{1}\Psi\}\,.
\end{equation}
This orthogonality condition comes from \eqref{orth-cond-2-half}:
\begin{equation}\label{condition-20-half}
\int_{\partial\Omega}Z^\b_{2}(U^{k,\Int}_{0})\Psi^l \,\mathrm{d}\sigma_x =0\,,\quad\mbox{if}\quad\!l\neq k\quad\!\lambda^l_{0}=\lambda^k_{0}\quad\!\lambda^l_{1}=\lambda^k_{1}\,.
\end{equation}
We can define a quadratic form $Q_2$ and the symmetric operator $L_2$ on $\mathrm{H}_1(\lambda_1)$ as
\begin{equation}\label{Q2}
Q_2(\Psi^k\,,\Psi^l)=\int_{\partial\Omega}Z^\b_{2}(U^{k,\Int}_{0})\Psi^l \,\mathrm{d}\sigma_x + \langle \mathcal{D} U^{k,\Int}_{0}|U^{l,\Int}_{0}\rangle\,,
\end{equation}
and $L_2 \Psi^k= i\lambda^k_{2} \Psi^k$, which satisfies that
\begin{equation}\nonumber
Q_2(\Psi^k,\Psi^l)= \int_\Omega L_2(\Psi^k)\Psi^l\,\mathrm{d}x\,.
\end{equation}
Be these definitions, we have $i\lambda^k_2 = Q_2(\Psi^k, \Psi^k)$, and the condition \eqref{condition-20-half} is
\begin{equation}\nonumber
Q_2(\Psi^k\,,\Psi^l)=0\,,\quad \mbox{if}\quad\! \Psi^k,\Psi^l\in \mathrm{H}_1(\lambda_1)\quad\!\mbox{and}\quad\! l\neq k\,.
\end{equation}

Under these conditions, the equation \eqref{sol-int-20-half} can be solved in the following way: Let $U^{k,\Int}_{2}= U^1 + U^2$, where $U^1$ satisfies the equation
\begin{equation}\nonumber
\begin{aligned}
 (\AA - i\lambda^k_{0})U^1 & = i\lambda^k_{1}P_0U^{k,\Int}_{1}\,,\\
    \mathrm{u}^1\!\cdot\!\mathrm{n}&= Z^\b_{1}(P_0U^{k,\Int}_{1})\,,
\end{aligned}
\end{equation}
whose solution in Ker$(\AA-i\lambda^k_{0})^\perp$ is $Z^\Int_{1}(P_0U^{k,\Int}_{1})$, and $U^2$ satisfies the equation
\begin{equation}\nonumber
\begin{aligned}
 (\AA - i\lambda^k_{0})U^2 & = i\lambda^k_{1}Z^\Int_{1}(U^{k,\Int}_{0})+ (i\lambda^k_{2}- \mathcal{D})U^{k,\Int}_{0}\,,\\
    \mathrm{u}^2\!\cdot\!\mathrm{n}&= Z^\b_{2}(U^{k,\Int}_{0})\,,
\end{aligned}
\end{equation}
whose solution in Ker$(\AA-i\lambda^k_{0})^\perp$ is {\em completely} determined, and is denoted by $Z^\Int_{2}(U^{k,\Int}_{0})$. In summary, the equation \eqref{sol-int-20-half} is
\begin{equation}\nonumber
U^{k,\Int}_{2}=P_0U^{k,\Int}_{2}+ Z^\Int_{1}(P_0U^{k}_{1})+ Z^\Int_{2}(U^{k,\Int}_{0})\,,
\end{equation}
where $P_0(U^{k,\Int}_{1})= (P_1+ P_1^\perp)(U^{k,\Int}_{1})$ in which $P_1^\perp U^{k,\Int}_{1}$ is already completely determined
in \eqref{a-kl-10-half} and $P_1U^{k,\Int}_{1}$ will be determined later, and $P_0 U^{k,\Int}_{2}$ is defined the same as in \eqref{P0U1}, i.e.
\begin{equation}\nonumber
P_0 U^{k,\Int}_{2} = \sum\limits_{l\neq k\,, \lambda^l_{0} = \lambda^k_{0}} a^{kl}_{2} U^{l,\Int}_{0}\,,
\end{equation}
where $a^{kl}_{2}= \< U^{k,\Int}_{2} \,|\, U^{l,\Int}_{0} \>$ will be determined later. Finally we can represent $g^{k,\Int}_{4}$ as
\begin{equation}\label{g-int-40}
\begin{aligned}
g^{k,\Int}_{4}
=&I_0(U^k_4)+ I_2(U^k_2) + I_4(U^{k,\Int}_{0})\\
:=& (1,v,\cvd)U^{k,\Int}_{4}+\grad \mathrm{u}^{k,\Int}_{2}\!:\!\widehat{\mathrm{A}}(v)+\grad \theta^{k,\Int}_{2}\!\cdot\!\widehat{\mathrm{B}}(v)\\
&+ \LL^{-1}\PP^\perp\left(
\partial^2_{x_ix_j}(\mathrm{u}^{k,\Int}_{0})_k v_i\widehat{\mathrm{A}}_{jk}+
\partial^2_{x_ix_j}\theta^{k,\Int}_{0}v_i\widehat{\mathrm{B}}_j\right)\\
& - i\lambda^k_{0}\left( \grad \mathrm{u}^{k,\Int}_{0}\!:\!\LL^{-1}\widehat{\mathrm{A}}+ \grad\theta^{k,\Int}_{0}\!\cdot\!\LL^{-1}\widehat{\mathrm{B}}\right)\,,\\
\end{aligned}
\end{equation}
where $I_4(U^{k,\Int}_{0})= \LL^{-1}\mathcal{P}^\perp\left\{(v\!\cdot\!\grad  - i\lambda^k_0)I_2(U^{k,\Int}_{0}) \right\} \,.$ Thus we conclude the {\bf Round 2} in the induction.

\subsection{General case: Induction hypothesis}

For $m \geq 3$, we assume that we have finished the $(m-1)\mbox{-}th$ round, i.e. used the information from the kinetic boundary layer, the viscous boundary and the interior till the order $O(\sqrt{\eps}^{m-3})$, $O(\sqrt{\eps}^{m-2})$ and $O(\sqrt{\eps}^{m-1})$ respectively. Before we solve the next round, we write down the hypothesis that summarizes what we were able to construct till now. We write down this in following 10 statements that we need to check for the $m\mbox{-}th$ round.

$\bf (P^1_{m-1}):$ For $2 \leq j \leq m-1$, $\rho^{k,\b}_j + \theta^{k,\b}_j = \sum^j_{h=2}Y^\b_h(\zeta, P_0 U^{k,\Int}_{j-h})\,;$ For $j=0, 1$, $\rho^{k,\b}_j + \theta^{k,\b}_j=0 \,.$

$\bf (P^2_{m-1}):$ For $0 \leq j \leq m-2$, $\mathrm{u}^{k,\b}_j\cdot\grad\pi= \sum^j_{h=0}\widetilde{Z}^{\b,\mathrm{u}}_h(\zeta, P_0 U^{k,\Int}_{j-h})\,;$

$\bf (P^3_{m-1}):$ For $0 \leq j \leq m-2$, $\theta^{k,\b}_j= \sum^j_{h=0}\widetilde{Z}^{\b,\theta}_h(\zeta, P_0 U^{k,\Int}_{j-h})\,;$

$\bf (P^4_{m-1}):$ For $1 \leq j \leq m-1$, $\mathrm{u}^{k,\b}_j\cdot\grad\mathrm{d} = \sum^j_{h=1}\widetilde{Z}^\b_{h}(\zeta, P_0 U^{k,\Int}_{j-h})\,.$ Taking $\zeta=0$, we deduce that on the boundary we have $-\mathrm{u}^{k,\b}_j\cdot\mathrm{n} = \sum^j_{h=1} Z^\b_{h}(P_0 U^{k,\Int}_{j-h})\,.$

$\bf (P^5_{m-1}):$ For $0 \leq j \leq m$, $g^{k,\b}_j = \sum^j_{h=0}B_h(U^\Int_{j-h})\,,$ where $B_h$ for $h\geq 0$ is defined iteratively starting from $B_0(U^\b)= (1,v,\frac{|v|^2}{2}-\frac{\D}{2})U^\b$:
\begin{equation}\label{B-h}
B_h(U^\b)= \LL^{-1}\mathcal{P}^\perp\{v\cdot\grad\mathrm{d}\partial_{\!\zeta}B_{h-1}(U^\b)+ v\cdot\grad\pi\partial_{\!\pi}B_{h-1}(U^\b)- \sum^{h-3}_{l=0}i\lambda^k_lB_{h-2-l}(U^\b)\}\,.
\end{equation}

$\bf (P^6_{m-1}):$ For $1 \leq j \leq m-1$,
\begin{equation}\label{K-bb}
g^{k,\bb}_j = \sum^{j}_{h=1}K_h(v,\xi, P_0(U^{k,\Int}_{j-h}))\,,
\end{equation}
where the linear operator $K_h(v,\xi, U^{k,\Int}_0)$ is the solution to the linear kinetic boundary layer equation \eqref{B-L-Equation}-\eqref{BC-beta>0} with the source term $s^\bb_h(U^{k,\Int}_0)$ and the boundary source term $h^\bb_h(U^{k,\Int}_0)$.

$\bf (P^7_{m-1}):$ For $1 \leq j \leq m-1$, $i\lambda^k_j = Q_j(\Psi^k,\Psi^k)$, where the quadratic form $Q_1$ and $Q_2$ are defined in \eqref{Q1} and \eqref{Q2} respectively, and $Q_j$ for $3 \leq j \leq m-1$ is defined as
\begin{equation}\label{Q-j}
   Q_j(\Psi^k,\Psi^l)= \sqrt{\tfrac{\D+2}{2\D}}\int_\pO \{Z^\b_j + V^\mathrm{n}_j\}(U^{k,\Int}_0)\Psi^l\,\mathrm{d}\sigma_x + \sum^j_{h=2}\langle \mathcal{G}_h(Z^\Int_{j-h}(U^{k,\Int}_0))| U^{l,\Int}_0 \rangle\,.
\end{equation}
Note $\mathcal{G}_h$ is defined in \eqref{G-definition}.

$\bf (P^8_{m-1}):$ For $0 \leq j \leq m-1$, $U^{k,\Int}_j= P_0 U^{k,\Int}_j + \sum^j_{h=1}Z^\Int_h(P_0 U^{k,\Int}_{j-h})\,;$

$\bf (P^9_{m-1}):$ For $0 \leq j \leq m+1$, $g^{k,\Int}_j = I_0(U^\Int_j) + I_2(U^\Int_{j-2}) + \sum^j_{h=4}I_j(U^\Int_{j-h}),$ where where $I_h$ for $h\geq 0$ is defined iteratively starting from $I_0(U^\Int)= (1,v,\frac{|v|^2}{2}-\frac{\D}{2})U^\Int$, $I_1=0$:
\begin{equation}\label{I-h}
I_h(U^\Int)=\LL^{-1}\mathcal{P}^\perp\{ v\cdot\grad I_{h-2}(U^\Int) - \sum^{h-4}_{l=0}i\lambda^k_l I_{h-2-m}(U^\Int)\}\,.
\end{equation}

$\bf (P^{10}_{m-1}):$ The last assumption to check deals with the number of orthogonality conditions needed and specifies what is already determined and what is still not determined in the construction. We distinguish between $m$ cases:

{\bf Case 1:} $i \lambda^k_{h}$ is a simple eigenvalue of $L_h$ for
 $ 0  \leq h \leq m-2$. No orthogonality condition is needed, and every term  in the expansion is fully determined;

{\bf Case j ($2 \leq j \leq m$):} $i \lambda^k_{h}$ is a multiple eigenvalue of $L_h$ for $0 \leq h \leq j-2$, but $i\lambda_{j-1}$ a simple eigenvalue of $L_{j-1}\,.$
(Note: the {\bf case m} means that
 all the eigenvalues  $i {\lambda^k_{h}}$ for $ 0  \leq h \leq m-2$  are multiple  eigenvalues.)

\begin{itemize}
\item We need the orthogonality conditions: For each $0 \leq h \leq j-2$,
                            \begin{equation}\label{orth-cond-h}
                                  Q_{h+1}(\Psi^k\,,\Psi^l)=0\,,\quad\mbox{for}\quad\! \Psi^k,\Psi^l\in \mathrm{H}_0\cap\cdots\cap \mathrm{H}_h\,,
                            \end{equation}
 where for $h \geq 1$, the space $\mathrm{H}_h= \mathrm{H}_h(\lambda_h)= \{\Psi\in \mathrm{H}_1(\lambda_1)\cap\cdots\cap \mathrm{H}_{h-1}(\lambda_{h-1}): L_h\Psi= i\lambda_h \Psi\}$.

\item  For $1 \leq h \leq m-j$, $U^{k,\Int}_{h}$ are completely determined. (for the {\bf case m}, no term is completely determined.)

\item For $m-j +1 \leq h \leq m-1$, $(P_0^\perp+ \cdots + P_{m-1-h}^\perp) U^{k,\Int}_{h}$ are determined.

\item For $m-j +1 \leq h \leq m-1$, $P_{m-1-h}U^{k,\Int}_{h}$ are not determined,

\end{itemize}
where $P_{h-1}$ is the orthogonal projection on $\mathrm{H}_1(\lambda_1)\cap\cdots\cap\mathrm{H}_{h-1}(\lambda_{h-1})$, and $P_{h-1}= P_h + P^\perp_h$, where $P^\perp_h$ is the orthogonal projection on $\mathrm{H}_1(\lambda_1)\cap\cdots\cap\mathrm{H}_{h-1}(\lambda_{h-1})\cap\mathrm{H}^\perp_h(\lambda_h)\,.$

{\bf Remark:} Regarding  the condition \eqref{orth-cond-h},
 actually we have  a stronger orthogonality
 property which is actually equivalent to \eqref{orth-cond-h}, namely :
for each $0 \leq h \leq j-2$,
 \begin{equation}\label{orth-cond-h-strong}
    \begin{aligned}
     Q_{h+1}(\Psi_k\,,\Psi_l)&=0\,,\quad\mbox{for}\quad\!  l\neq k , \quad
 \Psi_k,\Psi_l\in \mathrm{H}_0\,.
    \end{aligned}
    \end{equation}
Indeed,  we just need to use that  the $L_h$ leave stable the spaces
$\mathrm{H}_h$. Of course, we have to define $L_h$
over the whole space $\mathrm{H}_0$  even if the eigenvalue is simple, but
in this case we just take it to be the identity.

In the next subsection, we are going to prove the 10 hypotheses $\bf (P^1_{m})- (P^{10}_m)$  assuming  $\bf P_{i-1}$  for  $i \leq  m$.

\subsection{Induction: Round m.}
For $m \geq 3$, we assume that we have finished round $m-1$ in the induction process. For the round $m$, as before it includes three steps by considering terms in the kinetic, viscous boundary layers and interior alternatively.

{\bf \underline{Step 1: Order $O(\sqrt{\eps}^{m-2})$ in the kinetic boundary layer.}}

The order $O(\sqrt{\eps}^{m-2})$ of the kinetic boundary layer in the ansatz gives that $g^{k,\bb}_{m}$ satisfies the linear boundary layer equation
\begin{equation}\label{BL-Equation-m}
\begin{aligned}
\LL^{BL} g^{k,\bb}_{m} & = S^{k,\bb}_m\,, \quad \mbox{in}\quad\xi >0\,,\\
g^{k,\bb}_{m} &\longrightarrow 0\,,\quad\mbox{as}\quad \xi\rightarrow
\infty\,,
\end{aligned}
\end{equation}
with boundary condition at $\xi=0$
\begin{equation}\label{BC-bb-m}
L^\mathcal{R} g^{k,\bb}_{m} = H^{k,\bb}_{m}\,,\quad \mbox{on}\quad\xi
= 0\,,\quad v\cdot \mathrm{n}>0\,,
\end{equation}
The source term
\begin{equation}\nonumber
\begin{aligned}
S^{k,\bb}_m & = \left\{v\!\cdot\!\grad \pi^\alpha\partial_{\pi^\alpha} - i\lambda^k_0\right\} g^{k,\bb}_{m-2}-\sum^{m-3}_{j=1}i\lambda^k_j g^{k,\bb}_{m-2-j}\\
            & = \sum^m_{j=3}s^\bb_j(P_0 U^{k,\Int}_{m-j})\,,
\end{aligned}
\end{equation}
where for $3 \leq j \leq m-3$,
\begin{equation}\label{sk-m}
 s^{\bb}_j(U^\Int) = \left\{v\cdot\grad\pi\partial_{\!\pi}- i\lambda^k_0\right\} K_{j-2}(U^\Int)- \sum^{j-3}_{h=1}i\lambda^k_h K_{j-2-h}(U^\Int)\,,
\end{equation}
recalling the functions $K_j(U^\Int)$ are defined through \eqref{K-bb}. The boundary source term is
\begin{equation}\label{h-m-half}
\begin{aligned}
H^{k,\bb}_{m}&=- L^\mathcal{R}\tilde{g}^{k}_{m}+ L^\mathcal{D} (\tilde{g}^{k}_{m-1} + g^{k,\bb}_{m-1}) \\
& = \sum^m_{j=0} \tilde{h}^\bb_j(\widetilde{U}^k_{m-j}) = \sum^m_{j=1}h^\bb_j (P_0 U^{k,\Int}_{m-j})\,.
\end{aligned}
\end{equation}
Here
\begin{equation}\nonumber
\begin{aligned}
 \tilde{h}^\bb_j(\widetilde{U}^k_0)=& - L^{\mathcal{R}}\left\{I_j(U^{k,\Int}_0)+ B_j(U^{k,\b}_0)\right\}\\
                            & + L^\mathcal{D}\left\{I_{j-1}(U^{k,\Int}_0)+ B_{j-1}(U^{k,\b}_0)+ K_{j-1}(U^{k,\Int}_0)\right\}\,.
\end{aligned}
\end{equation}
The definition of $h^\bb_j(P_0 U^{k,\Int}_{m-j})$ is the following: represent all $I_0\,, \cdots\,, I_m$ and $B_0\,,\cdots\,, B_m$ in terms of $U^{k,\Int}_0\,, \cdots\,, P_0 U^{k,\Int}_{m-1}$, and collect the corresponding terms in $\sum^m_{j=0} \tilde{h}^\bb_j(\widetilde{U}^k_{m-j})$, which defines $h^\bb_j(P_0 U^{k,\Int}_{m-j})$ for $j=1\,,\cdots\,, m-1$. Note that there is no $U^{k,\Int}_m$ term, since it {\em formally} only appears in the term $- L^{\mathcal{R}}(I_0(U^\Int_m)+ B_0(U^\b_m))=-2(v\cdot\mathrm{n})(\mathrm{u}^{k,\Int}_m+\mathrm{u}^{k,\b}_m )\cdot\mathrm{n}$, which only depends on $U^{k,\Int}_0\,,P_0 U^{k,\Int}_1\,,\cdots\,, P_0 U^{k,\Int}_{m-1}$, see $\bf (P^4_{m-1})$.

The formulas \eqref{BC-D} and \eqref{BC-theta} give the boundary conditions
\begin{equation}\label{bc-m-u}
[\mathrm{u}^{k,\b}_{m-1}-\tfrac{\nu}{\chi}\partial_{\!\zeta}\mathrm{u}^{k,\b}_{m-1}]^{\mathrm{tan}} +[\mathrm{u}^{k,\Int}_{m-1}]^{\mathrm{tan}}
=\sum^{m-1}_{j=1}V^\mathrm{u}_j(P_0 U^\Int_{m-1-j})\,,
\end{equation}
and
\begin{equation}\label{bc-m-theta}
\theta^{k,\b}_{m-1}-\tfrac{\D+2}{\D+1}\tfrac{\kappa}{\chi}\partial_{\!\zeta}\theta^{k,\b}_{m-1}+ \theta^{k,\Int}_{m-1}
=\sum^{m-1}_{j=1}V^\theta_j(P_0 U^\Int_{m-1-j})\,,
\end{equation}
where
\begin{equation}\label{V-u-sum}
\begin{aligned}
&\sum^{m-1}_{j=1}V^\mathrm{u}_j(P_0 U^{k,\Int}_{m-1-j})
= -\tfrac{\nu}{\chi}[2d(\mathrm{u}^{k,\Int}_{m-2})\!\cdot\!\mathrm{n}]^\mathrm{tan} -\tfrac{\nu}{\chi}\nabla_{\!\pi}[\mathrm{u}^{k,\b}_{m-2}\!\cdot\!\mathrm{n}]\\
+ & \sum^{m-1}_{j=1}\int_{v\cdot\mathrm{n}>0}\left[L^\mathcal{D}\left\{B_j(U^{k,\b}_{m-1-j})+ K_j(P_0 U^{k,\Int}_{m-1-j})\right\}(v\!\cdot\!\mathrm{n})v\right]^{\mathrm{tan}}M\,\mathrm{d}v\\
- & \tfrac{1}{\chi}\left\langle (v\cdot\mathrm{n})v \{\sum^m_{j=4}I_j(U^{k,\Int}_{m-j})+ \sum^m_{j=3}B_j(U^{k,\b}_{m-j}) \}\right\rangle^\mathrm{tan} + \tfrac{1}{\chi}\int^\infty_0 \langle v S^{k,\bb}_{m}\rangle^\mathrm{tan}\mathrm{d}\xi\,,
\end{aligned}
\end{equation}
and
\begin{equation}\label{V-theta-sum}
\begin{aligned}
&\sum^{m-1}_{j=1}V^\theta_j(P_0 U^{k,\Int}_{m-1-j})
= -\tfrac{\D+2}{\D+1}\tfrac{\kappa}{\chi}\partial_{\mathrm{n}}\theta^{k,\Int}_{m-2} + \tfrac{\sqrt{2\pi}}{2(\D+1)}\tilde{\mathrm{u}}^{k}_{m-1}\cdot\mathrm{n}\\
+ & \tfrac{\sqrt{2\pi}}{\D+1}\sum^{m-1}_{j=1}\int_{v\cdot\mathrm{n}>0}\left[L^\mathcal{D}\left\{B_j(U^{k,\b}_{m-1-j})+ K_j(P_0 U^{k,\Int}_{m-1-j})\right\}(v\!\cdot\!\mathrm{n})|v|^2\right]^{\mathrm{tan}}M\,\mathrm{d}v\\
- & \tfrac{1}{(\D+1)\chi}\left\langle (v\cdot\mathrm{n})|v|^2[\sum^m_{j=4}I_j(U^{k,\Int}_{m-j})+ \sum^m_{j=3}B_j(U^{k,\b}_{m-j})]\right\rangle^\mathrm{tan} + \tfrac{\D+2}{\D+1}\tfrac{1}{\chi}\int^\infty_0 \langle (\tfrac{|v|^2}{\D+2}-1) S^{k,\bb}_{m}\rangle^\mathrm{tan}\mathrm{d}\xi\,.
\end{aligned}
\end{equation}
If we express in \eqref{V-u-sum} and \eqref{V-theta-sum} all $U^\Int$ and $U^\b$ in terms of $U^\Int_0\,,P_0 U^\Int_1\,, \cdots \,, P_0 U^\Int_{m-2}$, and collect the corresponding terms, then we can define $V^\mathrm{u}_{m-1}(U^{k,\Int}_0)$ and $V^\theta_{m-1}(U^{k,\Int}_0)$. Note that $V^\mathrm{u}_1\,, \cdots\,, V^\mathrm{u}_{m-2}$ and $V^\theta_1\,, \cdots\,, V^\theta_{m-2}$ have been defined in the previous rounds of the induction.

Furthermore, the formula \eqref{BC-N} gives the boundary condition on the normal direction:
\begin{equation}\label{BC normal-u-m-sum}
[\mathrm{u}^{k,\Int}_{m}+ \mathrm{u}^{k,\b}_{m}]\cdot\mathrm{n}=\sum^m_{j=3}V^\mathrm{N}_j(P_0 U^{k,\Int}_{m-j})\,, \quad \mbox{on}\quad\! \pO\,,
\end{equation}
where
\begin{equation}\nonumber
   V^\mathrm{N}_m(U^{k,\Int}_{0})= \int^\infty_0\left\langle \{v\!\cdot\!\grad\pi^\alpha \partial_{\!\pi^\alpha}- i\lambda^k_0\} K_{m-2}(U^{k,\Int}_{0})\right\rangle\,\mathrm{d}\xi
   -\sum^{m-3}_{j-1}\int^\infty_0\langle i\lambda^k_j K_{m-2-j}(U^{k,\Int}_{0})\rangle\,\mathrm{d}\xi\,.
\end{equation}

{\bf \underline{Step 2: Order $O(\sqrt{\eps}^{m-1})$ in the viscous boundary layer.}}

The equations of $[\mathrm{u}^{k,\b}_{m-1}]^{\mathrm{tan}}$ and
$\theta^{k,\b}_{m-1}$ can be derived by considering the order
$O(\eps^{m-1})$ of the viscous boundary layer:
\begin{equation}\label{m-b}
\LL g^{k,\b}_{m+1}=\VGRAD
\mathrm{d} \partial_{\!\zeta}g^{k,\b}_{m}+ v\!\cdot\!\grad
\pi^\alpha \p_{\pi^\alpha} g^{k,\b}_{m-1}
-\sum^{m-1}_{j=0}i\lambda^{k}_{j}g^{k,\b}_{m-1-j}\,,
\end{equation}
the solvability of which is the following system of ODEs:
\begin{equation}\label{U-b-m-1}
 - \AA^\mathrm{d} U^{k,\b}_m =  \sum^{m-1}_{j=0}(\mathcal{F}_j - i \lambda^k_j)U^{k,\b}_{m-1-j}\,,
\end{equation}
where $\mathcal{F}_{m-1}(U^{k,\b}_0)$ is defined by
\begin{equation}\label{F-m-1}
  \mathcal{P}\left\{v\!\cdot\!\grad\mathrm{d}\partial_{\!\zeta} B_m(U^{k,\b}_0)+ v\!\cdot\!\grad \pi^\alpha\partial_{\!\pi^\alpha}B_{m-1}(U^{k,\b}_0)\right\}= (1,v,\cvd)\mathcal{F}_{m-1}(U^{k,\b}_0)\,.
\end{equation}
Note that the linear operator $\mathcal{F}_1\,, \cdots \,, \mathcal{F}_{m-2}$ have been defined in the previous rounds of the induction. In particular, $\mathcal{F}_0=\AA^\pi + \mathcal{D}^\mathrm{d}$.

Projecting the system \eqref{U-b-m-1} on Null$(\AA^\mathrm{d})^\perp$, the $\mathrm{u}\mbox{-}$component and $\rho\mbox{-}$component of which give the equations of $\rho^{k,\b}_m + \theta^{k,\b}_m$ and $\mathrm{u}^{k,\b}_m\!\cdot\!\grad\mathrm{d}$ respectively:
\begin{equation}\label{rho-add-theta-m}
 -\partial_{\!\zeta}(\rho^{k,\b}_m + \theta^{k,\b}_m)= \left\{\nu(2-\tfrac{2}{\D})\partial^2_{\zeta\zeta}- i\lambda^k_0\right\}(\mathrm{u}^{k,\b}_{m-1}\!\cdot\!\grad\mathrm{d}) + \sum^{m-1}_{j=1}\left\{(\mathrm{I}-\Pi^\mathrm{d})(\mathcal{F}_j - i \lambda^k_j)U^{k,\b}_{m-1-j}\right\}^{\mathrm{u}}\,,
\end{equation}
and
\begin{equation}\label{u-normal-m}
\begin{aligned}
 &-\partial_{\!\zeta}(\mathrm{u}^{k,\b}_m \cdot\grad\mathrm{d})\\
 =& \mathrm{div}(\mathrm{u}^{k,\b}_{m-1}\cdot\grad \pi)+ \kappa\partial^2_{\zeta\zeta}\theta^{k,\b}_{m-1} - i\lambda^k_0 \tfrac{\D}{\D+2}(\rho^{k,\b}_{m-1} + \theta^{k,\b}_{m-1})+ \sum^{m-1}_{j=1}\left\{(\mathrm{I}-\Pi^\mathrm{d})(\mathcal{F}_j - i \lambda^k_j)U^{k,\b}_{m-1-j}\right\}^{\rho}\,.
 \end{aligned}
\end{equation}
Here we use the notation for a  vector $V=(V^\rho\,,V^\mathrm{u}\,,V^\theta)^\top\,.$ Integrating \eqref{rho-add-theta-m} from $\zeta$ to $\infty$ gives
\begin{equation}\label{sol-rho-add-theta-m}
  \rho^{k,\b}_m + \theta^{k,\b}_m = \sum^{m}_{j=2}Y^\b_j(\zeta\,,P_0 U^\Int_{m-j})\,,
\end{equation}
in which the linear operator $Y^\b_2\,, \cdots \,, Y^\b_{m-1}$ have been defined in the previous rounds of the induction. This corresponds to $\bf (P^1_m)$.

Next projecting the system \eqref{U-b-m-1} on Null$(\AA^\mathrm{d})$, the $\mathrm{u}\mbox{-}$component of which gives that $\mathrm{u}^{k,\b}_{m-1}\cdot \grad\pi$ satisfies the ODE
\begin{equation}\label{ODE-u-m-1-tangential}
\begin{aligned}
\left\{\nu \partial^2_{\zeta\zeta} - i\lambda^k_0\right\}\mathrm{u} &= -\partial_\pi(\rho^{k,\b}_{m-1} + \theta^{k,\b}_{m-1})- \sum^{m-1}_{j=1}\left\{\Pi^\mathrm{d}(\mathcal{F}_j - i \lambda^k_j)U^{k,\b}_{m-1-j}\right\}^{\mathrm{u}}\,,\\
[\mathrm{u}-\tfrac{\nu}{\chi}\partial_\zeta\mathrm{u}](\zeta=0)  &= - \mathrm{u}^{k,\Int}_{m-1}\!\cdot\!\grad\pi + \sum^{m-1}_{j=1}V^\mathrm{u}_j(P_0 U^\Int_{m-1-j})\,, \\
\lim_{\zeta\rightarrow \infty} \mathrm{u} &= 0\,,
\end{aligned}
\end{equation}
the solution of which is
\begin{equation}\label{sol-u-m-1-tangential}
 \mathrm{u}^{k,\b}_{m-1}\!\cdot\!\grad\pi = \sum^{m-1}_{j=0}\widetilde{Z}^{\b,\mathrm{u}}_j(\zeta\,, P_0 U^\Int_{m-1-j})\,,
\end{equation}
where the linear operator $\widetilde{Z}^{\b,\mathrm{u}}_0\,, \cdots \,, \widetilde{Z}^{\b,\mathrm{u}}_{m-2}$ have been defined in the previous rounds of the induction. This corresponds to $\bf (P^2_m)$.

Projecting the system \eqref{U-b-m-1} on Null$(\AA)$ gives that $\theta^{k,\b}_{m-1}$ satisfies the ODE
\begin{equation}\label{ODE-theta-m-1}
\begin{aligned}
\left\{\kappa \partial^2_{\zeta\zeta} - i\lambda^k_0\right\}\theta &= - i\lambda^k_0 \tfrac{2}{\D+2}(\rho^{k,\b}_{m-1} + \theta^{k,\b}_{m-1})- \sum^{m-1}_{j=1}\left\{\Pi^\mathrm{d}(\mathcal{F}_j - i \lambda^k_j)U^{k,\b}_{m-1-j}\right\}^\theta\,,\\
[\theta-\tfrac{\D+2}{\D+1}\tfrac{\kappa}{\chi}\partial_\zeta\theta](\zeta=0)  &= - \theta^{k,\Int}_{m-1} + \sum^{m-1}_{j=1}V^\theta_j(P_0 U^\Int_{m-1-j})\,, \\
\lim_{\zeta\rightarrow \infty} \theta &= 0\,,
\end{aligned}
\end{equation}
the solution of which is
\begin{equation}\label{sol-theta-m-1}
 \theta^{k,\b}_{m-1}= \sum^{m-1}_{j=0}\widetilde{Z}^{\b,\theta}_j(\zeta\,, P_0 U^\Int_{m-1-j})\,,
\end{equation}
where the linear operator $\widetilde{Z}^{\b,\theta}_0\,, \cdots \,, \widetilde{Z}^{\b,\theta}_{m-2}$ have been defined in the previous rounds of the induction. This corresponds to $\bf (P^3_m)$.

Having \eqref{sol-u-m-1-tangential} and \eqref{sol-theta-m-1}, go back to \eqref{u-normal-m} and integrate from $\zeta$ to $\infty$, we obtain
\begin{equation}\label{sol-u-m-normal}
\mathrm{u}^{k,\b}_m \cdot\grad\mathrm{d} = \sum^m_{j=1}\widetilde{Z}^\b_j(\zeta\,, P_0 U^\Int_{m-j})\,.
\end{equation}
In particular, by taking $\zeta=0$ in \eqref{sol-u-m-normal}, we obtain
\begin{equation}\label{BC u-m-normal}
-\mathrm{u}^{k,\b}_m \cdot \mathrm{n} = \sum^m_{j=1}Z^\b_j(P_0 U^\Int_{m-j})\,,
\end{equation}
Thus from \eqref{BC u-m-normal} and \eqref{BC normal-u-m-sum} we derive the boundary condition of $\mathrm{u}^{k,\Int}_m\cdot\mathrm{n}$ which will be used in the next step to solve $U^{k,\Int}_m$. This corresponds to $\bf (P^4_m)$.

Under these conditions, the equation \eqref{m-b} can be solved as $g^{k,\b}_{m+1} = \sum^{m+1}_{h=0}B_h(U^\Int_{j-h})\,,$ where $B_h(U^\Int)$ is defined in \eqref{B-h}. Note that $B_0\,,B_1\,,\cdots\,, B_m$ have been determined in the previous rounds of the induction. This corresponds to $\bf (P^5_m)$.

Finally we go back to the kinetic boundary equation \eqref{BL-Equation-m} to solve $g^{k,\bb}_m$ as
\begin{equation}
 g^{k,\bb}_m = \sum^m_{j=1} K_j(v,\xi, P_0 U^\Int_{m-j})\,,
\end{equation}
where the linear operator $K_m(v,\xi, U^{k,\Int}_0)$ is the solution to the linear kinetic boundary layer equation \eqref{B-L-Equation}-\eqref{BC-beta>0} with the source term $s^\bb_m(U^{k,\Int}_0)$ and the boundary source term $h^\bb_m(U^{k,\Int}_0)$. Note that $K_1\,, \cdots \,, K_{m-1}$ have been defined in the previous rounds of the induction. This corresponds to $\bf (P^6_m)$. Thus we finish {\bf \underline{Step 2}}.

{\bf \underline{Step 3: Order $O(\sqrt{\eps}^{m})$ in the interior.}}

The order $O(\eps^\frac{m}{2})$ in the interior part of the ansatz yields
\begin{equation}\label{int-m}
\LL g^{k,\Int}_{m+2}=v\!\cdot\!\grad g^{k,\Int}_{m}
- \sum^m_{j=0}i \lambda^{k}_{j,0}g^{k,\Int}_{m-j}\,,
\end{equation}
and the solvability condition of which is
\begin{equation}\label{sol-int-m}
\begin{aligned}
(\AA - i \lambda^{k}_{0}) U^{k,\Int}_{m} &= \sum^m_{j=1} (i\lambda^k_{j}- \mathcal{G}_j)U^k_{m-j}\,,\quad\mbox{in}\quad\! \Omega\,,\\
\mathrm{u}^{k,\Int}_m\cdot\mathrm{n} &=\sum^m_{j=1}(Z^\b_j+V^\mathrm{N}_j)(P_0 U^{k,\Int}_{m-j})\,,\quad\mbox{on}\quad\! \pO\,,
\end{aligned}
\end{equation}
where the vector-valued linear operator $\mathcal{G}_j$ for $j\geq 4$ is defined as
\begin{equation}\label{G-definition}
(1,v,\cvd)\mathcal{G}_j(U^{k,\Int}_0) = \mathcal{P}\left\{v\cdot\grad I_j(U^{k,\Int}_0)\right\}\,.
\end{equation}
Note that $\mathcal{G}_1=0$, $\mathcal{G}_2=\mathcal{D}$, $\mathcal{G}_3=0$ and $V^\mathrm{N}_j=0$ for $j=1, 2$.

Applying Lemma \ref{Solve-A} to \eqref{sol-int-m}, and recalling in \eqref{Q-j} the definition of $Q_j$ for $j=1\,,2\,,\cdots\,, m-1$, the formula \eqref{imu} gives
\begin{equation}\label{lambda-m}
\begin{aligned}
i\lambda^k_{m}= &\sqrt{\tfrac{\D+2}{2\D}}\int_{\pO}(Z^\b_m+V^\mathrm{n}_m)(U^{k,\Int}_0)\Psi^k\,\mathrm{d}\sigma_{\!x}+ \sum^m_{j=2}\left\langle\mathcal{G}_j(Z^\Int_{m-j}(U^{k,\Int}_0))|U^{k,\Int}_0 \right\rangle\\
& + \sum^{m-1}_h Q_h(P_0 U^{k,\Int}_{m-h}\,, \Psi^k)\,.
\end{aligned}
\end{equation}
The orthogonality condition \eqref{orth-cond-h-strong} implies that the second line of \eqref{lambda-m} vanishes. Thus we can define the righthand side of the first line of \eqref{lambda-m} as $Q_m(\Psi^k, \Psi^k)$. Thus we verify that $i\lambda_m = Q_m(\Psi^k, \Psi^k)$ which is completely determined. This corresponds to $\bf (P^7_m)$.

To solve the equation \eqref{sol-int-m}, we need to consider $m+1$ cases:

{\bf Case 1:} $i \lambda^k_{h}$ is a simple eigenvalue of $L_h$ for $0 \leq h \leq m-1$. No orthogonality condition is needed, and every term is fully determined;

{\bf Case j ($2 \leq j \leq m+1$):} $i \lambda^k_{h}$ is a multiple eigenvalue of $L_h$ for $0 \leq h \leq j-2$, and a simple eigenvalue of $L_h$ for $j-1 \leq h \leq m-1$.

We only consider the {\bf case m+1} here, i.e. all the eigenvalues   $i \lambda^k_{h}$ are multiple. The other cases are simpler. Taking the inner product with $U^{l,\Int}_0$, for $l \neq k$, $\lambda^l_{0}=\lambda^k_{0}$, which is
\begin{equation}\label{a-kl-m}
   \sum^{m-1}_{h=1}i \lambda^k_{h} a^{kl}_{m-h} = \sum^{m-1}_{h=1}Q_h(P_0 U^{k,\Int}_{m-h}, \Psi^l)+ Q_m(\Psi^k,\Psi^l)\,.
\end{equation}

If $\Psi^k , \Psi^l  \in \mathrm{H}_1(\lambda_1)\cap \mathrm{H}_2(\lambda_2)\cap\cdots \cap \mathrm{H}_{m-1}(\lambda_{m-1})$, then because of the orthogonality condition \eqref{orth-cond-h} for $1 \leq h \leq m-2$,
\begin{equation}\nonumber
\begin{aligned}
Q_h(P_0 U^{k,\Int}_{m-h}, \Psi^l  )&= Q_h(P_{h-1}U^{k,\Int}_{m-h}, \Psi^l  ) + Q_h(\sum^{h-1}_{\delta=1}P^\perp_{\delta}U^{k,\Int}_{m-h}, \Psi^l  )\\
&= i \lambda^l_{h} a^{kl}_{m-h}\,.
\end{aligned}
\end{equation}
For $h=m-1$, $Q_{m-1}(P_0 U^{k,\Int}_{1},\Psi^l  )= i \lambda^l_{m-2}a^{kl}_{1}+ Q_{m-1}(P^\perp_{m-2}U^{k,\Int}_{1},\Psi^l  )$. Thus, the identity \eqref{a-kl-m} implies that we need the orthogonality condition that for $k \neq l$,
\begin{equation}\label{orth-cond-m-1}
   Q_m(\Psi^k \,, \Psi^l  ) = \int_\Omega L_m(\Psi^k )\Psi^l  \,\mathrm{d}x = 0\,,\quad\mbox{for}\quad\! \Psi^k ,\Psi^l  \in \mathrm{H}_1\cap\cdots\cap \mathrm{H}_{m-1}\,,
\end{equation}
where the symmetric operator $L_m$ is defined by $L_m \Psi^l   = i\lambda^l_{m}\Psi^l  $, for $\Psi^l   \in \mathrm{H}_1\cap\cdots\cap \mathrm{H}_{m-1}\,.$

If $\Psi^k , \Psi^l  \in \mathrm{H}_1(\lambda_1)\cap \mathrm{H}_2(\lambda_2)\cap\cdots\cap \mathrm{H}_{m-2}(\lambda_{m-2})\cap \mathrm{H}^\perp_{m-1}(\lambda_{m-1})$, i.e $\lambda^k_{h}=\lambda^l_{h}$ for $0\leq h \leq m-2$, but $\lambda^k_{m-1}\neq\lambda^l_{m-1}$, from the identity \eqref{a-kl-m}, for these $k,l$, $a^{kl}_{1}$ can be determined by
\begin{equation}\nonumber
   a^{kl}_{1}= \tfrac{1}{i \lambda^k_{m-1} - i \lambda^l_{m-1}} Q_m(\Psi^k ,\Psi^l  )\,.
\end{equation}
This means that $(P^\perp_0+ P^\perp_1+ \cdots + P^\perp_{m-1})U^{k,\Int}_{1}$ is completely determined, but $P_{m-1}U^{k,\Int}_{1}$ is still left as undetermined.

If $\Psi^k ,\Psi^l  \in \mathrm{H}_1\cap\cdots\cap\mathrm{H}_{m-3}\cap\mathrm{H}^\perp_{m-2}$,
\begin{equation}\label{m-m-2perp}
Q_{m-1}(P^\perp_{m-1} U^{k,\Int}_{1}\,, \Psi^l  ) + Q_m(\Psi^k \,, \Psi^l  )\\
= (i\lambda^k_{m-2}- i\lambda^l_{m-2}) a^{kl}_{2} + i\lambda^k_{m-1} a^{kl}_{1}\,,
\end{equation}
from which $a^{kl}_{2}$ thus $P^\perp_{m-2}U^{k,\Int}_{2}$ is completely determined.

Under these solvability conditions, the equation \eqref{sol-int-m} can be solved as
\begin{equation}\nonumber
    U^{k,\Int}_m= P_0 U^{k,\Int}_m+ \sum^m_{h=1}Z^\Int_h(P_0 U^{k,\Int}_{m-h})\,,
\end{equation}
where $Z^\Int_m(U^{k,\Int}_0)$ is the solution to the following equation:
\begin{equation}\label{Z-int-m}
\begin{aligned}
(\AA - i \lambda^{k}_{0}) U &= \sum^m_{h=1} (i\lambda^k_{h}- \mathcal{G}_h)Z^\Int_{m-h}(U^{k,\Int}_0)\,,\quad\mbox{in}\quad\! \Omega\,,\\
\mathrm{u}\cdot\mathrm{n} &=( Z^\b_m +V^\mathrm{N}_m)U^{k,\Int}_0\,,\quad\mbox{on}\quad\! \pO\,.
\end{aligned}
\end{equation}
Thus, $U^{k,\Int}_m$ is determined modulo $P_0 U^{k,\Int}_m\,, P_1 U^{k,\Int}_{m-1}\,, \cdots\,, P_{m-1}U^{k,\Int}_1$
which are undetermined at this stage. Under these conditions, the equation \eqref{int-m} is solved as $g^{k,\Int}_{m+2} = I_0(U^\Int_{m+2}) + I_2(U^\Int_{m}) + \sum^{m+2}_{h=4}I_{m+2}(U^\Int_{m+2-h})$ This corresponds to $\bf (P^8_m)$, $\bf (P^9_m)$  and $\bf (P^{10}_m)$.

We can now inductively continue the process, namely
 go to the order $O(\sqrt{\eps}^{m-1})$ of the kinetic boundary layer, the order $O(\sqrt{\eps}^{m})$ of the viscous boundary layer, then the order $O(\sqrt{\eps}^{m+1})$ of the interior, and so on.  We should do this at least till the order $N+2$ where $N$ is the precision of the error in \eqref{error1}. Note however, that for a given  $\lambda= \lambda^k_0$, we may only need to construct a small
number of the $L_j$ if after few steps all the eigenvalues become simple, namely if for some $j$ all the eigenvalues of $L_j$ are simple on the
space   $\mathrm{H}_1(\lambda_1)\cap\cdots\cap \mathrm{H}_{j-1}(\lambda_{j-1})$.
It is clear that if  the eigenvalues become simple for some  $j \leq N+2$, then
the orthogonality condition \eqref{orth-cond-h} allows to determine
 the eigenfunctions $ \Psi^k $ uniquely.
If the process does not end, then we just need to satisfy the condition
till the order $N+2$ which yield a non-unique choice of eigenfunctions.
Also, in this case, we set all the undetermined pieces  of the eigenfunction,
namely those left undetermined to be zero.


\section{Proof of Proposition \ref{main-prop}: Truncation Error Estimates}

In the previous sections, we construct the kinetic-fluid boundary
layers up to any order for $\alpha_\eps = \sqrt{2\pi}\chi\eps^\beta$. Now we define the approximated eigenfunction and
eigenvalues $g^{k}_{\eps,N}$ and
$\lambda^{k}_{\eps,N}$ by truncation in the corresponding
ansatzs. More specifically,
\begin{equation}\nonumber
\begin{aligned}
g^{k}_{\eps,N} &= \sum^N_{j=0}\left\{ g^{k,\Int}_j +g^{k,\b}_j \right\}\eps^{\frac{j}{2}} + \sum^N_{j=1}g^{k,\bb}_j \eps^{\frac{j}{2}}\,,\\
\lambda^{k}_{\eps,N} &= \sum^N_{j=0}
\lambda^{k}_j\sqrt{\eps}^j\,.
\end{aligned}
\end{equation}

\subsection{Estimates of $R^{k}_{\eps,N}$}

Using the eigen-equation \eqref{eigen-Lep}, we can easily find that
the error term $R^{k}_{\eps,N}$ has the form of
\begin{equation}\nonumber
\begin{aligned}
R^{k}_{\eps,N} = &\{(i\lambda^k_0 - v\!\cdot\!\grad)g^{k,\Int}_{N-1}+ (i\lambda^k_0- v\!\cdot\!\grad\pi)\partial_\pi [g^{k,\b}_{N-1}+ g^{k,\bb}_{N-1}]- v\!\cdot\!\grad\mathrm{d} \partial_{\!\zeta}g^{k,\b}_N\\
& + \sum^{N-1}_{j=1} i \lambda^k_j \hat{g}^k_{N-1-j}\}\eps^{\frac{N-1}{2}} + ``\mbox{higher order terms}"\,,
\end{aligned}
\end{equation}
where $\hat{g}^k= g^{k,\Int}+ g^{k,\b} + g^{k,\bb}$. From the constructions of $g^{k,\Int}$, $g^{k,\b}$ and $g^{k,\bb}$,  it is easy to know that
\begin{equation}\nonumber
\|g^{k}_{j}\|_{L^r(\mathrm{d}x\,;L^p(aM\,\mathrm{d}v))}
\leq C\,,
\end{equation}
for all $j$, and $1 < r,p < \infty$, where $g^k$ stands for $g^{k,\Int}$, $g^{k,\b}$ or $g^{k,\bb}$.

Indeed, both the hydrodynamic and the kinetic parts of $g^{k,\Int}$ and $g^{k,\b}$ have coefficients in terms of the components of $U^{k,\Int}$ and $U^{k,\b}$. From Lemma \ref{Solve-A}, the solutions $U^{k,\Int}_j$ of the equation \eqref{A-Lambda} can be represented linearly in terms of components of $U^{\Int}_i$ for $0\leq i < j$ and the boundary terms of $U^{k,\b}_i$ and $g^{k,\b}_i$ for $0\leq i<j$. Note that the pseudo inverse operator $(\AA- i\lambda^{\tau,k})^{-1}$ is bounded, and furthermore, the boundary values of $U^{k,\b}_i$ and $g^{k,\b}_i$ are linearly in terms of $U^{k,\Int}_l$ for $0\leq l\leq i$. For $U^{k,\b}_j$, their components are solutions of second order ordinary differential equations with boundary conditions in terms of $U^{k,\Int}_i$ and $g^{k,\b}_i$ for $0 \leq i \leq j$. Moreover, the solutions of the linear kinetic boundary layer equation \eqref{B-L-Equation} for $g^{k,\bb}_{i}$ are bounded in $L^r(\mathrm{d}x,L^p(aM\mathrm{d}v))$ in terms of $U^{k,\Int}_l$ and $U^{k,\b}_l$ for $0\leq l < i$. So all $g^{k,\Int}$, $g^{k,\b}$ or $g^{k,\bb}$ are linearly depends on components of $U^{k,\Int}_0$, i.e. $\Psi^k$ and $\grad \Psi^k$ which is the eigenfunctions of $-\Delta_{\!x}$ with Neumann boundary conditions. From the basic regularity theory of elliptic operator, they are bounded in $L^r(\mathrm{d}x; \Omega)$ for any $1 < r \leq \infty$.

For the term $(v\!\cdot\!\grad\mathrm{d})\del_\zeta
g^{k,\bb}_{N}$,
we integrate over $\Omega\times\RD$ and use simple change of
variable $(y_1\,,y_2\,,\cdots\,, y_{\D-1})=
\pi(x)\,,y_{\D}=\frac{\mathrm{d}(x)}{\sqrt{\eps}}$, we can have extra
$\sqrt{\eps}$, so all will be in the higher order terms.
Thus, we have the error estimate \eqref{error1}.

\subsection{Estimates of $g^{k}_{\eps,N}-g^{k,\Int}_{0}$}

The leading order term of $g^{k}_{\eps,N}-g^{k,\Int}_{0}$ is
$g^{k,\b}_0$, so using the expressions above for $\mathrm{u}^{k,\b}_0\,,
\theta^{k,\b}_0$ and a simple change of variable
which will give an extra $\sqrt{\eps}$, we have
\begin{equation}\nonumber
\begin{aligned}
\|g^{k}_{\eps,N}-g^{\tau,k,\Int}_{0}\|_{L^r(\mathrm{d}x,L^p(aM\mathrm{d}v))}&
\leq C\eps^{\frac{1}{2r}}\,.
\end{aligned}
\end{equation}
Thus we get \eqref{error2}.

\subsection{Boundary error estimate}
Finally, the boundary error term $r^{k}_{\eps,N}$ is
\begin{equation}\nonumber
r^{k}_{\eps,N}= -\sqrt{\eps}^{N+1}L^\mathcal{D}(g^{k,\Int}_N + g^{k,\b}_N + g^{k,\bb}_N)
\end{equation}
from which we can get the estimate \eqref{r-error}. Thus we
finish the proof of the  Proposition \ref{main-prop}.

\section{Proof of the Weak Convergence in Theorem \ref{Diri-limit} and \ref{Navier-limit}}
In order to derive the fluid equation with the boundary conditions, we
need to pass to the limit in approximate local conservation laws
built from the renormalized Boltzmann equation $\eqref{BE-F}$. We
choose the renormalization used in \cite{LM}:
\begin{equation}\label{renorm-Gamma}
\Gamma(Z)=\frac{Z-1}{1+(Z-1)^2}\,.
\end{equation}
After multiplying $\Gamma'(\Ge)$ and dividing by $\eps$, equation $\eqref{BE-G}$ becomes
\begin{equation}\label{scale-BE}
\p_t\gps+\frac{1}{\eps}\Divv \gps=\frac{1}{\eps}\Gamma'(\Ge)\iint\limits_{\SD\times\RD}q_\eps
b(\omega\,,v_1-v)\,\mathrm{d}\omega M_1\,\mathrm{d}v_1\,,
\end{equation}
where $\gps=\tfrac{1}{\eps}\Gamma(\Ge)$ can be considered as the $L^2$ part of the fluctuations $g_\eps$ and $q_\eps$ is the scaled collision integrand defined as
\begin{equation}\label{renorm-g}
q_\eps= \frac{G'_{\eps 1}G'_\eps-G_{\eps 1}G_\eps}{\eps^2}\,.
\end{equation}
By introducing $N_\eps=1+\eps^2g^2_\eps$, we can write
\begin{equation}\nonumber
\gps=\frac{g_\eps}{N_\eps}\,,\quad \Gamma'(\Ge)=\frac{2}{N^2_\eps}-\frac{1}{N_\eps}\,.
\end{equation}

When moments of the renormalized Boltzmann equation
$\eqref{scale-BE}$ are formally taken with respect to any
$\zeta\in\mbox{span}\{1\,,v_1\,,\cdots\,,v_\D\,,|v|^2\}$, one obtains the local conservation laws with defects
\begin{equation}\label{formal-local}
\begin{aligned}
&\p_t\tilde{\rho}_\eps+\frac{1}{\eps}\di\tilde{\mathrm{u}}_\eps=\frac{1}{\eps}\left\langle\!\left\langle \Gamma'(\Ge)q_\eps\right\rangle\!\right\rangle\,,\\
&\p_t\tilde{\mathrm{u}}_\eps+\frac{1}{\eps}\grad(\tilde{\rho}_\eps+\tilde{\theta}_\eps)+\frac{1}{\eps}\di
\langle \mathrm{A}(v)\gps\rangle=\frac{1}{\eps}\left\langle\!\left\langle v\Gamma'(\Ge)q_\eps\right\rangle\!\right\rangle\,,\\
&\p_t\tilde{\theta}_\eps+\frac{1}{\eps}\frac{2}{\D}\di\tilde{\mathrm{u}}_\eps+\frac{2}{\D}\frac{1}{\eps}
\grad\!\cdot\!\langle \mathrm{B}(v)\gps\rangle=\frac{1}{\eps}\left\langle\!\!\left\langle \left(\tfrac{|v|^2}{\D}-1\right)
\Gamma'(\Ge)q_\eps\right\rangle\!\!\right\rangle\,,
\end{aligned}
\end{equation}
which can be written as
\begin{equation}\label{local}
\p_t\widetilde{U}_\eps+\frac{1}{\eps}\mathcal{A} \widetilde{U}_\eps+\widetilde{Q}_\eps=\widetilde{R}_\eps\,,
\end{equation}
where
\begin{equation}\nonumber
\widetilde{U}_\eps=(\tilde{\rho}_\eps\,,\tilde{\mathrm{u}}_\eps\,,\tilde{\theta}_\eps )= (\langle
\gps\rangle\,,\langle
v\gps\rangle\,,\langle(\tfrac{|v|^2}{\D}-1)\gps\rangle )\,,
\end{equation}
\begin{equation}\nonumber
\widetilde{Q}_\eps=\left(0\,,\tfrac{1}{\eps}\di\langle
\mathrm{A}(v)\gps\rangle\,,\tfrac{1}{\eps}\di\langle
\mathrm{B}(v)\gps\rangle\right)\,,
\end{equation}
and the local conservation defect
\begin{equation}\nonumber
\widetilde{R}_\eps=\frac{1}{\eps}\left\langle\!\!\!\left\langle (1\,,v\,,\tfrac{|v|^2}{\D}-1 )
\Gamma'(\Ge)q_\eps\right\rangle\!\!\!\right\rangle\,.
\end{equation}

Notice that we do not know if $\tilde{\mathrm{u}}_\eps\cdot\mathrm{n}=0$, so $\widetilde{U}_\eps$ is not necessary in the domain of $\AA$ for every $\eps >0$, thus the notation $\mathcal{A} \widetilde{U}_\eps$ in \eqref{local} is not quite rigorous. However we can show that the weak limit of $\tilde{\mathrm{u}}_\eps$, say, $\mathrm{u}$, satisfies $\mathrm{u}\cdot\mathrm{n}=0$ on the boundary, also see \cite{JLM}.  From the local conservation laws with defect \eqref{local}, formally the limit of $\widetilde{U}_\eps$ will be in the null space of the acoustic operator $\AA$. In other words, any weak limits of $(\tilde{\rho}_\eps\,,\tilde{\mathrm{u}}_\eps\,,\tilde{\theta}_\eps)$ will satisfy the incompressibility $\di\tilde{\mathrm{u}}=0$ and Boussinesq relation $\tilde{\rho}+ \tilde{\theta}=0$.

The term $\frac{1}{\eps}\mathcal{A} \widetilde{U}_\eps$ in
\eqref{local} describes the acoustic waves with propagation speed
$\frac{1}{\eps}$. As $\eps$ goes to zero, the sound waves
propagate faster and faster to make the fluid limit singular. To
derive the incompressible fluid equations, a natural way is to
project the local conservation laws \eqref{local} onto
Null$(\mathcal{A} )$ and Null$(\mathcal{A} )^\perp$ respectively. First $\widetilde{U}_\eps$ can be orthogonally decomposed as
\begin{equation}\label{Pi-Pi}
\begin{aligned}
\widetilde{U}_\eps&=\Pi\widetilde{U}_\eps+\Pi^\perp\widetilde{U}_\eps\\
&=\left(\langle(1-\tfrac{|v|^2}{\D+2})\gps\rangle\,,\mathbb{P}\langle v\gps\rangle\,,\langle(\tfrac{|v|^2}{\D+2}-1)
\gps\rangle\right)
\\
&+\left(\langle\tfrac{|v|^2}{\D+2}\gps\rangle\,,\mathbb{Q}\langle
v\gps\rangle\,,\langle\tfrac{2|v|^2}{\D(\D+2)}\gps\rangle\right)\,,
\end{aligned}
\end{equation}
in which we call $\Pi\widetilde{U}_\eps$ and $\Pi^\perp\widetilde{U}_\eps$ the incompressible and acoustic parts of $\widetilde{U}_\eps$ respectively.

By definition of Leray projection in a bounded domain \eqref{Leray-Proj}, the boundary conditions of $\mathbb{P}\langle v\gps\rangle$
and $\mathbb{Q}\langle
v\gps\rangle$ are
\begin{equation}\nonumber
\mathbb{P}\langle v\gps\rangle \cdot \mathrm{n} = 0\quad \mbox{and}
\quad \mathbb{Q}\langle v\gps\rangle \cdot \mathrm{n} =
\tilde{\mathrm{u}}_\eps \cdot \mathrm{n}\quad\mbox{on}\quad \pO\,.
\end{equation}

To derive the weak form of the evolution equations of
$\Pi\widetilde{U}_\eps$, we take the test function $Y$ in
\eqref{weak-BE-G} as special infinitesimal Maxwellian in the
incompressible mode:
\begin{equation}\nonumber
Y^{incom}(x\,,v)=-\chi+w\!\cdot\! v+\chi\left(\tfrac{|v|^2}{2}-\tfrac{\D}{2}\right)\,,
\end{equation}
where $(\chi\,,w)\in
C^\infty(\overline{\Omega}\,,\mathbb{R}^\D\times\mathbb{R})$ with
$\di w=0$ in $\Omega$ and $w\!\cdot\! \mathrm{n}=0$ on $\pO$.
Because $\chi$ and $w$ are independent, the weak form of
\eqref{local} can be written separately as:
\begin{equation}\label{weak-incompressible-u}
\begin{aligned}
&\int\limits_\Omega\mathbb{P}\langle v\gps(t_2)\rangle\cdot w\,\mathrm{d}x-\int\limits_\Omega\mathbb{P}\langle v\gps(t_1)\rangle\cdot w\,\mathrm{d}x\\
&-\frac{1}{\eps}\int^{t_2}_{t_1}\int\limits_\Omega
\langle \mathrm{A}\gps\rangle:\grad w\,\mathrm{d}x\,\mathrm{d}t
+\frac{1}{\sqrt{2\pi}\eps}\int^{t_2}_{t_1}\int\limits_\pO\langle\gamma\gps(w\!\cdot\! v)\rangle_\pO\,\mathrm{d}\sigma_x \,\mathrm{d}t\\
&=\frac{1}{\eps}\int^{t_2}_{t_1}\int\limits_\Omega
w\!\cdot\!\left\langle\!\left\langle
v\Gamma'(\Ge)q_\eps\right\rangle\!\right\rangle\,\mathrm{d}x\,\mathrm{d}t\,,
\end{aligned}
\end{equation}
and
\begin{equation}\label{weak-incompressible-theta}
\begin{aligned}
&\tfrac{\D+2}{2}\int\limits_\Omega\left\langle \left(\tfrac{|v|^2}{\D+2}-1\right)\gps(t_2)\right\rangle
\chi\,\mathrm{d}x-\tfrac{\D+2}{2}\int\limits_\Omega\left\langle \left(\tfrac{|v|^2}{\D+2}-1\right)\gps(t_1)\right\rangle \chi\,\mathrm{d}x\\
&-\frac{1}{\eps}\int^{t_2}_{t_1}\int\limits_\Omega
\langle \mathrm{B}\gps\rangle\cdot\grad \chi\,\mathrm{d}x\,\mathrm{d}t
+\frac{1}{\sqrt{2\pi}\eps}\int^{t_2}_{t_1}\int\limits_\pO\left\langle\chi\left(\tfrac{|v|^2}{\D+2}-1\right)
\gamma\gps\right\rangle_\pO\,\mathrm{d}\sigma_x \,\mathrm{d}t\\
&=\frac{1}{\eps}\int^{t_2}_{t_1}\int\limits_\Omega
\chi\left\langle\!\!\left\langle \left(\tfrac{|v|^2}{\D+2}-1\right)
\Gamma'(\Ge)q_\eps\right\rangle\!\!\right\rangle\,\mathrm{d}x\,\mathrm{d}t\,.
\end{aligned}
\end{equation}


Identities \eqref{weak-incompressible-u} and
\eqref{weak-incompressible-theta} are the local conservation laws in
the incompressible modes. It is the starting point of the proof of
the weak convergence to the incompressible Navier-Stokes equations
with boundary conditions in the Main Theorems \ref{Diri-limit} and
\ref{Navier-limit}. It has been proved in \cite{LM} the convergence
of the interior terms of \eqref{weak-incompressible-u} and
\eqref{weak-incompressible-theta} as $\eps\rightarrow 0$ to recover
the weak form of incompressible Navier-Stokes equations. It is only
left to derive the boundary conditions of the limiting equations.

The strategy to recover the boundary conditions in the limit is
basically the same as \cite{M-S} except some necessary modifications. For the convenience of the readers, and also because we work in more general collision kernels, we briefly go through the proof here. In \cite{LM}, the author proved
that inside the domain $\mathbb{R}^+\times \Omega \times \RD$, the
family of fluctuations $g_\eps$ is relatively compact in
$w\mbox{-}L^1_{loc}(\dd t;w\mbox{-}L^1(\sigma M\dd v\dd x))$, and that every
limit point $g$ has the form \eqref{limit-g}. Lemma 5.1 of
\cite{M-S} showed that the trace of the limit point $\gamma g$
belongs to $L^1_{loc}(\dd
t;L^1(M|v\!\cdot\!\mathrm{n}(x)|\dd\sigma_x))$ and satisfies
\begin{equation}\label{trace-g}
\gamma g=v\!\cdot\!\gamma \mathrm{u}+\left(\tfrac{|v|^2}{2}-\tfrac{\D+2}{2}\right)\gamma\theta\,,
\end{equation}
where $\gamma \mathrm{u}$ and $\gamma\theta$ denote the fluid traces of $\mathrm{u}$
and $\theta$.

We list some key {\em a priori} estimates from  \cite{M-S} on
$\gamma g_\eps$. The first one is from the inside, we generalize
it to the more general collision kernel case considered in this
paper.

\begin{Lemma}\label{Inside}
For all $p>0$, as $\eps\rightarrow 0$,
\begin{equation}\label{Inside-1}
\gamma\gps\rightarrow \gamma g\quad \!\mbox{in}\quad\!
w\mbox{-}L^1_{loc}
(\mathrm{d}t;w\mbox{-}L^1(M(1+|v|^p)|v\!\cdot\!\mathrm{n}(x)|\mathrm{d}v\mathrm{d}\sigma_x
))\,.
\end{equation}
\end{Lemma}
\begin{proof}

The proof is essentially the same as Lemma 5.2 of
\cite{M-S}, except for some new argument  to treat the soft
potential collision kernel case. First, using the function
\begin{equation}\nonumber
\Gamma(Z)=\left(\frac{Z-1}{1+(Z-1)^2}\right)^{5/3}
\end{equation}
in the renormalized formulation \eqref{BE-Gamma} gives
\begin{equation}\label{renorm-2}
(\eps\partial_t+\Divv)\gps^{5/3}=\frac{5}{3}\iint\limits_{\SD\times\RD}\gps^{2/3}q_\eps\left(\frac{2}{N^2_\eps}
-\frac{1}{N_\eps}\right)b(\omega,v-v_1)\mathrm{d}\omega
M_1\,\mathrm{d}v_1\,.
\end{equation}

To estimate the right-hand side in \eqref{renorm-2}, we apply the
classical Young's inequality, namely,
\begin{equation}\nonumber
pz \leq r^*(p) + r(z)\,,
\end{equation}
for every $p$ and $z$ in the domains of $r^*$ and $r$. Here the
function $r$ is defined over $z > -1$ by $r(z) = z \log (1 + z)$ which
is strictly convex, and $r^*$ is the Legendre dual of $r$.
\begin{equation}\label{bound-gamma-g}
\begin{aligned}
\left| \frac{q_\eps}{N_\eps^2} |\gps|^{2/3} \right| &\leq \frac{1}{\eps^4}\Ge
G_{\eps 1}r\left(\frac{\eps^2 q_\eps}{\Ge G_{\eps 1}}\right)+\frac{1}{\eps^4}\Ge G_{\eps 1} r^*\left(\frac{\eps^2\gps^{2/3}}{N^2_\eps}\right)\\
&\leq \frac{1}{\eps^4}\Ge G_{\eps 1} r \left( \frac{\eps^2
q_\eps}{\Ge G_{\eps 1}} \right)+ \Ge G_{\eps 1}
\frac{|\gps|^{4/3}}{N^4_\eps} r^*(1)\,,
\end{aligned}
\end{equation}
The second inequality above used the superquadratic homogeneity of
$r^*$. By the entropy dissipation rate bound, the first term on the
right-hand side of \eqref{bound-gamma-g} is bounded in
$L^1_{loc}(\mathrm{d}t\,, L^1(d\nu \,\mathrm{d}x))$. Since
$N_\eps\geq 1$ and $\Ge\leq \sqrt{2 N_\eps}$, the integral of the
second term can be bounded as follows:
\begin{equation}\label{bound-gamma-2}
\begin{aligned}
&\sqrt{2}\iint\limits_{  \RD \times \RD}|\gps|^{4/3} G_{\eps 1}\frac{\overline{b}(v_1-v)}{a(v_1)a(v)} a_1 M_1\,\mathrm{d}v_1 a M\mathrm{d}v\\
\leq &\sqrt{2}\iint\limits_{  \RD \times
\RD}\frac{|\gps|^{4/3}}{N^3_\eps\sqrt{N_\eps}}\left(
1+\eps|\tilde{g}_{\eps 1}|
\right)\frac{\overline{b}(v_1-v)}{a(v_1)a(v)} a_1 M_1\,\mathrm{d}v_1 a M\mathrm{d}v\\
+ &\sqrt{2}\iint\limits_{ \RD \times \RD}|\gps|^{4/3}\eps |g_{\eps
1}-\tilde{g}_{\eps 1}| \frac{\overline{b}(v_1-v)}{a(v_1)a(v)} a_1
M_1\,\mathrm{d}v_1 a M\mathrm{d}v\,.
\end{aligned}
\end{equation}
Using the {\em assumption 3}, namely \eqref{assum3} and $|\eps \gps|\leq
\frac{1}{2}$, the first term in \eqref{bound-gamma-2} is bounded.
Indeed, it is bounded by
\begin{equation}\nonumber
C\int_{\RD} \frac{|\gps|^{4/3}}{N^3_\eps\sqrt{N_\eps}} aM
\,\mathrm{d} v \leq C \int_{\RD} \frac{g^2_\eps}{\sqrt{N_\eps}}a
M\,\mathrm{d}v \leq C\,.
\end{equation}
The second term in \eqref{bound-gamma-2} is bounded as
\begin{equation}\label{bound-gamma-3}
\begin{aligned}
&\sqrt{2}\iint\limits_{\RD \times \RD}|\eps\gps|^{4/3} \eps^{2/3} \frac{|\eps g_{\eps 1}|}{\sqrt{N_{\eps 1}}}\frac{g^2_{\eps 1}}{\sqrt{N_{\eps 1}}}
\frac{\overline{b}(v_1-v)}{a(v_1)a(v)} a M\mathrm{d}v\, a_1 M_1 \mathrm{d}v_1\\
\leq & \eps^{2/3} 2\left(\frac{1}{2}\right)^{4/3}\int\limits_{\RD}\frac{g^2_{\eps 1}}{\sqrt{N_{\eps 1}}}a_1 M_1 \mathrm{d}v_1\,.
\end{aligned}
\end{equation}

Since the righthand side of \eqref{bound-gamma-3} vanishes as $\eps$ goes to zero, we
deduce that $(\eps\partial_t+\Divv)\tilde{g}^{5/3}_\eps$ is
uniformly bounded in
$L^1_{loc}(\mathrm{d}t\,,L^1(M\mathrm{d}v\mathrm{d}x))$. The rest of
the proof of \eqref{Inside-1} is the same as that of Lemma 5.2 in
\cite{M-S}.

\end{proof}

The next lemma is the {\em a priori} estimate of $\gamma g_\eps$
from the boundary term in the entropy inequality
\eqref{entropy-inequality}. The proof is the same as Lemma 6.1
in \cite{M-S} with some trivial modification. So we just state the
lemma without giving the proof.
\begin{Lemma}\label{Bound}
Define $\gamma_\eps=\gamp g_\eps-\langle\gamp g_\eps\rangle_\pO$ and
\begin{equation}
\gamma_\eps^{(1)}=\gamma_\eps\mathbf{1}_{\gamp G_\eps\leq 2\langle
G_\eps\rangle_{\pO}\leq 4\gamp G_\eps}\,,\quad
\gamma_\eps^{(2)}=\gamma_\eps-\gamma_\eps^{(1)}\,.
\end{equation}
Then each of these is bounded as follows:
\begin{equation}\label{bound-1}
\sqrt{\frac{\ale}{\eps}}\frac{\gamma_\eps^{(1)}}{[1+\eps^2(\gamp
g_\eps)^2]^{\frac{1}{4}}}\quad\mbox{in}\quad
L^2_{loc}(\mathrm{d}t;L^2(M|v\!\cdot\!\mathrm{n}(x)|\mathrm{d}v\mathrm{d}\sigma_x
))\,,
\end{equation}
\begin{equation}\label{bound-2}
\sqrt{\frac{\ale}{\eps}}\frac{\gamma_\eps^{(1)}}{[1+\eps^2\<\gamp
g_\eps\>_\pO^2]^{\frac{1}{4}}}\quad\mbox{in}\quad
L^2_{loc}(\mathrm{d}t;L^2(M|v\!\cdot\!\mathrm{n}(x)|\mathrm{d}v\mathrm{d}\sigma_x
))\,,
\end{equation}
\begin{equation}\label{bound-3}
\frac{\ale}{\eps^2}\gamma_\eps^{(2)}\quad\mbox{in}\quad L^1_{loc}(\mathrm{d}t;L^1(M|v\!\cdot\!\mathrm{n}(x)|\mathrm{d}v\mathrm{d}\sigma_x ))\,.
\end{equation}
\end{Lemma}

Using Lemma \ref{Inside} and Lemma \ref{Bound}, we can prove the
following lemma which describes how to define the renormalized
outgoing mass flux $\mathbf{1}_{\Sigma_+} \rho$.

\begin{Lemma}\label{bc-renorm-g}
Assume that $\frac{\ale}{\sqrt{2\pi}\eps}\rightarrow \chi\in
(0\,,+\infty]$. Then up to the extraction of a subsequence,
\begin{equation}\nonumber
\frac{\gamma_\eps}{1+\eps^2\gamp g_\eps^2}\quad \mbox{and}\quad \frac{\gamma_\eps}{1+\eps^2\< \gamp g_\eps\>^2_{\partial\Omega}}
\end{equation}
converge in $w\mbox{-}L^1_{loc}(\dd t;w\mbox{-}L^1(
M|v\!\cdot\!\mathrm{n}(x)|\dd v\dd x))$ and have the same weak
limit. Moreover, there exists $\rho\in L^1_{loc}(\dd t; L^1(\dd
\sigma_x))$ such that, up to the extraction of a subsequence,
\begin{equation}\nonumber
\frac{\mathbf{1}_{\Sigma_+}\<\gamp
g_\eps\>_{\partial\Omega}}{1+\eps^2\gamp g_\eps^2}\rightarrow
\mathbf{1}_{\Sigma_+} \rho \quad\mbox{in}\!\!\quad
w\mbox{-}L^1_{loc}(\dd t;w\mbox{-}L^1( M|v\!\cdot\!\mathrm{n}(x)|\dd
v\dd x))\,.
\end{equation}
Furthermore,
\begin{equation}\nonumber
\rho=\< \gamp g\>_{\partial\Omega}\,.
\end{equation}
\end{Lemma}
Lemma \ref{bc-renorm-g} is nothing but Lemma 6.2 and Lemma 6.3 in \cite{M-S}.
The proof is basically the same except some trivial modifications
because we use some different renormalizations. Thus we skip the proof here.

Now it is ready to recover the Dirichlet boundary condition. For the
case $\frac{\ale}{\eps}\rightarrow \infty$, from \eqref{bound-1} and
\eqref{bound-3}, we deduce that
\begin{align*}
\frac{\gamma_\eps^{(1)}}{1+\eps^2\< \gamp g_\eps\>^2_{\partial\Omega}} \rightarrow 0&\quad\mbox{strongly in}\!\!\quad
L^2_{loc}(\dd t;L^2( M|v\!\cdot\!\mathrm{n}(x)|\dd v\dd x)),\\
\frac{\gamma_\eps^{(2)}}{1+\eps^2\< \gamp g_\eps\>^2_{\partial\Omega}} \rightarrow 0&\quad\mbox{strongly in}\!\!\quad
L^1_{loc}(\dd t;L^1( M|v\!\cdot\!\mathrm{n}(x)|\dd v\dd x));\\
\end{align*}
hence, we get
\begin{equation}\label{gamm-positive1}
\frac{\gamma_\eps}{1+\eps^2\< \gamp g_\eps\>^2_{\partial\Omega}}
\rightarrow 0\quad\mbox{strongly in}\!\!\quad L^1_{loc}(\dd t;L^1(
M|v\!\cdot\!\mathrm{n}(x)|\dd v\dd x))\,.
\end{equation}
On the other hand,
\begin{equation}\label{gamm-positive2}
\frac{\gamma_\eps}{1+\eps^2\gamp g_\eps^2} = \gamp \gps -
\frac{\mathbf{1}_{\Sigma_+}\<\gamp
g_\eps\>_{\partial\Omega}}{1+\eps^2\gamp g_\eps^2} \rightarrow \gamp
g - \mathbf{1}_{\Sigma_+} \rho\,,
\end{equation}
in $w\mbox{-}L^1_{loc}(\dd t;w\mbox{-}L^1(
M|v\!\cdot\!\mathrm{n}(x)|\dd v\dd x))$. Then \eqref{gamm-positive1}
and \eqref{gamm-positive2} imply that
\begin{equation}\nonumber
\gamp g=\mathbf{1}_{\Sigma_+}\rho
\end{equation}
where $\rho$ depends only on $(t,x)$. Thus, by \eqref{trace-g} we get the Dirichlet boundary condition
\begin{equation}\nonumber
\gamma \mathrm{u}=0\quad \mbox{and}\quad \gamma\theta=0\,.
\end{equation}

Now, we concentrate on the Navier boundary condition case.
Using the previous convergence results, we can take limits in the
conservation laws \eqref{weak-incompressible-u} and
\eqref{weak-incompressible-theta} to get the weak form of the
boundary conditions. In the weak forms \eqref{weak-incompressible-u}
and \eqref{weak-incompressible-theta}, the limits of the interior
terms have been carried in \cite {LM}, which is
stated in the following lemma:
\begin{Lemma}
Assume that $\frac{\ale}{\sqrt{2\pi}\eps}\rightarrow
\chi\in[0,\infty)$, then up to the extraction of a sequence,
$\mathbb{P}\< v\gps\>$ and $\< (\frac{|v|^2}{\D+2}-1)\gps\>$
converge to $\mathrm{u}$ and $\theta$ in $C([0,\infty); w\mbox{-}L^1(\dd x))$
such that, for all $w\in C^\infty(\overline{\Omega}; \RD)$ with
$\DIV w=0$ in $\Omega$ and $w\!\cdot\! \mathrm{n}=0$ on $\pO$, for
all $\chi\in C^\infty(\overline{\Omega}; \mathbb{R})$, and for all
$t_1\,, t_2 >0$,
\begin{equation}\label{Navier-u}
\begin{aligned}
&\int_\Omega \mathrm{u}(t_2)\!\cdot\! w\,\dd x-\int_\Omega \mathrm{u}(t_1)\!\cdot\! w\,\dd x-\int^{t_2}_{t_1}\int_\Omega
\sum\limits_{i,j} \mathrm{u}_i \mathrm{u}_j\partial_i w_j\,\dd x\dd t\\
&+ \nu  \int^{t_2}_{t_1}\int_\Omega \sum\limits_{i,j} (\partial_i \mathrm{u}_j+\partial_j \mathrm{u}_i)\partial_i w_j\,\dd x\dd t\\
&=-\lim\limits_{\eps\rightarrow
0}\frac{\ale}{\sqrt{2\pi}\eps}\int^{t_2}_{t_1}\int_{\partial\Omega}
\left\< \frac{\gamma^{(1)}_\eps(w\!\cdot\! v)\mathbf{1}_{|v|^2\leq
20|\log \eps|}}{(1+\eps^2\gamp g^2_\eps)(1+\eps^2 \gamp
\hat{g}_\eps^2)} \right\>_{\partial\Omega}\,\dd \sigma_x\dd t\,,
\end{aligned}
\end{equation}
\begin{equation}\label{Navier-theta}
\begin{aligned}
&\int_\Omega \theta(t_2)\!\cdot\! \chi\,\dd x-\int_\Omega \theta(t_1)\!\cdot\! \chi\,\dd x-\int^{t_2}_{t_1}
\int_\Omega \theta \mathrm{u}\!\cdot\!\grad \chi\,\dd x\dd t
+ \tfrac{2}{\D+2}\kappa \int^{t_2}_{t_1}\int_\Omega \grad\theta\!\cdot\!\grad \chi\,\dd x\dd t\\
&=-\lim\limits_{\eps\rightarrow
0}\frac{\ale}{\sqrt{2\pi}\eps}\int^{t_2}_{t_1}\int_{\partial\Omega}\left\<
\frac{\gamma^{(1)}_\eps }{(1+\eps^2\gamp g^2_\eps)(1+\eps^2 \gamp
\hat{g}_\eps^2)}
\chi\left(\tfrac{|v|^2}{\D+2}-1\right)\mathbf{1}_{|v|^2\leq 20|\log
\eps|}\right\>_{\partial\Omega}\,\dd \sigma_x\dd t\,.
\end{aligned}
\end{equation}
where $\gamp \hat{g}_\eps=(1-\ale)\gamp g_\eps+\ale \< \gamp g_\eps\>_{\pO}$.
\end{Lemma}
\begin{proof}
It is the analogue of Lemma 7.1 and Lemma 7.2 of \cite{M-S} and the
idea of the proof is the same. Denote by $Y_\eps$ the test function
$(w\!\cdot\! v)\mathbf{1}_{|v|^2\leq 20|\log \eps|}$ or
$\chi(\tfrac{|v|^2}{\D+2}-1)\mathbf{1}_{|v|^2\leq 20|\log \eps|}$.
Then $Y_\eps$ has the property: $Y_\eps= L Y_\eps$, recalling $L$ is the local reflection operator defined in \eqref{reflect-L}. From
\eqref{weak-b-BE},  the renormalized form of the Maxwell boundary
condition reads
\begin{equation}\label{renormBC}
\gamma_-\gps=(1-\ale)\frac{L\gpg}{1+\eps^2(L\gamma_+\hat{g}_\eps)^2}+\ale\frac{\gpgb}{1+\eps^2(L\gamma_+\hat{g}_\eps)^2}\,,
\end{equation}
where
\begin{equation}\nonumber
\begin{aligned}
\gamma_+\hat{g}_\eps&=(1-\ale)\gamma_+g_\eps+\ale\langle\gamma_+g_\eps\rangle_\pO\,,\\
&=\gpg-\ale\gamma_\eps\,.
\end{aligned}
\end{equation}
Then
\begin{equation}\label{boundary-term}
\begin{aligned}
\frac{1}{\eps}\int_{\pO}\< \gamma \gps Y_\eps\>_{\pO}\dd \sigma_x &=\frac{1}{\eps}\int_{\pO}
\left\< \frac{\eps^2\gamp g_\eps(\gamp \hat{g}^2_\eps-\gamp g^2_\eps)}{(\gammag)(\gammagh)} Y_\eps \mathbf{1}_{\Sigma_+}\right\>_{\!\!\pO}\dd \sigma_x\\
&+ \frac{\ale}{\eps}\int_{\pO} \left\< \frac{\gamma_\eps}{\gammagh} Y_\eps \mathbf{1}_{\Sigma_+}\right\>_{\!\!\pO}\dd \sigma_x\\
&=\frac{\ale}{\eps}\int_{\pO}\left\<  \frac{\gamma^{(1)}_\eps +\gamma^{(2)}_\eps}{(\gammag)(\gammagh)}
Y_\eps \mathbf{1}_{\Sigma_+}\right\>_{\!\!\pO}\dd \sigma_x\\
&-\frac{\ale}{\eps}\int_{\pO}\left\< \frac{(\gamma^{(1)}_\eps+
\gamma^{(2)}_\eps) \eps^2 (\gamp g_\eps \gamp
\hat{g}_\eps)}{(\gammag)(\gammagh)}   Y_\eps
\mathbf{1}_{\Sigma_+}\right\>_{\!\!\pO}\dd \sigma_x\,.
\end{aligned}
\end{equation}

By \eqref{bound-3},
\begin{equation}\label{gamma-2}
\begin{aligned}
&\int^{t_2}_{t_1} \left\vert \frac{\ale}{\eps}\int_{\pO}\left\<  \frac{\gamma^{(2)}_\eps}{(\gammag)(\gammagh)}
Y_\eps \mathbf{1}_{\Sigma_+}\right\>_{\!\!\pO}\dd \sigma_x \right\vert\dd t\\
&\leq C \eps\left\| \frac{Y_\eps}{(\gammag)(\gammagh)}\right\|_\infty \leq C \eps|\log \eps|\,.
\end{aligned}
\end{equation}
The $\gamma^{(2)}_\eps$ part in the last term of \eqref{boundary-term} can be estimated as
\eqref{gamma-2}. For the $\gamma^{(1)}_\eps$ part, from \eqref{bound-1},
\begin{equation}\label{product-1}
\sqrt{\frac{\ale}{\eps}}\frac{\gamma^{(1)}_\eps}{\sqrt{\gammag}}\quad\mbox{is relatively compact in}
\!\!\quad w\mbox{-}L^1_{loc}(\dd t; w\mbox{-} L^1(M|v\!\cdot\!\mathrm{n}|\dd v\dd \sigma_x))
\end{equation}
Use the fact that
\begin{equation}\label{product-2}
\sqrt{\frac{\ale}{\eps}}\frac{\eps \gamp g_\eps \eps \gamp \hat{g}_\eps}{\sqrt{\gammag}(\gammagh)}
\end{equation}
is bounded in $L^\infty$ and goes to 0 a.e. Then by the Product
Limit Theorem of \cite{BGL2}, the product of \eqref{product-1} and
\eqref{product-2} goes to 0 in $L^1_{loc}(\dd t)$ as
$\eps\rightarrow 0$. Thus we finish the proof of the lemma.
\end{proof}

Now, it is ready to recover the Navier boundary condition by taking
limit in the last terms in \eqref{Navier-u} and
\eqref{Navier-theta}. As in \cite{M-S}, we can deduce that
\begin{equation}\nonumber
\frac{\ale}{\sqrt{2\pi}\eps}\left\<
\frac{\gamma^{(1)}_\eps(w\!\cdot\! v)\mathbf{1}_{|v|^2\leq 20|\log
\eps|}}{(1+\eps^2\gamp g^2_\eps)(1+\eps^2 \gamp \hat{g}_\eps^2)}
\right\>_{\!\!\pO}\rightarrow \lambda\<(\gamp
g-\mathbf{1}_{\Sigma_+}\<\gamp g\>_{\pO})(w\!\cdot\! v)\>_{\pO}\,,
\end{equation}
\begin{equation}\nonumber
\begin{aligned}
\frac{\ale}{\sqrt{2\pi}\eps}& \left\< \frac{\gamma^{(1)}_\eps }{(1+\eps^2\gamp g^2_\eps)(1+\eps^2 \gamp \hat{g}_\eps^2)}
\chi\left(\tfrac{|v|^2}{\D+2}-1\right)\mathbf{1}_{|v|^2\leq 20|\log \eps|}\right\>_{\partial\Omega}\\
&\rightarrow \lambda \<(\gamp g-\mathbf{1}_{\Sigma_+}\<\gamp g\>_{\pO}) \chi(\tfrac{|v|^2}{\D+2}-1) \>_\pO
\end{aligned}
\end{equation}
in $w\mbox{-} L^1_{loc}(\dd t; w\mbox{-} L^1(\dd \sigma_x))$. Use
\eqref{trace-g}, we finally prove the weak form of the
incompressible Navier-Stokes equations with Navier boundary
conditions:
\begin{equation}\nonumber
\begin{aligned}
&\int_\Omega \mathrm{u}(t_2)\!\cdot\! w\,\dd x-\int_\Omega \mathrm{u}(t_1)\!\cdot\! w\,\dd x-\int^{t_2}_{t_1}\int_\Omega
\sum\limits_{i,j} \mathrm{u}_i \mathrm{u}_j\partial_i w_j\,\dd x\dd t\\
&+ \nu  \int^{t_2}_{t_1}\int_\Omega \sum\limits_{i,j} (\partial_i \mathrm{u}_j+\partial_j \mathrm{u}_i)\partial_i w_j\,\dd x\dd t\\
&=\lambda \int^{t_2}_{t_1}\int_\pO \gamma \mathrm{u}\!\cdot\! w \dd \sigma_x \dd t\,,
\end{aligned}
\end{equation}
\begin{equation}\nonumber
\begin{aligned}
&\int_\Omega \theta(t_2)\!\cdot\! \chi\,\dd x-\int_\Omega \theta(t_1)\!\cdot\! \chi\,\dd x-\int^{t_2}_{t_1}
\int_\Omega \theta \mathrm{u}\!\cdot\!\grad \chi\,\dd x\dd t\\
&+ \tfrac{2}{\D+2}\kappa \int^{t_2}_{t_1}\int_\Omega \grad\theta\!\cdot\!\grad \chi\,\dd x\dd t
=\tfrac{\D+1}{\D+2}\alpha\int^{t_2}_{t_1}\int_\pO \gamma \theta \chi\sigma_x\dd t\,.
\end{aligned}
\end{equation}
Thus we finish the proof of the weak convergence results in the Theorem \ref{Diri-limit} and Theorem \ref{Navier-limit}.

\section{Proof of the Strong Convergence in Theorem \ref{Diri-limit}}

In the previous section, we proved that the incompressible part of the fluid moments $\widetilde{U}_\eps$, i.e.
$\Pi\widetilde{U}_\eps$ converges only {\em weakly} to solutions of
the incompressible NSF equations. This weak convergence is caused by the persistence of fast acoustic part
$\Pi^\perp\widetilde{U}_\eps$, as in the periodic domain
\cite{LM}. If $\Pi^\perp\widetilde{U}_\eps$ vanishes in some strong sense
as $\eps$ goes to zero, we can improve the convergence of $\Pi\widetilde{U}_\eps$
from weak to strong. The main novelty of this paper is to prove that in the bounded domain $\Omega$, when $\alpha_\eps = O(\sqrt{\eps})$, the acoustic part will
be damped {\em instantaneously}. This damping effect comes from the kinetic-fluid coupled boundary layers. More precisely, we have the following proposition:
\begin{Prop}\label{acoustic-vanish}
Let $\Pi^\perp\widetilde{U}_\eps$ be  defined as \eqref{Pi-Pi}. If
$\alpha_\eps = O(\sqrt{\eps})$,  then
\begin{equation}\nonumber
\Pi^\perp\widetilde{U}_\eps\rightarrow 0\quad\mbox{in}\!\!\quad L^2_{loc}(\dd t; L^2(\dd x))\,,
\end{equation}
as $\eps\rightarrow 0$.
\end{Prop}
This proposition is also true for $\alpha_\eps = O(\eps^\beta)$, $0\leq \beta < frac{1}{2}$ and $\frac{1}{2} < \beta < 1$. These cases will be treated in a separate paper.

Now we apply Proposition \ref{acoustic-vanish} to prove the Main Theorem \ref{Diri-limit}, and leave its proof to the next subsection.

\subsection{Strong Convergence in $L^1$: Proof of Theorem \ref{Diri-limit}} We first show that we can improve the relative compactness of the family of
fluctuations $g_\eps$ from weak to strong in $L^1_{loc}(\dd
t;L^1(\sigma M\dd v\dd x))$. Indeed, $g_\eps$ can be decomposed as
\begin{equation}\nonumber
\begin{aligned}
g_\eps=&\mathcal{P}\gps+{\mathcal{P}}^\perp \gps+\frac{\eps^2 g^3_\eps}{N_\eps}\\
=&v\!\cdot\!\mathbb{P}\tilde{\mathrm{u}}_\eps+\left(\tfrac{\D}{\D+2}\tilde{\theta}_\eps-\tfrac{2}{\D+2}\tilde{\rho}_\eps\right)\left(
\tfrac{|v|^2}{2}-\tfrac{\D+2}{2}\right)+v\!\cdot\!{\mathbb{Q}}  \tilde{\mathrm{u}}_\eps+\tfrac{|v|^2}{\D+2}\left(\tilde{\rho}_\eps+\tilde{\theta}_\eps\right)\\
&+{\mathcal{P}}^\perp \gps+\frac{\eps^2
g_\eps}{\sqrt{N_\eps}}\frac{g^2_\eps}{\sqrt{N_\eps}}\,,
\end{aligned}
\end{equation}
where $\PP$ is the projection to Null$(\LL)$ defined in \eqref{projection-p}, $\mathbb{P}$ is the Leray projection, and $\mathbb{Q}=\mathbb{I}-\mathbb{P}$.

It has been proved in \cite{LM} that ${\mathcal{P}}^\perp \gps\rightarrow 0$ in $L^2_{loc}(\dd t; L^2(aM\dd v\dd x))$, (see (6.41) in \cite{LM}). We can also show that
\begin{equation}\label{L2-strong}
\mathbb{P}\tilde{\mathrm{u}}_\eps\rightarrow \mathrm{u}\,,\quad
\tfrac{\D}{\D+2}\tilde{\theta}_\eps-\tfrac{2}{\D+2}\tilde{\rho}_\eps\rightarrow
\theta\,,\quad \mbox{in}\!\!\quad L^2_{loc}(\dd t; L^2(\dd x))\,.
\end{equation}
Indeed, this convergence is justified in Lemma 5.6 in \cite{Go-Sai04}. Although the renormalization and decomposition of $g_\eps$ are different in \cite{Go-Sai04} and the current paper, the proof of the convergence \eqref{L2-strong} can follow the argument in the proof of Lemma 5.6 in \cite{Go-Sai04}. Furthermore, the Proposition \ref{acoustic-vanish} yields that
\begin{equation}\nonumber
v\!\cdot\!{\mathbb{P}}^\perp
\tilde{\mathrm{u}}_\eps+\tfrac{|v|^2}{\D+2}\left(\tilde{\rho}_\eps+\tilde{\theta}_\eps\right)
\rightarrow 0 \quad\mbox{in}\!\!\quad L^2_{loc}(\dd t;L^2(M\dd v\dd
x))\,.
\end{equation}
Thus $\mathcal{P}\gps\rightarrow g=v\!\cdot\!
\mathrm{u}+\left(\tfrac{1}{2}|v|^2-\tfrac{\D+2}{2}\right)\theta$ in
$L^2_{loc}(\dd t; L^2(M \dd v\dd x))$, as $\eps\rightarrow 0$. The
key nonlinear estimate in \cite{BGL2} claims that
\begin{equation}\nonumber
\sigma \frac{g^2_\eps}{\sqrt{N_\eps}}=O(|\log \eps|)\quad\mbox{in}\!\!\quad L^\infty(\dd t;  L^1(aM\dd v\dd x))\,.
\end{equation}
It is easy to see that $\frac{\eps\gps}{\sqrt{N_\eps}}$ is bounded, hence
\begin{equation}\label{g-sharp}
\frac{\eps^2 g^3_\eps}{N_\eps} \rightarrow 0\quad\mbox{in}\!\!\quad L^1_{loc}(\dd t; L^1(\sigma M\dd v\dd x))\,.
\end{equation}
We deduce  that $g_\eps$ is
relatively compact in $L^1_{loc}(\dd t;L^1(\sigma M\dd v\dd x))$
and that  every limit $g$ has the form \eqref{limit-g}, combining  the above  estimates.

Next, we can also improve the convergence of the moments of
$g_\eps$. In \cite{LM}, it was proved that the incompressible
part $(\mathbb{P}\< v g_\eps\>\,,\< (\frac{1}{\D+2}|v|^2-1)
g_\eps\>)$ converge to $(\mathrm{u}\,,\theta)$ in
$C([0,\infty);w\mbox{-}L^1(\dd x))$. We also have
$(\mathbb{P}\< v g_\eps\>\,,\< (\frac{1}{\D+2}|v|^2-1)
g_\eps\>)$ converge to $(\mathrm{u}\,,\theta)$ in $L^2_{loc}(\dd t; L^2(\dd x))$.
Now, from Proposition \ref{acoustic-vanish}, we know that  the
acoustic part ${\mathbb{Q}}  \< v \gps\>$ and $\<
(\frac{1}{\D+2}|v|^2 \gps\>$ converge strongly to $0$ in
$L^2_{loc}(\dd t; L^2(\dd x))$. So combining  this with
\eqref{g-sharp}, we get
\begin{equation}\nonumber
\begin{aligned}
\< vg_{\eps}\>\rightarrow \mathrm{u}\quad&\mbox{in}\!\!\quad L^1_{loc}(\dd t;L^1(\dd x;\RD))\cap C([0,\infty);w\mbox{-}L^1(\dd x;\RD)) \,,\\
\<(\tfrac{1}{\D}|v|^2-1)g_{\eps}\>\rightarrow
\theta\quad&\hbox{in}\!\!\quad L^1_{loc}(\dd t;L^1(\dd
x;\mathbb{R}))\cap C([0,\infty);w\mbox{-}L^1(\dd x;\mathbb{R}))\,.
\end{aligned}
\end{equation}

Furthermore, since now we have $\tilde{\mathrm{u}}_\eps\rightarrow \mathrm{u}$ and
$\tilde{\theta}\rightarrow \theta$ in $L^2_{loc}(\dd t; L^2(\dd
x))$, we can improve the Quadratic Limit Theorem 13.1 in \cite{LM}
to
\begin{equation}\label{quadratic}
\begin{aligned}
\tilde{\mathrm{u}}_\eps\otimes \tilde{\mathrm{u}}_\eps\rightarrow \mathrm{u}\otimes \mathrm{u}\,,\quad
\tilde{\theta}_\eps \tilde{\mathrm{u}}_\eps\rightarrow \mathrm{u}\theta\,, \quad
\tilde{\theta}^2_\eps\rightarrow \theta^2\quad\mbox{in}\!\!\quad
L^1_{loc}(\dd t; L^1(\dd x))\,,
\end{aligned}
\end{equation}
as $\eps\rightarrow 0$.

Let $s\in (0,\infty]$ be from the assumed bound \eqref{assum3} on
$b$. Let $p=2+\frac{1}{s-1}$, so that $p=2$ when $s=\infty$. Let
$\hat{\xi}\in L^p(aM\dd v)$ be such that $\mathcal{P}\hat{\xi}=0$ and
set $\xi=\mathcal{L}\hat{\xi}$, hence,
\begin{equation}\nonumber
\frac{1}{\eps}\< \xi \gps\>=\frac{1}{\eps}\< \xi \mathcal{P}^\perp \gps\>
=\< \hat{\xi}\mathcal{Q}(\gps\,,\gps)\> -\<\!\< \hat{\xi}\tilde{q}_\eps\>\!\>+ \<\!\< \hat{\xi}T_\eps\>\!\>\,.
\end{equation}
We know from in \cite{LM} that
\begin{equation}\label{TT}
\<\!\< \hat{\xi}T_\eps\>\!\>\rightarrow 0\quad\mbox{in}\!\!\quad L^1_{loc}(\dd t; L^1(\dd x))\,,
\end{equation}
and
\begin{equation}\label{qq}
\<\!\< \hat{\xi}\tilde{q}_\eps\>\!\>\rightarrow \<
\xi\widehat{\mathrm{A}}\>:\grad u +\<
\xi\widehat{\mathrm{B}}\>\!\cdot\! \grad
\theta\quad\mbox{in}\!\!\quad w\mbox{-}L^2_{loc}(\dd t;
w\mbox{-}L^2(\dd x))\,.
\end{equation}
Note that
\begin{equation}\nonumber
\begin{aligned}
\< \hat{\xi}\mathcal{Q}(\gps\,,\gps)\>=\< \hat{\xi}\mathcal{Q}(\mathcal{P}\gps\,,\mathcal{P}\gps)\>&
+2\< \hat{\xi}\mathcal{Q}(\mathcal{P}\gps\,,\mathcal{P}^\perp\gps)\>\\
&+\< \hat{\xi}\mathcal{Q}(\mathcal{P}^\perp\gps\,,\mathcal{P}\gps)\>\,.
\end{aligned}
\end{equation}
It is easy to show that the last two terms above vanish as $\eps\rightarrow 0$. For the first term,
\begin{equation}\label{QQ}
\begin{aligned}
&\< \hat{\xi}\mathcal{Q}(\mathcal{P}\gps\,,\mathcal{P}\gps)\> =\tfrac{1}{2}\< \xi \mathcal{P}^\perp(\mathcal{P}\gps)^2\>\\
=& \tfrac{1}{2}\< \xi \mathrm{A}\>: (\tilde{\mathrm{u}}_\eps\otimes
\tilde{\mathrm{u}}_\eps)+ \< \xi \mathrm{B}\>\!\cdot\! \tilde{\mathrm{u}}_\eps
\tilde{\theta}_\eps + \tfrac{1}{2}\< \xi
\mathrm{C}\>\tilde{\theta}^2_\eps \,.
\end{aligned}
\end{equation}
Applying the quadratic limit \eqref{quadratic}, \eqref{QQ} can be
taken limit in $L^1_{loc}(\dd t; L^1(\dd x))$ strongly. Combining
with convergence \eqref{TT} and \eqref{qq}, we get
\begin{equation}\nonumber
\frac{1}{\eps}\< \xi \mathcal{P}^\perp \gps\> \rightarrow \left\<
\xi \left(\tfrac{1}{2}\mathrm{A}: \mathrm{u}\otimes \mathrm{u}+\mathrm{B}\!\cdot\!
\mathrm{u}\theta +\tfrac{1}{2}\mathrm{C}\theta^2-\widehat{\mathrm{A}}:\grad
\mathrm{u}-\widehat{\mathrm{B}}\!\cdot\!\grad\theta\right)\right\>
\end{equation}
in $w\mbox{-}L^1_{loc}(\dd t;w\mbox{-} L^1(\dd x))$. Since
$g_\eps-\gps\rightarrow 0$ in $L^\infty(\dd t; L^1(\sigma M\dd v\dd
x))$, the convergence above implies \eqref{P-Perp}. Thus we finish
the proof of the Main Theorem \ref{Diri-limit}.


\subsection{Proof of Proposition \ref{acoustic-vanish}}
We will reduce the proof of the Proposition \ref{acoustic-vanish} to
show that the projection of $\widetilde{U}_\eps $ on each {\em
fixed} acoustic mode goes to zero in $L^2_{loc}(\dd t; L^2(\dd x))$.
We  know that $\Pi^\perp\widetilde{U}_\eps$ is uniformly bounded in
$L^\infty(\dd t; L^2(\dd x))$, so it can be represented as
\begin{equation}\label{orth-coe}
\begin{aligned}
\Pi^\perp\widetilde{U}_\eps&=\sum\limits_{k\in\mathbb{N}}
\<\widetilde{U}_\eps,U^{+,k}\>_{\mathbb{H}} U^{+,k}+\< \widetilde{U}_\eps,U^{-,k}\>_{\mathbb{H}} U^{-,k}\\
&=\tfrac{\D+2}{2\D}\sum\limits_{k\in\mathbb{N}}\begin{pmatrix}
   \tfrac{2\D}{(\D+2)^2}\int_\Omega\< |v|^2\gps\>\overline{\Psi^k}\,\mathrm{d}x\!\!\!\quad \Psi^k \\
   2 \int_\Omega \<
v\gps\>\!\cdot\!\frac{\grad\overline{\Psi^k }}{i\lambda^k}\,\mathrm{d}x\!\!\!\quad\frac{\grad\Psi^k }{i\lambda^k}\\
   \tfrac{4}{(\D+2)^2}\int_\Omega\< |v|^2\gps\>\overline{\Psi^k}\,\mathrm{d}x \!\!\!\quad\Psi^k
  \end{pmatrix}\,,
\end{aligned}
\end{equation}
recalling that the inner product $\<\cdot\,, \cdot\>_{\mathbb{H}}$ is defined in \eqref{innerproduct}, $U^{+,k}$ and $U^{-,k}$ are defined in \eqref{eigen-A}.

The above summation includes infinitely many terms. To reduce the
problem to a finite number of modes, we need some regularity in $x$
of $\< v\gps\>$ and $\<|v|^2\gps\>$. The tool adapted to
investigating this property is the velocity averaging theorem given
in \cite{GLPS} and the improvement to $L^1$ averaging in \cite{GS-1}.

Following the similar argument in the proof of Proposition 11.2 in \cite{LM}, and apply to \eqref{scale-BE}, we can show that
for each $\zeta\in \mbox{Span}\{1\,,v\,, |v|^2\}$ and $T > 0$, there exists a function $\eta: \mathbb{R}_+\rightarrow \mathbb{R}_+$ such that $\lim\limits_{z\rightarrow 0^+}\eta(z)=0$,
\begin{equation}\label{L2-compact}
\|\< \zeta\gps(t,x+y,v)- \zeta\gps(t,x,v)\>\|_{L^2([0,T]\times\Omega)} \leq \eta(|y|)\,,
\end{equation}
for every $y\in \Omega$ such that $|y|\leq 1$, uniformly in $\eps\in [0,1]$. From the classical criterion of compactness in $L^2$, $\< v\gps\>$ and $|v|^2\gps$ are relatively compact in $L^2_{loc}(\dd t\,, L^2(\dd x))$ which implies that
\begin{equation}\label{finite-k}
\sum\limits_{k>N}\int^T_0\left |
\<\widetilde{U}_\eps,U^{\tau,k}\>_{\mathbb{H}}\right
|^2\,\mathrm{d}t \leq
C_N\|\Pi^\perp\widetilde{U}_\eps\|_{L^2([t_1,t_2];L^2(\mathrm{d}x))}\rightarrow
0\quad\mbox{as}\quad N\rightarrow \infty\,,
\end{equation}
recalling from \eqref{L2-compact} that $C_N \rightarrow 0$ as
$N\rightarrow \infty$. \eqref{finite-k} implies that, to show
$\Pi^\perp\widetilde{U}_\eps\rightarrow 0$ strongly in
$L^2_{loc}(\mathrm{d}t\,,L^2(\mathrm{d}x))$, we need only to prove
that $\<\widetilde{U}_\eps,U^{\tau,k}\>_{\mathbb{H}}$
converges strongly to $0$ in $L^2(0,T)$ for any fixed acoustic mode
$k$. Furthermore, the relation
\begin{equation}\nonumber
\<\widetilde{U}_\eps,U^{\tau,k}\>_{\mathbb{H}}=\int\limits_\Omega\left\langle\gps\,,g^{\tau,k,\Int}_0\right\rangle\,\mathrm{d}x
\end{equation}
implies that the proof of Proposition \ref{acoustic-vanish} is
reduced to showing that :
\begin{Prop}\label{perp-U}
Assume that $\alpha_\eps = O(\sqrt{\eps})$  and
let $\gps$ be the renormalized fluctuation defined in
\eqref{renorm-g}, satisfying the scaled Boltzmann equation
\eqref{scale-BE}, and $g^{\tau,k,int}_0$ ($\tau$ is + or -) be the infinitesimal
Maxwellian of acoustic mode $k\geq 1$:
\begin{equation}\nonumber
g^{\tau,k,int}_0=\tfrac{\D}{\D+2}\Psi^k +\tfrac{\grad
\Psi^k }{\tau i\lambda^{k}}\cdot v
+\tfrac{2}{\D+2}\Psi^k (\tfrac{|v|^2}{2}-\tfrac{\D}{2})\,.
\end{equation}
Then, for any fixed mode $k$,
\begin{equation}\nonumber
\int_\Omega\left\langle\gps\,,
g^{\tau,k,\Int}_0\right\rangle\,\mathrm{d}x\rightarrow 0\,\quad
\mbox{in}\quad L^2(0,T)\,,\quad\mbox{as}\quad \eps\rightarrow 0\,.
\end{equation}
\end{Prop}


\begin{proof}
We start from the weak formulation of the rescaled Boltzmann
equation \eqref{weak-BE-G} with the renormalization $\Gamma$ defined in \eqref{renorm-Gamma} and the test function $Y$ taken to be
the approximate eigenfunctions of $\mathcal{L}_\eps$
constructed in Proposition \ref{main-prop} to the  order $N=4$,
namely $Y =g^{\tau,k}_{\eps,4}$ :
\begin{equation}\label{weak-BE-com}
\begin{aligned}
&\int_\Omega\langle\gps(t_2)g^{\tau,k}_{\eps,4}\rangle\,\mathrm{d}x-\int_\Omega\langle\gps(t_1)g^{\tau,k}_{\eps,4}\rangle\,\mathrm{d}x\\
&+\frac{1}{\eps}\int^{t_2}_{t_1}\int_\Omega \langle\gps
\mathcal{L}_\eps
g^{\tau,k}_{\eps,4}\rangle\,\mathrm{d}x\,\mathrm{d}t
+\frac{1}{\eps}\int^{t_2}_{t_1}\int_\pO\langle\gamma\gps \gamma
g^{\tau,k}_{\eps,4} (v\!\cdot\!
\mathrm{n})\rangle\,\mathrm{d}\sigma_x \,\mathrm{d}t\\
&=\frac{1}{\eps}\int^{t_2}_{t_1}\int_\Omega\langle\!\langle R_\eps
g^{\tau,k}_{\eps,4}\rangle\!\rangle\,\mathrm{d}x\,\mathrm{d}t\,,
\end{aligned}
\end{equation}
where
\begin{equation}\nonumber
R_\eps=
\Gamma'(G_\eps) q_\eps+\frac{1}{\eps} \left(
\frac{g_{\eps 1}}{N_{\eps 1}}+\frac{g_\eps}{N_\eps}-\frac{g'_{\eps 1}}{N'_{\eps 1}}-\frac{g'_\eps}{N'_\eps}\right)\,.
\end{equation}

Define
\begin{equation}\nonumber
\tilde{b}^{\tau,k}_{\eps}(t)=\int_\Omega\langle\gps(t)
g^{\tau,k}_{\eps,4} \rangle\,\mathrm{d}x\,.
\end{equation}
Then from \eqref{weak-BE-com} $ \tilde{b}^{\tau,k}_{\eps}(t)$ satisfies
\begin{equation}\label{ode-b}
\tilde{b}^{\tau,k}_{\eps}(t_2)-\tilde{b}^{\tau,k}_{\eps}(t_1)-\frac{1}{\eps}\overline{i\lambda^{\tau,k}_{\eps,4}}
\int^{t_2}_{t_1}\tilde{b}^{\tau,k}_{\eps}(t)\,\mathrm{d}t=
\int^{t_2}_{t_1}c^{\tau,k}_{\eps}(t)\,\mathrm{d}t\,,
\end{equation}
where $c^{\tau,k}_{\eps}(t)$ is:
\begin{equation}\label{remainder-c}
\begin{aligned}
c^{\tau,k}_{\eps}(t)=&-\frac{1}{\eps}\int_\Omega\langle \gps(t)
R^{\tau,k}_{\eps,4}\rangle\,\mathrm{d}x-\frac{1}{\eps}
\int_{\partial\Omega}\<\gamma \gps \gamma g^{\tau,k}_{\eps,4}(v\!\cdot\!\mathrm{n})\>\,\mathrm{d}\sigma_x \\
&+\frac{1}{\eps}\int_\Omega\<\!\< R_\eps g^{\tau,k}_{\eps,4}
\>\!\>\,\mathrm{d}x\,.
\end{aligned}
\end{equation}

We claim that the boundary contribution in \eqref{remainder-c} is zero as $\eps\rightarrow 0$, i.e.
\begin{Lemma}\label{Boundary-Vanish}
Let $g^{\tau,k}_{\eps,4}$ be the  approximate eigenfunction of
$\LL_\eps$ constructed in Proposition \ref{main-prop}. Then,
\begin{equation}\label{Boundary-V-estimate}
\frac{1}{\eps}\int\limits_\pO\langle\gamma\gps\gamma
g^{\tau,k}_{\eps,4}(v\!\cdot\!\mathrm{n})\rangle\,\mathrm{d}\sigma_x = \Gamma^{\tau,k}_1 + \Gamma^{\tau,k}_2\,,
\end{equation}
where $\Gamma^{\tau,k}_1$ is bounded in
$L^p_{loc}(\mathrm{d}t)$ for $p >1$, and $\Gamma^{\tau,k}_2$ vanishes in
$L^1_{loc}(\mathrm{d}t)$ as $\eps \rightarrow 0$.
\end{Lemma}
We leave the proof of Lemma \ref{Boundary-Vanish} to the section
\ref{last-section}.

\subsection{Estimates of $c^{\tau,k}_{\eps}$}
We will decompose $c^{\tau,k}_{\eps}(t)$ into two parts: one is
vanishing in $L^1_{loc}(\mathrm{d}t)$, the other is bounded in
$L^p_{loc}(\mathrm{d}t)$ for some $p>1$. First, taking $N=4$,
$r=p=2$ in \eqref{error1} and noticing the
$L^\infty_{loc}(\mathrm{d}t\,,L^2(aM\mathrm{d}v\mathrm{d}x))$
boundedness of $\gps$, we have the estimate of the first term in
\eqref{remainder-c}:
\begin{equation}\nonumber
\begin{aligned}
\Big\vert \frac{1}{\eps}\int_\Omega \langle \gps(t)
R^{\tau,k}_{\eps,4}\rangle \,\mathrm{d}x\Big\vert
&\leq \frac{1}{\eps} \| R^{\tau,k}_{\eps,4} \|_{L^2(\frac{1}{a}M\mathrm{d}v\mathrm{d}x)} \| \gps\|_{L^2(aM\mathrm{d}v\mathrm{d}x)}\\
& \leq C \sqrt{\eps}\,.
\end{aligned}
\end{equation}
Second, Lemma \ref{Boundary-Vanish} implies that one part of the boundary term
in \eqref{remainder-c}, namely $\Gamma^{\tau,k}_2$ in \eqref{Boundary-V-estimate} will be vanishing in $L^1_{loc}(\mathrm{d}t)$
as $\eps$ goes to zero. The third term in \eqref{remainder-c} is
estimated as follows:
\begin{equation}\label{3-remainder-c}
\begin{aligned}
&\frac{1}{\eps}\int_\Omega\<\!\< R_\eps g^{\tau,k}_{\eps,4} \>\!\>\,\mathrm{d}x\\
=&\frac{1}{\eps}\int_\Omega\<\!\< R_\eps \PP g^{\tau,k}_{\eps,4}
\>\!\>\,\mathrm{d}x + \frac{1}{\eps}\int_\Omega\<\!\< R_\eps
\PP^\perp g^{\tau,k}_{\eps,4} \>\!\>\,\mathrm{d}x\,.
\end{aligned}
\end{equation}

For the first term in the right-hand side of \eqref{3-remainder-c}, because $\PP g^{\tau,k}_{\eps,4}$ is in Null$(\LL)$,, it has the form of
\begin{equation}\label{R-null}
\frac{1}{\eps}\int_\Omega\<\!\< \Gamma'(G_\eps)q_\eps \zeta \>\!\>\,\mathrm{d}x+\frac{1}{\eps^2}\int_\Omega\<\LL\gps\zeta\>\dd x\,,\quad\mbox{for some}\quad\zeta(t,x)\in\mbox{Null}(\LL)\,.
\end{equation}
The second term above is zero, the first term converges to zero
strongly in $L^1_{loc}(\mathrm{d}t)$ as $\eps\rightarrow 0$ by the
Conservation Defect Theorem (Proposition 8.1) in
\cite{LM}.

For the second term in the right-hand side of
\eqref{3-remainder-c}, from the calculations in Proposition
\ref{main-prop} again, we have
\begin{equation}\label{3rd-term}
\begin{aligned}
\frac{1}{\eps}{\mathcal{P}}^\perp g^{\tau,k}_{\eps,4}
=&\sqrt{\tfrac{\D+2}{2\D}}\left(\tfrac{\nabla_{\!x}^2\Psi^k }{\tau
i\lambda^k}\!:\!\widehat{\mathrm{A}} +\tfrac{2\grad
\Psi^k }{\D+2}\!\cdot\!\widehat{\mathrm{B}}\right)
+ \tfrac{1}{\sqrt{\eps}}\left(\grad
\mathrm{d}\!\otimes\!\partial_{\!\zeta}\mathrm{u}^{\tau,k,\b}_0\!:\!\AHat+ \partial_{\!\zeta}\theta^{\tau,k,\b}_0\grad
\mathrm{d}\!\cdot\!\BHat\right) \\
+& \tfrac{1}{\sqrt{\eps}} g^{\tau,k,\bb}_1+ \mbox{higher order terms}\,.
\end{aligned}
\end{equation}
We decompose $R_\eps$ into
\begin{equation}\label{decompose-R}
R_\eps=T_\eps+(\tilde{g}'_{\eps 1}\tilde{g}'_{\eps }-\tilde{g}_{\eps
1}\tilde{g}_{\eps})+q_\eps\left(\frac{2}{N^2_\eps}
-\frac{1}{N_\eps}-\frac{1}{N'_{\eps 1}N'_\eps N_{\eps
1}N_\eps}\right)\,,
\end{equation}
where $T_\eps$ is
\begin{equation}\nonumber
T_\eps=\frac{q_\eps}{N'_{\eps 1}N'_\eps N_{\eps 1}N_\eps}-\frac{1}{\eps}(\tilde{g}'_{\eps 1}+\tilde{g}'_{\eps }
-\tilde{g}_{\eps 1}-\tilde{g}_{\eps})-(\tilde{g}'_{\eps 1}\tilde{g}'_{\eps }-\tilde{g}_{\eps 1}\tilde{g}_{\eps 1})\,.
\end{equation}
When we integrate \eqref{3rd-term} over $\Omega$, for the second
term of \eqref{3rd-term} which is a function of
$(\pi(x)\,,\frac{\mathrm{d}(x)}{\sqrt{\eps}})$, we make the change of
variables:
\begin{equation}\label{change-varible}
y_1=\pi^1(x)\,,\cdots \,,y_{\D-1}=\pi^{\D-1}(x)\,,
y_\D=\tfrac{\mathrm{d}(x)}{\sqrt{\eps}}\,.
\end{equation}
Then $\mathrm{d}x=\sqrt{\eps}\mathrm{T}^*\mathrm{d}y$, where $$\mathrm{T}^* =\mathrm{det}^{-1}\begin{pmatrix}\grad \pi\\ \grad \mathrm{d} \end{pmatrix} > 0\,.$$

This extra
$\sqrt{\eps}$ cancels with the $\sqrt{\eps}^{-1}$ in the second term
of \eqref{3rd-term}. Similarly, for the third term of
\eqref{3rd-term}, we make the change of variables
$y_1=\pi(x)^1\,,\cdots\,,y_{\D-1}=\pi(x)^{\D-1}\,,y_\D=\frac{\mathrm{d}(x)}{\eps}$,
and consequently $\mathrm{d}x=\eps \mathrm{T}^*\mathrm{d}y$\,. Thus, for
the integral of \eqref{3rd-term}, the first two terms are the same
order, namely $O(1)$, while the third term is of order $O(\sqrt{\eps})$, and the rest term is even higher order in $\eps$. By the Flux Remainder Theorem (Proposition 10.1) of
\cite{LM}, we have
\begin{equation}\nonumber
\frac{1}{\eps}\int_\Omega\<\!\< T_\eps {\mathcal{P}}^\perp
g^{\tau,k}_{\eps,4} \>\!\> \mathrm{d}x\rightarrow
0\,,\quad\mbox{in}\quad L^1_{loc}(\mathrm{d}t)\,,
\end{equation}
as $\eps\rightarrow 0$. Furthermore, from the Bilinear estimates (Lemma 9.1) of \cite{LM}, we have
\begin{equation}\nonumber
\frac{1}{\eps}\int_\Omega\<\!\< (\tilde{g}'_{\eps 1}\tilde{g}'_{\eps
}-\tilde{g}_{\eps 1}\tilde{g}_{\eps 1}) {\mathcal{P}}^\perp
g^{\tau,k}_{\eps,4} \>\!\> \mathrm{d}x \leq  C\int_\Omega\<
a\gps^2\>\,\mathrm{d}x\leq C\,.
\end{equation}

The third term in \eqref{decompose-R} can be written as
$q_\eps/\sqrt{N_\eps}$ times a bounded sequence that vanishes almost
everywhere as $\eps\rightarrow 0$. Then the $L^2(\mathrm{d}\nu
\mathrm{d}x\mathrm{d}t)$ boundedness of $q_\eps/\sqrt{N_\eps}$
implies that it times \eqref{3rd-term} is relatively compact in
$w\mbox{-}L^1_{loc}(\mathrm{d}t;w\mbox{-}L^1(\mathrm{d}x))$, then
following from the Product Limit Theorem of \cite{BGL2},
\begin{equation}\nonumber
\begin{aligned}
&\frac{1}{\eps}\int_\Omega\<\!\<
q_\eps\left(\tfrac{2}{N^2_\eps}-\tfrac{1}{N_\eps}-\tfrac{1}{N'_{\eps
1}N'_\eps N_{\eps 1}N_\eps}\right) {\mathcal{P}}^\perp
g^{\tau,k}_{\eps,4} \>\!\> \mathrm{d}x\\
&\rightarrow 0\,,\quad\mbox{in}\quad L^1_{loc}(\mathrm{d}t)\,,\quad\mbox{as}\quad \eps \rightarrow 0\,.
\end{aligned}
\end{equation}

Now we can decompose $c^{\tau,k}_{\eps}(t)$ into
\begin{equation}\nonumber
c^{\tau,k}_{1,\eps}(t)= - \frac{1}{\eps}\int_\Omega\<\!\<
(\tilde{g}'_{\eps 1}\tilde{g}'_{\eps }-\tilde{g}_{\eps
1}\tilde{g}_{\eps 1}) {\mathcal{P}}^\perp g^{\tau,k}_{\eps,4}
\>\!\> \mathrm{d}x - \Gamma^{\tau,k}_1\,,
\end{equation}
and $c^{\tau,k}_{2,\eps}(t)=c^{\tau,k}_{\eps}(t)-c^{\tau,k}_{1,\eps}(t)$,
where $\Gamma^{\tau,k}_1$ appears in \eqref{Boundary-V-estimate}.

The above arguments show that $c^{\tau,k}_{2,\eps}(t)\rightarrow 0$, in
$L^1_{loc}(\mathrm{d}t)$, and Lemma \ref{Boundary-Vanish} gives that $c^{\tau,k}_{1,\eps}(t)$ is bounded in
$L^p_{loc}(\mathrm{d}t)$ for some $p>1$.

\subsection{Estimates of $\tilde{b}^{\tau,k}_{\eps}$}

From \eqref{ode-b}, $\tilde{b}^{\tau,k}_{\eps}$ satisfies the ordinary
differential equation
\begin{equation}\label{ode-b-1}
\frac{\mathrm{d}}{\mathrm{d}t}\tilde{b}^{\tau,k}_{\eps}-\frac{1}{\eps}\overline{i\lambda^{\tau,k}_{\eps,4}}\tilde{b}^{\tau,k}_{\eps}
=c^{\tau,k}_{1,\eps}(t)+c^{\tau,k}_{2,\eps}(t)\,.
\end{equation}

The solution to $\eqref{ode-b-1}$ is given by
\begin{equation}\label{b-eps}
\tilde{b}^{\tau,k}_{\eps}(t)=
\tilde{b}^{\tau,k}_{\eps}(0)e^{\frac{1}{\eps}\overline{i\lambda^{\tau,k}_{\eps,4}}t}
+\int^t_0
[c^{\tau,k}_{1,\eps}(s)+c^{\tau,k}_{2,\eps}(s)]e^{-\frac{1}{\eps}
\overline{i\lambda^{\tau,k}_{\eps,4}}(s-t)}\,\mathrm{d}s\,.
\end{equation}
From the  Proposition \ref{main-prop}, $i\lambda^{\tau,k}_{\eps,4}= \tau i \lambda^k + i\lambda^{\tau,k}_1\sqrt{\eps} + i\tilde{\lambda}^{\tau,k}_1 \eps$, where $\tilde{\lambda}^{\tau,k}_1 = O(1)$.
\begin{equation}\label{12-beta}
\begin{aligned}
\frac{1}{\eps}\overline{i\lambda^{\tau,k}_{\eps,4}}t&=\frac{1}{\sqrt{\eps}}\left[\mathrm{Re}(i\lambda^{\tau,k}_{1})
+\sqrt{\eps}\mathrm{Re}(i\tilde{\lambda}^{\tau,k}_{1})\right]t\\
&-i\left[\tau \frac{1}{\eps} \lambda^k +
\frac{1}{\sqrt{\eps}}\mathrm{Im} (i\lambda^{\tau,k}_{1})+
\mathrm{Im} (i\tilde{\lambda}^{\tau, k}_{1})\right]t\,.
\end{aligned}
\end{equation}
Using \eqref{12-beta}, the first term in \eqref{b-eps} is estimated as follows:
\begin{equation}\nonumber
\begin{aligned}
&\| \tilde{b}^{\tau,k}_{\eps}(0)e^{\frac{1}{\eps}\overline{i\lambda^{\tau,k}_{\eps,4}}t}\|_{L^2(0,T)}\\
=&| \tilde{b}^{\tau,k}_{\eps}(0)|
\left[-2\left(\mathrm{Re}(i\lambda^{\tau,k}_{1})+\sqrt{\eps}\mathrm{Re}
(i\tilde{\lambda}^{\tau,k}_{1})\right)\right]^{-\frac{1}{2}}\left(1-e^{\frac{1}{\sqrt{\eps}}
[\mathrm{Re}(i\lambda^{\tau,k}_1)+\sqrt{\eps}\mathrm{Re}(i\tilde{\lambda}^{\tau,k}_{1})]T}\right)^{\frac{1}{2}}
\eps^{\frac{1}{4}}\,.
\end{aligned}
\end{equation}
To estimate $| \tilde{b}^{\tau,k}_{\eps}(0)|$, from
\begin{equation}\nonumber
 \tilde{b}^{\tau,k}_{\eps}(0)=\int\limits_\Omega\langle\gps^{\mathrm{in}}\,,g^{k,\Int}_{0}\rangle\,\mathrm{d}x+
\int\limits_\Omega\langle\gps^{\mathrm{in}}\,,g^{\tau,k}_{\eps,4}-g^{k,\Int}_{0}\rangle\,\mathrm{d}x\,,
\end{equation}
noticing that $g^{\tau,k,\Int}_{0}\in$Null$(\mathcal{L})$ and
$\|\langle \zeta(v)\gps^{\mathrm{in}}\rangle\|_{L^2(\mathrm{d}x)}$
is bounded for every $\zeta(v)\in$Null$(\mathcal{L})$, and the error
estimate for $g^{\tau,k}_{\eps,4}-g^{\tau,k,\Int}_{0}$ in
\eqref{error2}, we deduce that $| \tilde{b}^{\tau,k}_{\eps}(0)|$ is
bounded. Using the key fact that
$\mathrm{Re}(i\lambda^{\tau,k}_{1})<0$, we deduce that for any $0<T<\infty$, sufficiently small
$\eps$:
\begin{equation}\nonumber
\|
\tilde{b}^{\tau,k}_{\eps}(0)e^{-\frac{1}{\eps^2}\overline{i\lambda^{\tau,k}_{\eps,4}}t}\|_{L^2(0,T)}
\leq C\eps^{\frac{1}{4}}\,.
\end{equation}

In order to estimate the remaining term in $\eqref{b-eps}$, we
observe that for any $a\in L^p(0,t)$ and $1\leq p, r\leq\infty$,
such that $p^{-1}+r^{-1}=1$, we have
\begin{equation}\nonumber
\Big\vert\int^t_0
a(s)e^{-\frac{1}{\eps}\overline{i\lambda^{\tau,k}_{\eps,4}}(s-t)}\,\mathrm{d}s\Big\vert
\leq
 C \int^t_0 e^{-\frac{1}{\sqrt{\eps}}\mathrm{Re}(i\lambda^{\tau,k}_{1})(s-t)}\vert a(s)\vert\,\mathrm{d}s\,.
\end{equation}
Direct calculations show that
\begin{equation}\nonumber
\left\|e^{-\frac{1}{\sps}\mathrm{Re}(i\lambda^{\tau,k}_{1})(t-s)}\right\|_{L^r(0,t)}=\eps^{\frac{1}{2r}}\left[\frac{1}
{-r\mathrm{Re}(i\lambda^{\tau,k}_{1})}\left(e^{-\frac{r}{\sps}\mathrm{Re}(i\lambda^{\tau,k}_{1})t}-1\right)
\right]^{\frac{1}{r}}
e^{-\frac{1}{\sqrt{\eps}}\mathrm{Re}(i\lambda^{\tau,k}_{1})t}
\,.
\end{equation}
Using the fact $\mathrm{Re}(i\lambda^{\tau,k}_{1})<0$
again, we have
\begin{equation}\label{a-0}
\Big\vert\int^t_0
a(s)e^{-\frac{1}{\eps}\overline{i\lambda^{\tau,k}_{\eps,4}}(s-t)}\,\mathrm{d}s\Big\vert\leq
C\|a\|_{L^p(0,t)} \eps^{\frac{1}{2r}}\,.
\end{equation}

Now applying $a(t)$ in \eqref{a-0} to
$c^{\tau,k}_{1,\eps}$ and $c^{\tau,k}_{2,\eps}$, finally we get:
\begin{equation}\nonumber
\tilde{b}^{\tau,k}_{\eps}\rightarrow 0\,,\quad\mbox{strongly in}\quad
L^2_{loc}(\mathrm{d}t)\,.
\end{equation}
To finish the proof of the Proposition, we notice that
\begin{equation}\nonumber
\int\limits_\Omega\left\langle\gps\,,
g^{\tau,k,\Int}_{0}\right\rangle\,\mathrm{d}x
=\tilde{b}^{\tau,k}_{\eps} +\int\limits_\Omega\langle\gps\,,
g^{\tau,k,\Int}_{0}-
g^{\tau,k}_{\eps,4} \rangle\,\mathrm{d}x\,.
\end{equation}
Applying the error estimate \eqref{error2} in  Proposition \ref{main-prop}, we finish the proof of the Proposition \ref{perp-U}.
\end{proof}
Consequently, we prove the Proposition \ref{acoustic-vanish}.

\subsection{Proof of Lemma \ref{Boundary-Vanish} }\label{last-section}

\begin{proof}
Using the boundary condition of $g^{\tau,k}_{\eps,4}$, namely
$\eqref{dual-BC}$, simple calculations yields that
\begin{equation}\label{boundary-cal-1}
\begin{aligned}
\frac{1}{\eps}&\int\limits_\pO\langle\gamma\gps\gamma
g^{\tau,k}_{\eps,4}(v\!\cdot\!\mathrm{n})\rangle\,\mathrm{d}\sigma_x
=\frac{1}{\eps}\Signe\gamn\gps\gamn g^{\tau,k}_{\eps,4}(v\!\cdot\!\mathrm{n})M\mathrm{d}v\,\mathrm{d}\sigma_x \\
&+\frac{1}{\eps}\Sigpo\gamp\gps\left[(1-\ale)L\gamn
g^{\tau,k}_{\eps,4}+\ale\langle\gamn
g^{\tau,k}_{\eps,4}\rangle_{\pO})
+ r^{k}_{\eps,4}\right](v\!\cdot\!\mathrm{n})M\mathrm{d}v\,\mathrm{d}\sigma_x \\
=&\frac{1}{\eps}\iint\limits_{\Sigma_-}\gamma_{-}g^{\tau,k}_{\eps,4}\left[\gamma_{-}\gps-(1-\alpha_\eps)L\gamma_{+}\gps
-\alpha_\eps\langle\gamma_{+}\gps\rangle_{\pO}\right](v\!\cdot\!\mathrm{n})M\mathrm{d}v\,\mathrm{d}\sigma_x \\
&+\frac{1}{\eps}\BP \gamp\gps
r^{k}_{\eps,4}(v\!\cdot\!\mathrm{n})M\mathrm{d}v\,\mathrm{d}\sigma_x
\,,\\
=&\frac{1}{\eps}\BP(\gamp\gps-L\gamn\gps)\dnu-\frac{\ale}{\eps}\BP(\gamp\gps-\langle\gamp\gps\rangle_\pO)\dnu\\
&+\frac{1}{\eps}\BP \gamp\gps
r^{k}_{\eps,4}(v\!\cdot\!\mathrm{n})M\mathrm{d}v\,\mathrm{d}\sigma_x
\,,
\end{aligned}
\end{equation}
where the measure
$\dnu=L\gamma_{-}g^{\tau,k}_{\eps,4}(v\cdot\mathrm{n})M\mathrm{d}v\,\mathrm{d}\sigma_x
$. From the boundary error estimate \eqref{r-error} in  Proposition
\ref{main-prop} (letting $r=\infty, p=2$), and the fact that $\gamp\gps$ is bounded in
$L^1(\dd \sigma_x,L^2(aM\dd v))$, it is easy to see that
\begin{equation}\nonumber
\frac{1}{\eps}\BP \gamp\gps
r^{k}_{\eps,4}(v\!\cdot\!\mathrm{n})M\mathrm{d}v\,\mathrm{d}\sigma_x
\rightarrow 0\,,\quad\mbox{in}\quad L^1_{loc}(\mathrm{d}t)\,,
\end{equation}
as $\eps\rightarrow 0$.

It remains to show that the first two terms in the righthand side of
\eqref{boundary-cal-1} go to zero as $\eps\rightarrow 0$. It again
follows from the {\em a priori} estimates Lemma \ref{Inside} and
\ref{Bound}. The main difficulty is $\ale/\eps\rightarrow \infty$ as $\eps \rightarrow 0$, since $\alpha_\eps = \sqrt{2\pi}\chi \sqrt{\eps}$.
\begin{equation}\label{equation-1}
\begin{aligned}
&\frac{1}{\eps}\BP(\gamp\gps-L\gamn\gps)\dnu-\frac{\ale}{\eps}\BP(\gamp\gps-\langle\gamp\gps\rangle_\pO)\dnu\\
=&\frac{1}{\eps}\BP(\gamp\gps-L\gamn\gps)\dnu-\frac{\ale}{\eps}\BP\frac{\gamma_\eps}{1+\eps^2\gamp g_\eps^2}\dnu\\
&+\frac{\ale}{\eps}\BP\left(\frac{\langle\gamp g_\eps\rangle_\pO}{1+\eps^2\langle\gamp g_\eps\rangle_\pO^2}
-\frac{\langle\gamp g_\eps\rangle_\pO}{1+\eps^2\gamp g_\eps^2}\right)\dnu\\
&+\frac{\ale}{\eps}\BP\left(\left\langle\frac{\gamp g_\eps}{1+\eps^2\gamp g_\eps^2}\right\rangle_\pO
-\frac{\langle\gamp g_\eps\rangle_\pO}{1+\eps^2\langle\gamp g_\eps\rangle_\pO^2}\right)\dnu\,.
\end{aligned}
\end{equation}

The renormalized boundary condition $\eqref{renormBC}$ yields that
\begin{equation}\label{equation-2}
\begin{aligned}
&\frac{1}{\eps}(\gamp\gps-L\gamn\gps)-\frac{\ale}{\eps}\frac{\gamma_\eps}{1+\eps^2\gamp g_\eps^2}\\
=&-\frac{\ale}{\eps}\frac{\gamma_\eps\eps^2\gamp g_\eps(\gamp g_\eps+\gamp\hat{g}_\eps)}{(1+\eps^2\gamp g_\eps^2)
(1+\eps^2\gamp \hat{g}_\eps^2)}+\frac{\ale}{\eps}\left(\frac{\gamma_\eps}{1+\eps^2\gamp \hat{g}_\eps^2}
-\frac{\gamma_\eps}{1+\eps^2\gamp g_\eps^2}\right)\,.
\end{aligned}
\end{equation}
Thus, after simple calculations, we have
\begin{equation}\label{equation-3}
\begin{aligned}
&\frac{1}{\eps}\int\limits_\pO\langle\gamma\gps\gamma g^{\tau,k}_{\eps,4}(v\!\cdot\!\mathrm{n})\rangle\,\mathrm{d}\sigma_x \\
=&-\frac{\ale}{\eps}\BP\frac{\gamma_\eps\eps^2\gamp \hat{g}_\eps(\gamp g_\eps+\gamp\hat{g}_\eps)}{(1+\eps^2\gamp g_\eps^2)(1+\eps^2\gamp
\hat{g}_\eps^2)}L\gamma_{-}g^\pm_{k,\epsilon,2}(v\!\cdot\!\mathrm{n})M\mathrm{d}v\,\mathrm{d}\sigma_x \\
&+\frac{\ale}{\eps}\BP\frac{\gamma_\eps\eps^2\langle\gamp g_\eps\rangle_\pO(\gamp g_\eps
+\langle\gamp g_\eps\rangle_\pO)}{(1+\eps^2\gamp g_\eps^2)(1+\eps^2 \langle\gamp
g_\eps\rangle_\pO^2)}L\gamma_{-}g^\pm_{k,\epsilon,2}(v\!\cdot\!\mathrm{n})M\mathrm{d}v\,\mathrm{d}\sigma_x \\
&-\frac{\ale}{\eps}\BP\left\langle \frac{\gamma_\eps\eps^2\gamp g_\eps(\gamp g_\eps
+\langle\gamp g_\eps\rangle_\pO)}{(1+\eps^2\gamp g_\eps^2)(1+\eps^2 \langle\gamp g_\eps\rangle_\pO^2)}
\right\rangle_\pO L\gamma_{-}g^\pm_{k,\epsilon,2}(v\!\cdot\!\mathrm{n})M\mathrm{d}v\,\mathrm{d}\sigma_x \,.
\end{aligned}
\end{equation}

The {\em a priori} estimates from boundary yields that all the three
terms on the right-hand side of \eqref{equation-3} are bounded in in
$L^p_{loc}(\mathrm{d}t)$ for $p>1$. In deed, the integral of the
first term over $[t_1\,,t_2]$ is bounded by
\begin{equation}\label{bound-first}
\begin{aligned}
&\int^{t_2}_{t_1}\BP\sqrt{\frac{\ale}{\eps}}\frac{\gamma^{(1)}_\eps}{(1+\eps^2\gamp\gep^2)^{1/4}}
\frac{\sqrt{\ale}(\eps\gamp\hat{g}_\eps)}{1+\eps^2\gamp\hat{g}^2_\eps}\frac{(\sqrt{\eps}\gamp
\gep
+\sqrt{\eps}\gamp \hat{g}_\eps)}{(1+\eps^2\gamp\gep^2)^{3/4}}L\gamma_{-}g^\pm_{k,\epsilon,2}(v\!\cdot\!\mathrm{n})M\mathrm{d}v\,\mathrm{d}\sigma_x \\
-&\eps\int^{t_2}_{t_1}\BP \frac{\ale}{\eps^2}\gamma^{(2)}_\eps
\frac{\eps\gamp \hat{g}_\eps(\eps\gamp
g_\eps+\eps\gamp\hat{g}_\eps)}{(1+\eps^2\gamp
g_\eps^2)(1+\eps^2\gamp
\hat{g}_\eps^2)}L\gamma_{-}g^\pm_{k,\epsilon,2}(v\!\cdot\!\mathrm{n})M\mathrm{d}v\,\mathrm{d}\sigma_x
\end{aligned}
\end{equation}

Note that
\begin{equation}\nonumber
\frac{\sqrt{\ale}(\eps\gamp\hat{g}_\eps)}{1+\eps^2\gamp\hat{g}^2_\eps}\quad
\mbox{and}\!\quad \frac{(\sqrt{\eps}\gamp \gep +\sqrt{\eps}\gamp
\hat{g}_\eps)}{(1+\eps^2\gamp\gep^2)^{3/4}}\mathbf{1}_{\gamp
G_\eps\leq 2\langle G_\eps\rangle_{\pO}\leq 4\gamp G_\eps}
\end{equation}
are bounded. Furthermore,$
L\gamma_{-}g^\pm_{k,\epsilon,2}(v\!\cdot\!\mathrm{n})$ is bounded in
$L^q((v\!\cdot\!\mathrm{n})M\mathrm{d}v\,\mathrm{d}\sigma_x)$ for
$q\geq 2$.  Now, the estimates \eqref{bound-1}, \eqref{bound-3} of
Lemma \ref{Bound}  imply that the first term in \eqref{bound-first}
is bounded in $L^p_{loc}(\mathrm{d}t)$, the second term vanishes as
$\eps\rightarrow 0$. Similarly, using Lemma \ref{Inside},  Lemma
\ref{Bound},  we can prove that the integrals over $[t_1\,,t_2]$ of
the second and third terms of \eqref{equation-3} can be decomposed
into two terms, one is bounded in $L^p_{loc}(\mathrm{d}t)$, the
other vanishes in $L^1_{loc}(\mathrm{d}t)$. Thus we proved the Lemma
\ref{Boundary-Vanish}.

\end{proof}

\end{document}